\documentclass{article}
\usepackage[T1]{fontenc}
\usepackage[latin9]{inputenc}
\usepackage{geometry}
\usepackage{amsmath,amsbsy,amstext,amsthm,amssymb}
\usepackage{enumerate}
\usepackage{color}

\makeatletter
\@addtoreset{section}{part}
\makeatother

\theoremstyle{plain}
\newtheorem{thm}{\protect\theoremname}[section]
  \theoremstyle{plain}
  \newtheorem{conjecture}[thm]{\protect\conjecturename}
  \theoremstyle{definition}
  
  \theoremstyle{definition}
  \newtheorem{defn}[thm]{\protect\definitionname}
  \theoremstyle{plain}
  \newtheorem{prop}[thm]{\protect\propositionname}
  \theoremstyle{plain}
  \newtheorem{cor}[thm]{\protect\corollaryname}
  \theoremstyle{definition}
  \newtheorem{example}[thm]{\protect\examplename}
  \theoremstyle{plain}
  \newtheorem{lem}[thm]{\protect\lemmaname}
  \theoremstyle{definition}
  \newtheorem{rem}[thm]{Remark}
  \theoremstyle{definition}
  \newtheorem{setup}[thm]{Setup}

\newcommand{\cc}[1]{\sigma (#1)}

\DeclareMathOperator{\Bin}{Bin}
\DeclareMathOperator{\im}{Im}

\newcommand{\OO}{\Omega}

\usepackage[english]{babel}
  \providecommand{\conjecturename}{Conjecture}
  \providecommand{\corollaryname}{Corollary}
  \providecommand{\definitionname}{Definition}
  \providecommand{\examplename}{Example}
  \providecommand{\lemmaname}{Lemma}
  \providecommand{\problemname}{Problem}
  \providecommand{\propositionname}{Proposition}
\providecommand{\theoremname}{Theorem}
\providecommand{\claimname}{Claim}

\usepackage{thmtools}

\begin{document}

\global\long\def\connected{\text{highly connected}}
\global\long\def\tra{\mathsf{T}_{q\to p}}
\global\long\def\art{\mathsf{T}^{p\to q}}
 \global\long\def\f{\mathcal{F}}
\global\long\def\pn{\mathcal{P}\left[n\right]}
\global\long\def\g{\mathcal{G}}
\global\long\def\s{\mathcal{S}}
\global\long\def\j{\mathcal{J}}
\global\long\def\d{\mathcal{D}}
\global\long\def\Inf{}
\global\long\def\p{\mathcal{P}}
\global\long\def\h{\mathcal{H}}
\global\long\def\n{\mathbb{N}}
\global\long\def\a{\mathcal{A}}
\global\long\def\b{\mathcal{B}}
\global\long\def\c{\mathcal{C}}
\global\long\def\e{\mathbb{E}}
\global\long\def\x{\mathbf{x}}
\global\long\def\y{\mathbf{y}}
\global\long\def\z{\mathbf{z}}
\global\long\def\c{\mathbf{c}}
\global\long\def\av{\mathsf{A}}
\global\long\def\chop{\mathrm{Chop}}
\global\long\def\stab{\mathrm{Stab}}
\global\long\def\t{\mathsf{T}}
\global\long\def\fabs#1{\hat{|}#1\hat{|}}

\newcommand\comment[1]{\fbox{\tt #1}}
\newcommand\E{\mathop{\mathbb{E}}}
\newcommand\card[1]{\left| {#1} \right|}
\newcommand\sett[2]{\left\{ \left. #1 \;\right\vert #2 \right\}}
\newcommand\settright[2]{\left\{ #1 \;\left\vert\; #2 \right.\right\}}
\newcommand\set[1]{{\left\{ #1 \right\}}}
\newcommand\Prob[2]{{\Pr_{#1}\left[ {#2} \right]}}
\newcommand\BProb[2]{{\Pr_{#1}\Big[ {#2} \Big]}}
\newcommand\bProb[2]{{\Pr_{#1}\big[ {#2} \big]}}
\newcommand\cProb[3]{{\Pr_{#1}\left[ \left. #3 \;\right\vert #2 \right]}}
\newcommand\cbProb[3]{{\Pr_{#1}\big[ \big. #3 \;\big\vert #2 \big]}}
\newcommand\cBProb[3]{{\Pr_{#1}\Big[ \Big. #3 \;\Big\vert #2 \Big]}}
\newcommand\Expect[2]{{\mathop{\mathbb{E}}_{#1}\left[ {#2} \right]}}
\newcommand\bExpect[2]{{\mathop{\mathbb{E}}_{#1}\big[ {#2} \big]}}
\newcommand\BExpect[2]{{\mathop{\mathbb{E}}_{#1}\Big[ {#2} \Big]}}
\newcommand\cProbright[3]{{\Pr_{#1}\left[ #3\;\left\vert  #2\right. \right]}}
\newcommand\cExpect[3]{{\mathbb{E}_{#1}\left[ \left. #3 \;\right\vert #2 \right]}}
\newcommand\norm[1]{\| #1 \|}
\newcommand\power[1]{\set{0,1}^{#1}}
\newcommand\ceil[1]{\lceil{#1}\rceil}
\newcommand\round[1]{\left\lfloor{#1}\right\rfloor}
\newcommand\half{{1\over2}}
\newcommand\quarter{{1\over4}}
\newcommand\eigth{{1\over8}}
\newcommand\defeq{\stackrel{def}{=}}
\newcommand\skipi{{\vskip 10pt}}
\newcommand\sqrtn{ {\sqrt n} }
\newcommand\cube[1]{ { \set{-1,1}^{#1} } }
\newcommand\minimum[2] {  \min\left\{{{#1}, {#2}}\right\}   }
\newcommand\spn{{\sf Span}}
\newcommand\spnn[1]{{\sf Span}({#1})}
\newcommand\inner[2]{\langle{#1},{#2}\rangle}
\newcommand\eps{\varepsilon}
\newcommand\sgn[1]{sign({#1})}
\newcommand*\xor{\mathbin{\oplus}}
\newcommand*\bxor{\mathbin{\bigoplus}}
\newcommand*\Hell[2]{{H\left({#1}, \ {#2}\right)}}
\newcommand*\cHell[4]{{H\left({#1}\vert({#2}),\ {#3}\vert({#4}) \right)}}
\newcommand*\dimm{{\sf dim}}
\newcommand*\tp{t}
\newcommand*\R{\mathbb{R}}
\newcommand*\mult[1]{\vec{#1}}
\newcommand*\zin[2]{\mu_{{\sf in}({#1})}({#2})}
\newcommand*\zout[2]{\mu_{{\sf out}({#1})}({#2})}
\newcommand*\zinout[3]{\mu_{{\sf in}({#1}), {\sf out}({#2})}({#3})}
\newcommand\lvlfn[2]{{#1}_{={#2}}}
\newcommand\approxlvlfn[2]{{#1}_{\approx{#2}}}
\newcommand{\cupdot}{\mathbin{\mathaccent\cdot\cup}}
\newcommand{\M}{\mathcal{M}}
\newcommand{\indeg[1]}{{\sf in}\text{-}{\sf deg}({#1})}
\newcommand{\outdeg[1]}{{\sf out}\text{-}{\sf deg}({#1})}
\newcommand{\NP}{\sf NP}

\newcommand{\mc}[1]{\mathcal{#1}}
\newcommand{\mb}[1]{\mathbb{#1}}
\newcommand{\mf}[1]{\mathfrak{#1}}
\newcommand{\nib}[1]{\noindent {\bf #1}}
\newcommand{\nim}[1]{\noindent {\em #1}}
\newcommand{\brac}[1]{\left( #1 \right)}
\newcommand{\bracc}[1]{\left\{ #1 \right\}}
\newcommand{\bsize}[1]{\left| #1 \right|}
\newcommand{\brak}[1]{\left[ #1 \right]}
\newcommand{\bgen}[1]{\left\langle #1 \right\rangle}
\newcommand{\sgen}[1]{\langle #1 \rangle}
\newcommand{\bfl}[1]{\left\lfloor #1 \right\rfloor}
\newcommand{\bcl}[1]{\left\lceil #1 \right\rceil}

\newcommand{\sub}{\subset}
\newcommand{\subn}{\subsetneq}
\newcommand{\sups}{\supseteq}
\newcommand{\ex}{\mbox{ex}}
\newcommand{\nsub}{\notsubseteq}

\newcommand{\lra}{\leftrightarrow}
\newcommand{\Lra}{\Leftrightarrow}
\newcommand{\Ra}{\Rightarrow}
\newcommand{\sm}{\setminus}
\newcommand{\ov}{\overline}
\newcommand{\ul}{\underline}
\newcommand{\wt}{\widetilde}
\newcommand{\es}{\emptyset}
\newcommand{\pl}{\partial}
\newcommand{\wg}{\wedge}

\newcommand{\la}{\leftarrow}
\newcommand{\ra}{\rightarrow}
\newcommand{\da}{\downarrow}
\newcommand{\ua}{\uparrow}

\newcommand{\aA}{\alpha}
\newcommand{\bB}{\beta}
\newcommand{\gG}{\gamma}
\newcommand{\dD}{\delta}
\newcommand{\iI}{\iota}
\newcommand{\kK}{\kappa}
\newcommand{\zZ}{\zeta}
\newcommand{\lL}{\lambda}
\newcommand{\tT}{\theta}
\newcommand{\sS}{\sigma}
\newcommand{\oO}{\omega}
\newcommand{\ups}{\upsilon}

\newcommand{\Tt}{\Theta}
\newcommand{\DD}{\Delta}
\newcommand{\Oo}{\Omega}

\newcommand{\xb}{\boldsymbol{x}}
\newcommand{\yb}{\boldsymbol{y}}
\newcommand{\zb}{\boldsymbol{z}}

\title{Global hypercontractivity and its applications}
\date{\vspace{-5ex}}

\author{Peter Keevash\thanks{Mathematical Institute,
University of Oxford, Oxford, UK. Email: keevash@maths.ox.ac.uk.
\newline \hspace*{1.8em}Research supported
in part by ERC Consolidator Grant 647678.}
\and Noam Lifshitz\thanks{Einstein Institute for Mathematics, Hebrew University of Jerusalem. Email: noamlifshitz@gmail.com. 
\newline \hspace*{1.8em}Research supported in part by ERC advanced grant 834735.}
\and Eoin Long\thanks{University of Birmingham, Birmingham, UK. Email: e.long@bham.ac.uk.}
\and Dor Minzer\thanks{Department of Mathematics, Massachusetts Institute of Technology. Email: minzer.dor@gmail.com.
\newline \hspace*{1.8em}Some of the work was done while the author was a postdoc in the Institute for Advanced Study, Princeton, 
supported NSF grant CCF-1412958 and Rothschild Fellowship.}
}

\maketitle

\begin{abstract}
The classical hypercontractive inequality for the noise operator
on the discrete cube plays a crucial role in many of the
fundamental results in the Analysis of Boolean functions,
such as the KKL (Kahn-Kalai-Linial) theorem,
Friedgut's junta theorem and the invariance principle
of Mossel, O'Donnell and Oleszkiewicz.
In these results the cube is equipped
with the uniform ($1/2$-biased) measure,
but it is desirable, particularly for applications
to the theory of sharp thresholds, to also obtain
such results for general $p$-biased measures.
However, simple examples show that when $p$ is small
there is no hypercontractive inequality
that is strong enough for such applications.

In this paper, we establish an effective hypercontractive
inequality for general $p$ that applies to `global functions',
i.e.\ functions that are not significantly
affected by a restriction of a small set of coordinates.
This class of functions appears naturally,
e.g.\ in Bourgain's sharp threshold theorem,
which states that such functions exhibit a sharp threshold.
We demonstrate the power of our tool by
strengthening Bourgain's theorem,
thereby making progress on a conjecture of Kahn and Kalai
and by establishing a $p$-biased analog of the seminal invariance principle
of Mossel, O'Donnell, and Oleszkiewicz.

Our sharp threshold results also have significant applications in
Extremal Combinatorics. Here we obtain new results on the Tur\'an number
of any bounded degree uniform hypergraph obtained
as the expansion of a hypergraph of bounded uniformity. These
are asymptotically sharp over an essentially optimal
regime for both the uniformity and the number of edges and solve
a number of open problems in the area.
In particular, we give general conditions under which
the crosscut parameter asymptotically determines the Tur\'an number,
answering a question of Mubayi and Verstra\"ete.
We also apply the Junta Method to refine our asymptotic results
and obtain several exact results, including proofs of
the Huang--Loh--Sudakov conjecture on cross matchings
and the F\"uredi--Jiang--Seiver conjecture on path expansions.
\end{abstract}

\newpage

\section*{Introduction}

The field of Analysis of Boolean functions is centered around the
study of functions on the discrete cube $\{0,1\}^{n}$,
via their Fourier--Walsh expansion,
often using the classical hypercontractive inequality
for the noise operator, obtained independently
by Bonami \cite{bonami1970etude},
Gross \cite{gross1975logarithmic}
and Beckner \cite{beckner1975inequalities}.
In particular, the fundamental `KKL' theorem
of Kahn, Kalai and Linial \cite{kahn1988influence} applies hypercontractivity to obtain
structural information on Boolean valued functions with
small `total influence' / `edge boundary'
(see Section \ref{subsec:isoinf});
such functions cannot be `global':
they must have a co-ordinate with large influence.

The theory of sharp thresholds is closely connected
(see Section \ref{sec:kahn kalai}) to the structure
of Boolean functions of small total influence,
not only in the KKL setting of uniform measure on the cube,
but also in the general $p$-biased setting.
However, we will see below that the hypercontractivity theorem
is ineffective for small $p$.
This led Friedgut \cite{friedgut1999sharp},
Bourgain \cite[appendix]{friedgut1999sharp},
and Hatami \cite{hatami2012structure} to develop new ideas for proving
$p$-biased analogs of the KKL theorem.
The theme of these works
can be roughly summarised by the statement:\
an effective analog of the KKL theorem holds
for a certain class of `global' functions.
However, these theorems were incomplete
in two important respects:
\begin{itemize}
\item{\emph{Sharpness:}}
Unlike the KKL theorem,
they are not sharp up to constant factors.
\item{\emph{Applicability:}}
They are only effective in the
`dense setting' when $\mu_p(f)$
is bounded away from $0$ and $1$,
whereas the `sparse setting' $\mu_p(f)=o(1)$
is needed for many important open problems.
\end{itemize}
A sparse analogue of the KKL theorem was a key missing
ingredient in a strategy suggested by Kahn and Kalai \cite{kahn2007thresholds}
for their well-known conjecture relating
critical probabilities to expectation thresholds.

\subsection*{Main contribution}

The most fundamental new result of this paper
is a hypercontractive theorem for functions that are `global'
(in a sense made precise below). This has many applications,
of which the most significant are as follows.
\begin{itemize}
\item We strengthen Bourgain's Theorem
by obtaining an analogue of the KKL theorem
that is both quantitively tight and
applicable in the sparse regime.
\item We obtain a sharp threshold result
for global monotone functions in the spirit
of the Kahn-Kalai conjecture, bounding the ratio between
the critical probability (where $\mu_p(f)=\frac{1}{2}$)
and the smallest $p$ for which $\mu_p(f)$ is non-negligible.
\item We obtain a $p$-biased generalisation of
the seminal invariance principle
of Mossel, O'Donnell and Oleszkiewicz \cite{mossel2010noise}
(itself a generalisation of the Berry-Esseen theorem
from linear functions to polynomials of bounded degree),
thus opening the door to $p$-biased versions of its
many striking applications in Hardness of Approximation
and Social Choice Theory
(see O'Donnell \cite[Section 11.5]{o2014analysis})
and Extremal Combinatorics (see Dinur--Friedgut--Regev
\cite{dinur2008independent}).
\item We obtain strong new estimates on a wide class of
hypergraph Tur\'an numbers, which are central and challenging
parameters in Extremal Combinatorics.
Our results apply to bounded degree uniform hypergraphs which are obtained
as an expansion of a hypergraph of bounded uniformity,
and allow us to solve a number of open problems in the area.
\end{itemize}

\subsection*{Structure of the paper}

To facilitate navigation between various topics
we have divided this paper into five parts.
The first part is an extended synopsis
in which we motivate and state our main results.
The second part introduces global hypercontractivity,
our fundamental new contribution that underpins the
applications in the subsequent three parts of the paper.
After presenting the basic form of our hypercontractivity inequality
needed for the later applications, we continue to develop the general theory,
as this has independent interest, and also has further applications
via our generalised invariance principle.
The third part contains our analytic applications
to results on sharp thresholds and isoperimetric stability.
We then move to combinatorial applications in the final two parts,
where for the benefit of any reader whose primary interest lies in these applications,
we would highlight that these draw upon the earlier parts in a `black-box' manner,
which are therefore not pre-requisite reading for the final parts.
The fourth part concerns pseudorandomness notions for set systems
and their application to approximation by juntas,
which is the basis of the Junta Method in Extremal Combinatorics.
We then apply these results in the fifth part
to obtain several exact results on hypergraph Tur\'an numbers.

\subsection*{Subsequent work} \label{subsect: Further applications}

Since the first appearance of this paper,
our global hypercontractivity inequality
has found several further applications.
    Noise sensitivity of sparse sets is related to small-set expansion
    on graphs, which has found many applications in Computer Science.
    Here the interpretation of Theorem \ref{thm:Noise sensitivity}
    is that although not all small sets in the $p$-biased cube expand,
    global small sets do expand.
    Results of a similar nature were proved for the Grassman graph
    (see \cite{subhash2018pseudorandom}) and the Johnson graph
    (see \cite{khot2018johnson}). The former result
    was essential in the proof of the $2$-to-$2$ Games
    Conjecture, a prominent problem in the field of hardness of
    approximation. Both these works involve long calculations, and have
    sub-optimal parameters. In subsequent works
    \cite{ellis2019hypercontractivity,filmus2019hypercontractivity_slice,
      filmus2019hypercontractivity_symmetric, keevash2019random}
     hypercontractive  results for global functions
     are proven for various domains by reducing to
    the $p$-biased cube and using
    Theorem \ref{thm:Hypercontractivity}. The results of
    \cite{ellis2019hypercontractivity,filmus2019hypercontractivity_slice}
    imply the
    corresponding results about small expanding sets in the
    Grassman/Johnson graph with optimal
    parameters. A similar result was also established for a certain
    noise operator on the symmetric group
    \cite{filmus2019hypercontractivity_symmetric}.

\newpage

\part{Results} \label{part:results}

This part is an extended synopsis of our paper,
in which we motivate and state our results.
Section \ref{sec:hyp-intro} concerns our
results on global hypercontractivity.
We consider its applications to Analysis
in the two subsequent sections.
In Section \ref{sec:kahn kalai}
we discuss sharp thresholds and the Kahn--Kalai Conjecture.
Section \ref{sec:noise-intro} concerns noise sensitivity,
which gives an alternative approach to sharp thresholds
and is of interest in its own right.
We conclude the part by discussing our applications
to Extremal Combinatorics in Section \ref{sec:turan-intro}.

\section{Hypercontractivity of global functions} \label{sec:hyp-intro}

Before formally stating our main theorem,
we start by recalling (the $p$-biased version of)
the classical hypercontractive inequality.
Let $p \in \left(0,\frac{1}{2}\right]$
(the case $p>\frac{1}{2}$ is similar).
For $r \ge 1$ we write $\|\cdot\|_r$
(suppressing $p$ from our notation)
for the norm on $L^r(\{0,1\}^n,\mu_p)$.

\begin{defn}[Noise operator]
  For  $x \in \{0,1\}^n$
  we define the $\rho$-correlated distribution
$N_{\rho}(x)$ on $\left\{ 0,1\right\} ^{n}$: a sample
$\mathbf{y}\sim N_\rho(x)$ is obtained by, independently for each $i$
setting $\mathbf{y}_i=x_i$  with probability $\rho$,
or otherwise (with probability $1-\rho$) we resample $\mathbf{y}_i$
with $\mb{P}(\mathbf{y}_i=1)=p$. We define
the noise operator $\mathrm{T}_\rho$ on $L^2(\{0,1\}^n,\mu_p)$ by
\[
\mathrm{T}_{\rho}\left(f\right)\left(x\right)=\mathbb{E}_{\yb
\sim N_{\rho}\left(x\right)}\left[f\left(\yb\right)\right].
  \]
\end{defn}

H\"older's inequality gives $\|f\|_r\le\|f\|_s$ whenever $r\le s$.
The hypercontractivity theorem gives an inequality
in the other direction after applying noise to $f$;
for example, for $p=1/2$, $r=2$ and $s=4$ we have
\[\|\mathrm{T}_{\rho}f\|_4 \le \|f\|_2\]
for any $\rho\le\frac{1}{\sqrt{3}}$.
A similar inequality also holds when $p=o(1)$,
but the correlation $\rho$ has to be so small
that it is not useful in applications;
e.g.\  if $f(x)=x_1$
(the `dictator' or `half cube'),
then $\|f\|_2 = \sqrt{\mu_p(f)} = \sqrt{p}$ and
$\mathrm{T}_{\rho}f(x)
= \mathbb{E}_{\yb \sim N_{\rho}\left(x\right)} \mathbf{y}_1
= \rho x_1 + (1-\rho)p$,
so $\|\mathrm{T}_{\rho}f\|_4
> (\mb{E}[\rho^4 x_1^4])^{1/4} = \rho p^{1/4}$.
Thus we need $\rho = O(p^{1/4})$ to obtain any
hypercontractive inequality for general $f$.

\subsection{Local and global functions}
To resolve this issue, we note that the tight examples
for the hypercontractive inequality are \emph{local},
in the sense that a small number of
coordinates can significantly influence the output of the function.
On the other hand, many functions of interest are \emph{global},
in the sense that a small
number of coordinates can change the output of the function only with a
negligible probability; such global functions appear naturally
in Random Graph Theory \cite{achlioptas1999sharp},
Theoretical Computer Science \cite{friedgut1999sharp}
and Number Theory \cite{friedgut2016sharp}.
Our hypercontractive inequality will show that
constant noise suffices for functions that are global
in a sense captured by \emph{generalised influences},
which we will now define.

Let $f\colon\left\{ 0,1\right\} ^{n}\to\mathbb{R}$.
For $S\sub[n]$ and $x\in\{0,1\}^S$, we write $f_{S\to x}$
for the function obtained from $f$ by restricting the
coordinates of $S$ according to $x$
(if $S=\{i\}$ is a singleton we simplify notation to $f_{i\to x}$).
We write $|x|$ for the number of ones in $x$.
For $i \in [n]$, the $i$th influence is
$\mathrm{I}_i(f) = \| f_{i \to 1} - f_{i \to 0} \|_2^2$,
where the norm is with respect to the implicit measure $\mu_p$. In general,
we define the influence with respect to any $S \sub [n]$
by sequentially applying the operators $f\mapsto f_{i\to 1}-f_{i \to
  0}$ for all $i \in S$, as follows.
\begin{defn}
\label{def:generalised influences}
For $f\colon\left\{ 0,1\right\} ^{n}\to\mathbb{R}$ and $S \sub [n]$
we let (suppressing $p$ in the notation)
\[
\mathrm{I}_{S}\left(f\right)=\mathbb{E}_{\mu_{p}}\bigg [\Big (\sum_{x\in\left\{ 0,1\right\} ^{S}}\left(-1\right)^{\left|S\right|-\left|x\right|}f_{S\to x}\Big )^{2}\bigg ].
\]
We say $f$ has \emph{$\bB$-small generalised influences} if
$\mathrm{I}_S(f)\le\bB\ \mb{E}[f^2]$ for all $S\subseteq\left[n\right].$
\end{defn}

The reader familiar with the KKL theorem and the invariance principle
may wonder why it is necessary to introduce generalised influences
rather than only considering influences (of singletons).
The reason is that under the uniform measure
the properties of having small influences
or small generalised influences are qualitatively equivalent,
but this is no longer true in the $p$-biased setting
for small $p$ (consider $f(x) =
\frac{x_1x_2 +\cdots +x_{n-1}x_n}{\|x_1x_2 +\cdots +x_{n-1}x_n\|}$).

We are now ready to state our main theorem,
which shows that global\footnote{Strictly speaking,
our assumption is stronger than the most natural notion
of global functions: we require all generalised influences
to be small, whereas a function should be considered global
if it has small generalised influences
$I_S(f)$ for small sets $S$.
However, in practice, the hypercontractivity Theorem
is typically applied to low-degree truncations of Boolean functions
(see Section \ref{sec:practice}),
when there is no difference between these notions,
as $I_S(f)=0$ for large $S$.} functions are
hypercontractive for a noise operator with a constant rate.
Moreover, our result applies to general $L^r$ norms and product spaces
(see Section \ref{sec:hyp}), but for simplicity here we just highlight
the case of $(4,2)$-hypercontractivity in the cube.

\begin{thm} \label{thm:Hypercontractivity}
Let $p \in \left(0,\frac{1}{2}\right]$.
Suppose $f\in L^{2}\left(\left\{ 0,1\right\} ^{n},\mu_{p}\right)$
has $\bB$-small generalised influences (for $p$).
Then $\|\mathrm{T}_{1/5} f\|_{4} \le \bB^{1/4} \|f\|_2$.
\end{thm}

We now move on to demonstrate the power
of global hypercontractivity in the contexts
of isoperimetry, noise sensitivity, sharp thresholds, and invariance.
We emphasise that Theorem \ref{thm:Hypercontractivity} is the only new
ingredient required for these applications, so we expect that it will
have many further applications to generalising results proved via
usual hypercontractivity on the cube with uniform measure.

\subsection{Isoperimetry and influence} \label{subsec:isoinf}

Stability of isoperimetric problems is a prominent open problem at the interface of Geometry, Analysis and Combinatorics.
This meta-problem is to characterise sets whose boundary is close to
the minimum possible given their volume; there are many specific problems
obtained by giving this a precise meaning. Such results in Geometry
were obtained for the classical setting of
Euclidean Space by Fusco, Maggi and Pratelli \cite{fusco2008iso}
and for Gaussian Space by Mossel and Neeman \cite{mossel2015robust}.

The relevant setting for our paper is that of the cube $\{0,1\}^n$,
endowed with the $p$-biased measure $\mu_p$. We refer to this problem
as the ($p$-biased) edge-isoperimetric stability problem.
We identify any subset of $\{0,1\}^n$ with its characteristic
Boolean function $f:\{0,1\}^n \to \{0,1\}$,
and define its `boundary' as the (total) influence\footnote{
Everything depends on $p$, which we fix and suppress in our notation.}
\[\mathrm{I}\left[f\right] =\sum_{i=1}^{n}\mathrm{I}_{i}\left[f\right],
\text{ where each }
\mathrm{I}_i\left[f\right]=\Pr_{\xb\sim\mu_{p}}
\left[f\left(\xb\oplus e_{i}\right)\ne f\left(\xb\right)\right],\]
i.e.\ the \emph{$i$th influence} $\mathrm{I}_i\left[f\right]$
of $f$ is the probability that $f$ depends on bit $i$
at a random input according to $\mu_p$.
(The notion of influence for real-valued functions,
given in Section \ref{sec:hyp-intro}, coincides with this notion
for Boolean-valued functions).
When $p=1/2$ the total influence corresponds to the classical
combinatorial notion of edge-boundary\footnote{
For the vertex boundary, stability results
showing that approximately isoperimetric sets
are close to Hamming balls were obtained independently by
Keevash and Long \cite{keevash2018vertexiso}
and by Przykucki and Roberts \cite{prz2018vertexiso}.}.

The KKL theorem of Kahn, Kalai and Linial \cite{kahn1988influence}
concerns the structure of functions $f:\{0,1\}^n \to \{0,1\}$,
considering the cube under the uniform measure,
with variance bounded away from $0$ and $1$
and with total influence is upper bounded by some number $K$.
It states that $f$ has a coordinate with influence
at least $e^{-O\left(K\right)}$.
The tribes example of Ben-Or and Linial \cite{ben1990collective}
shows that this is sharp.

\subsection{$p$-biased versions}
The $p$-biased edge-isoperimetric stability problem is
somewhat understood in the \emph{dense regime}
(where $\mu_p\left(f\right)$ is bounded away from $0$ and $1$)
especially for Boolean functions $f$ that are \emph{monotone}
(satisfy $f\left(x\right)\le f\left(y\right)$
whenever all $x_{i}\le y_{i}$).
Roughly speaking, most edge-isoperimetric stability results in the dense regime say that Boolean functions of small influence have some
`local' behaviour (see the seminal works
of Friedgut--Kalai \cite{friedgut1996every},
Friedgut \cite{friedgut1998boolean, friedgut1999sharp},
Bourgain \cite[Appendix]{friedgut1999sharp},
and Hatami \cite{hatami2012structure}).
In particular, Bourgain
(see also \cite[Chapter 10]{o2014analysis})
showed that for any monotone Boolean function $f$
with $\mu_{p}\left(f\right)$ bounded away from $0$ and $1$ and
$pI\left[f\right]\le K$
there is a set $J$ of $O\left(K\right)$ coordinates
such that $\mu_{p}\left(f_{J\to1}\right)
\ge\mu_{p}\left(f\right)+e^{-O\left(K^{2}\right)}$.
This result is often interpreted as `almost isoperimetric (dense) subsets of the
$p$-biased cube must be local' or on the contrapositive as `global
functions have large total influence'. Indeed, if a
restriction of a small set of coordinates can significantly boost the $p$-biased measure of a function, then this intuitively means
that it is of a local nature.

For monotone functions, the conclusion in Bourgain's theorem
is equivalent (see Section \ref{sec:equiv})
to having some set $J$ of size $O(K)$ with
$\mathrm{I}_J\left(f\right) \ge e^{-O\left(K^2\right)}$.
Thus Bourgain's theorem can be viewed as a
$p$-biased analog of the KKL theorem, where influences are replaced
by generalised influences. However, unlike the KKL Theorem,
Bourgain's result is not sharp,
and the anti-tribes example of Ben-Or and Linial only shows that
the $K^2$ term in the exponent cannot drop below $K$.

As a first application of our hypercontractivity theorem we replace the
term $e^{-O(K^2)}$ by the term $e^{-O(K)}$, which is sharp by Ben-Or
and Linial's example, see Section \ref{sec:equiv}.

\begin{thm}
\label{thm:Bourgain+}
Let $p\in\left(0,\frac{1}{2}\right]$, and let  $f\colon\left\{ 0,1\right\} ^{n}\to\left\{ 0,1\right\} $
be a monotone Boolean function
with $\mu_{p}\left(f\right)$ bounded away from $0$ and $1$
and $\mathrm{I}\left[f\right]\le\frac{K}{p}$.
Then there is a set $J$ of $O\left(K\right)$ coordinates such that
$\mu_{p}\left(f_{J\to1}\right)\ge\mu_{p}\left(f\right)+e^{-O\left(K\right)}$.
\end{thm}

For general functions we prove a similar result, where the conclusion
$\mu_{p}\left(f_{J\to1}\right)\ge\mu_{p}\left(f\right)+e^{-O\left(K\right)}$
is replaced with $\mathrm{I}_J\left(f\right)\ge e^{-O\left(K\right)}$.

\subsection{The sparse regime}

On the other hand, the \emph{sparse regime}
(where we allow any value of $\mu_p(f)$)
seemed out of reach of previous methods in the literature.
Here Russo \cite{russo1982approximate}, and independently Kahn and Kalai \cite{kahn2007thresholds}, gave a proof of the
$p$-biased isoperimetric inequality:
$p\mathrm{I}\left[f\right]\ge
\mu_p\left(f\right)\log_p\left(\mu_p\left(f\right)\right)$ for every
$f$. They also showed that equality holds only for the monotone
sub-cubes. Kahn and Kalai posed the problem of determining
the structure of monotone Boolean functions $f$ that they called
\emph{$d$-optimal}, meaning that $p\mathrm{I}\left[f\right]\le
d\mu_p\left(f\right)\log_p\left(\mu_p\left(f\right)\right)$,
i.e.\ functions with total influence within a certain
multiplicative factor of the minimal value guaranteed
by the isoperimetric inequality.
They conjectured in \cite[Conjecture 4.1(a)]{kahn2007thresholds}
that for any constant $C>0$
there are constants $K,\dD>0$ such that
if $f$ is $C\log\left(1/p\right)$-optimal
then there is a set $J$ of
$\le K \log \frac{1}{\mu_{p}\left(f\right)}$ coordinates such that
$\mu_{p}\left(f_{J\to1}\right)\ge (1+\dD)\mu_p(f)$.

The corresponding result with a similar conclusion was open
even for $C$-optimal functions! Our second theorem is a variant
of the Kahn--Kalai conjecture which applies to
$C\log\left(1/p\right)$-optimal functions
when $C$ is sufficiently small
(whereas the conjecture requires
an arbitrary constant $C$).
We compensate for our stronger hypothesis in the following result
by obtaining
a much stronger conclusion than that asked for by Kahn and Kalai;
for example, if $f$ is $\frac{\log\left(1/p\right)}{100C}$-optimal
then $\mu_{p}\left(f_{J\to1}\right)\ge \mu_p(f)^{0.01}$.
We will also show that our result
is sharp up to the constant factor $C$.

\begin{thm}
\label{thm:Variant of Kahn kalai}
Let $p\in\left(0,\frac{1}{2}\right]$, $K \ge 1$
and let $f$ be a Boolean function with
$p\mathrm{I}\left[f\right] < K\mu_{p}\left(f\right)$.
Then there is a set $J$ of $\le CK$ coordinates,
where $C$ is an absolute constant, such that
$\mu_{p}\left(f_{J\to1}\right)\ge e^{-CK}$.
\end{thm}

\section{Sharp thresholds} \label{sec:kahn kalai}

The results of Friedgut and Bourgain mentioned above
also had the striking consequence that any `global'
Boolean function has a sharp threshold, which was
a breakthrough in the understanding of this phenomenon,
as it superceded many results for specific functions.

The sharp threshold phenomenon concerns the behaviour
of $\mu_p(f_n)$ for $p$ around the critical probability,
defined as follows.
Consider any sequence $f_n\colon\left\{0,1\right\}^n \to \{0,1\}$
of monotone Boolean functions. For $t \in [0,1]$
let $p_n(t) = \inf \{ p: \mu_p(f_n) \ge t\}$.
In particular, $p^c_n := p_n(1/2)$ is
commonly known as the `critical probability'
(which we think of as small in this paper).
A classical theorem of
Bollob\'{a}s and Thomason \cite{bollobas1987threshold}
shows that for any $\eps>0$ there is $C>0$
such that $p_n(1-\eps) \le C p_n(\eps)$.
This motivates the following definition:
we say that the sequence $(f_n)$
has a \emph{coarse threshold}
if for each $\eps>0$ the length of the interval
$[p_n(\eps),p_n(1-\eps)]$ is $\Theta(p^c_n)$,
otherwise we say that it has a \emph{sharp threshold}.

The classical approach for understanding sharp thresholds is based
on the Margulis--Russo formula
$\frac{d\mu_{p}\left(f\right)}{dp}=\mathrm{I}_{\mu_{p}}\left(f\right)$,
see \cite{margulis1977} and \cite{russo1982approximate}.
Here we note that if $f$ has a coarse threshold, then by the Mean
Value Theorem there is a constant $\epsilon>0$, some $p$ with $\mu_p(f)\in(\epsilon,1-\epsilon)$
and $p\mathrm{I}_{\mu_{p}}\left(f\right) = \Tt(1)$,
so one can apply various results mentioned in Section \ref{subsec:isoinf}.
Thus Bourgain's Theorem implies that
there is a set $J$ of $O\left(K\right)$ coordinates
such that $\mu_{p'}\left(f_{J\to1}\right)
\ge\mu_{p'}\left(f\right)+e^{-O\left(K^{2}\right)}$.
While this approach is useful for studying the behaviour of $f$
around the critical probability, it rarely gives any information
regarding the location of the critical probability.
Indeed, many significant papers are devoted to locating
the critical probability of specific interesting functions,
see e.g.\ the breakthroughs of Johansson, Kahn and Vu
\cite{johansson2008factors} and Montgomery
\cite{montgomery2018trees}.

A general result was conjectured
by Kahn and Kalai for the class of Boolean functions of the form
$f_n\colon \{0,1\}^{\binom{[n]}{2}}\to \left\{0,1\right\}$,
whose input is a graph $G$ and whose output is 1
if $G$ contains a certain fixed graph $H$.
For such functions there is a natural
`expectation heuristic' $p^E_n$ for the critical probability,
namely the least value of $p$ such that the expected number of copies
of any subgraph of $H$ in $G\left(n,p\right)$ is at least $1/2$.
Markov's inequality implies $p^c_n \ge p^E_n$,
and the hope of the Kahn--Kalai Conjecture is that there
is a corresponding upper bound up to some multiplicative factor.
They conjectured in \cite[Conjecture 2.1]{kahn2007thresholds}
that $p^c_n = O\left(p^E_n \log n\right)$, but this is widely open,
even if $\log n$ is replaced by $n^{o(1)}$.

The proposed strategy of Kahn and Kalai to this conjecture
via isoperimetric stability is as follows.
\begin{itemize}
\item{Prove a lower bound on $\mu_{p^E_n}\left(f_n\right)$.} 
\item{Show (e.g.\ via Russo's lemma)
that if $\left|\left[p_E,p_c\right]\right|$ is too large,
then the $p$-biased total influence at some point
in the interval $\left[p_E,p_c\right]$ must be relatively small.}
\item{Prove an edge-isoperimetric stability result that
      rules out the latter possibility.}
 \end{itemize}

Theorem \ref{thm:Variant of Kahn kalai}
makes progress on the third ingredient.
Combining it with Russo's Lemma, we obtain
the following result that can be used
to bound the critical probability.
Let $f$ be a monotone Boolean function.
We say that $f$ is \emph{$M$-global} in an interval $I$
if for each set $J$ of size $\le M$ and each $p\in I$ we have
$\mu_p\left(f_{J\to1}\right)\le \mu_p\left(f\right)^{0.01}$.

\begin{thm}\label{thm:Sharp threshold result}
  There exists an absolute constant $C$ such that the following
  holds for any monotone Boolean function $f$
  with critical probability $p_c$ and $p\le p_c$.
  Suppose for some $M>0$ that $f$ is $M$-global
  in the interval  $\left[p,p_c\right]$
  and that $\mu_{p}\left(f\right)\ge  e^{-M/C}$.
  Then $p_c\le M^C p$.
\end{thm}

To see the utility of Theorem \ref{thm:Sharp threshold result},
imagine that one wants to bound the critical probability as
$p^c_n \le p$, but instead of showing $\mu_p(f_n)\ge \frac{1}{2}$
one can only obtain a weaker lower bound
$\mu_p\left(f\right)\ge e^{-M/C}$, where $f$ is $M$-global;
then one can still bound the critical probability
as $p^c_n \le M^{O(1)} p$.

\section{Noise sensitivity} \label{sec:noise-intro}

Studying the effect of `noise' on a Boolean function
is a fundamental paradigm in various contexts,
including hypercontractivity (as in Section \ref{sec:hyp-intro})
and Gaussian isoperimetry
(via the invariance principle, see Section \ref{sec:inv}).
Roughly speaking, a function $f$ is `noise sensitive'
if $f(x)$ and $f(y)$ are approximately independent
for a random input $x$ and random small perturbation $y$ of $x$;
an equivalent formulation (which we adopt below)
is that the `noise stability'
of $f$ is small (compared to $\mu_p\left(f\right)$).
Formally, we use the following definition.

\begin{defn}

The noise stability $\mathrm{Stab}_{\rho}(f)$
of $f \in L^2(\{0,1\}^n,\mu_p)$ is defined by
\[
\mathrm{Stab}_{\rho}\left(f\right)=\left\langle f,\mathrm{T}_{\rho}f\right\rangle =\mathbb{E}_{\xb\sim\mu_{p}}\left[f\left(\xb\right)\mathrm{T}_{\rho}f\left(\xb\right)\right].
\]
A sequence $f_n$ of Boolean functions is said to be noise sensitive if
for each fixed $\rho$ we have $\mathrm{Stab}_{\rho}\left(f_n\right)=\mu_p\left(f_n\right)^2+o\left(\mu_p\left(f_n\right)\right).$
\end{defn}

Note that everything depends on $p$, but this will be clear from the context,
so we suppress $p$ from the notation $\mathrm{Stab}_{\rho}$.
Kahn, Kalai, and Linial \cite{kahn1988influence}
(see also \cite[Section 9]{o2014analysis})
showed that sparse subsets of the uniform cube
are noise sensitive, where we recall that the sequence
$(f_n)$ is \emph{sparse} if $\mu_{p}\left(f_n\right)=o\left(1\right)$ and
\emph{dense} if $\mu_{p}\left(f_n\right)=\Theta\left(1\right)$.

The relationship between noise and influence
in the cube under the uniform measure was further studied
by Benjamini, Kalai, and Schramm \cite{benjamini1999percolation}  (with applications to percolation), who gave
a complete characterisation:
a sequence $(f_n)$ of monotone dense Boolean functions
is noise sensitive if and only if the
sum of the squares of the influences of $f_n$ is $o\left(1\right)$.
Schramm and Steif \cite{schramm2010quantitative} proved that
any dense Boolean function on $n$ variables that can be computed by an
algorithm that reads $o\left(n\right)$ of the input bits is noise
sensitive. Their result had the striking application that the set of
exceptional times in dynamical critical site percolation on the
triangular lattice, in which an infinite cluster exists, is of Hausdorff
dimension in the interval  $\left[\frac{1}{6},\frac{31}{36}\right]$.
Ever since, noise sensitivity was considered in many other contexts
(see e.g.\ the recent results and open problems of Lubetzky--Steif \cite{Lubetzky2015strong} and Benjamini \cite{benjamini2019noise}).


In contrast to the uniform setting, in the $p$-biased setting
for small $p$ it is no longer true that sparse sets
are noise sensitive (e.g.\ consider dictators).
Our main contribution to the theory of noise sensitivity is showing
that `global' sparse sets are noise sensitive. Formally,
we say that a sequence $f_n$ of sparse Boolean functions
is \emph{weakly global} if for any $\eps, C > 0$ there is $n_0>0$
so that $\mu_p\left(\left(f_n\right)_{J\to 1}\right) < \eps$
for all $n>n_0$ and $J$ of size at most $C$.

\begin{thm}
  \label{thm:Noise sensitivity}
  Any weakly global sequence of Boolean functions is noise sensitive.
\end{thm}

We will deduce the following sharp threshold result,
which will underpin our combinatorial applications
discussed in the next section.

\begin{restatable}{thm}{qrsharpthreshold}
\label{theorem: quasirandom sharp threshold theorem}
For any $\aA >0$ there is $C>0$ such that
for any $\eps,p,q \in (0,1/2)$ with $q\ge (1+\aA)p$,
writing $r=C\log \eps^{-1}$ and $\dD=10^{-3r-1}\eps^3$,
any monotone $(r,\delta)$-global Boolean function $f$
with $\mu_p(f) \le \delta$
satisfies $\mu_q(f) \ge \mu_p(f) / \eps $.
\end{restatable}

\section{Hypergraph Tur\'an numbers}\label{sec:turan-intro}

A longstanding and challenging direction of research
in Extremal Combinatorics, initiated by Tur\'an in the 1940's,
is that of determining the maximum size of a $k$-graph
($k$-uniform hypergraph) $\mc{H} \sub \tbinom{[n]}{k}$
on $n$ vertices not containing some fixed $k$-graph $F$;
this is the Tur\'an number, denoted $\ex(n,F)$.
Tur\'an numbers of graphs (the case $k=2$)
are quite well-understood (if $F$ is not bipartite),
but there are very few results even for specific hypergraphs,
let alone general results for families of hypergraphs
(see the survey \cite{keevash2011hypergraph}).
Here we prove a number of general results on
Tur\'an numbers for the family of bounded degree
expanded hypergraphs (to be defined below),
thus solving several open problems.
Our proofs build upon our new sharp threshold theorems
and the Junta Method
of Keller and Lifshitz \cite{keller2017junta}
(which greatly extended the applications
of an approach initiated by Dinur and Friedgut
\cite{dinur2009intersecting}).
A striking feature of our results is their applicability
across an essentially optimal range of uniformities and sizes,
which previously seemed entirely out of reach.

\subsection{Cross matchings}

Before introducing the general setting of expanded hypergraphs,
we first consider an important case, which is in itself a source
of many significant problems, namely the problem of finding matchings.
In both theory and application, a wide range of significant questions
can be recast as existence questions for matchings
(see e.g.\ the books \cite{Lovasz,Schrijver}
and the survey \cite{keevash2018icm}).

Perhaps the most well-known open question concerning matchings,
due to Erd\H{o}s \cite{erdHos1965problem}, asks
how large a family ${\cal F} \subset \tbinom {[n]}{k}$
can be if it does not contain an $s$-matching,
i.e.\ sets $\{A_1,\ldots, A_s\}$ with $A_i \cap A_j = \es $
for all distinct $i,j\in [s]$.
Two natural families of such ${\cal F}$
are stars ${\cal S}_{n,k,s-1} :=
\big \{ A \in \tbinom {[n]}{k}: A \cap [s-1] \neq \es  \big \}$
and cliques ${\cal C}_{k,s-1} := \tbinom {[ks-1]}{k}$.
Erd\H{o}s conjectured that one of these families is always extremal.

\begin{conjecture}[Erd\H{o}s Matching Conjecture]
	\label{conj: EMC}
	Let $n \geq ks$ and suppose that ${\cal F} \subset
	\tbinom {[n]}{k}$ does not contain an $s$-matching.
	Then $|{\cal F}| \leq \max \big \{ |{\cal S}_{n,k,s-1}|,
	|{\cal C}_{k,s-1}| \big \}$.
\end{conjecture}

This conjecture remains open, despite an extensive literature,
of which we will mention a few highlights.
The case $s = 2$ is the classical Erd\H{o}s--Ko--Rado theorem \cite{erdos1961intersection}.
Erd\H{o}s and Gallai \cite{erdHos1959paths} confirmed the conjecture for $k=2$.
The case $k=3$ was proven by
\/Luczak and Mieczkowska \cite{luczak2014erdHos} for large $s$
and by Frankl \cite{frankl2012maximum} for all $s$.
Bollob\'as, Daykin and Erd\H{o}s \cite{bollobas1976sets}
proved the conjecture provided $n = \Omega (k^3s)$,
which was reduced to $n = \Omega (k^2s)$ by Huang, Loh and Sudakov
 \cite{huang2012size} and finally to $n = \Omega (ks)$
by Frankl \cite{frankl2013improved} (in fact to $n\gtrsim 2ks$,
recently improved by Frankl and Kupavskii \cite{franklkupavskii2018}
to $n\geq 5ks/3$ for large $s$), which is the optimal order of magnitude
for the extremal family to be a star rather than a clique
-- or even to just contain $s$ disjoint $k$-sets.

Our first result in this context is a cross version of that of Frankl,
which proves (a strengthened form of)
a conjecture of Huang, Loh and Sudakov \cite{huang2012size}.
Here we say that families $\mc{F}_1,\dots,\mc{F}_s$
cross contains a hypergraph  $\{A_1,\ldots, A_s\}$
(e.g.\ an $s$-matching) if $A_i \in \mc{F}_i$ for each $i \in [s]$.

\begin{thm} \label{thm: cross EMC bound}
There is a constant $C>0$ so that
if $n,s,k_1,\ldots, k_s \in {\mathbb N}$ with $k_i \le n/s$
and ${\cal F}_i \subset \tbinom {[n]}{k_i}$ with
$|{\cal F}_i| \geq |{\cal S}_{n,k_i,s-1}|$ for all $i\in [s]$,
either ${\cal F}_1,\ldots, {\cal F}_s$
cross contain an $s$-matching, or
there is $J \subset [n]$ with $|J| = s-1$
such that each ${\cal F}_i = {\cal S}_{n,k_i,J}
:= \{A \in \tbinom {[n]}{k_i}: A \cap J \neq \es \}$.
\end{thm}

\begin{rem} \label{rem:crossEMC}
Theorem \ref{thm: cross EMC bound}
in the case that all $k_i=k$
was proved by Huang, Loh and Sudakov  \cite{huang2012size}
for $n = \Omega (k^2s)$ and recently by
Frankl and Kupavskii \cite{franklkupavskii2019}
for $n = \Omega (ks \log s)$;
our result applies to $n = \Omega (ks)$,
which is the optimal order of magnitude.
Moreover, we obtain a strong stability result
(see Theorem \ref{thm:HLSapprox} below)
which gives structural information even if we only assume
that the size of each family is within a constant factor
of that of a star:\ either there is a cross matching
or some family correlates strongly with a star.
Besides having independent interest,
this stability result will play a key role
in the proof of our general Tur\'an results.
\end{rem}

\subsection{Expanded hypergraphs}

As mentioned above, there are very few general results
on Tur\'an numbers for a family of hypergraphs.
One family for which there has been substantial progress
is that of expanded graphs (see the survey \cite{mubayi2016survey}).
Given an $r$-graph $G$ and $k \geq r$, the $k$-expansion
$G^+ = G^+(k)$ is the $k$-uniform hypergraph
obtained from $G$ by adding $k-r$ new vertices to each edge,
i.e.\ $G^+$ has edge set $\{e \cup S_e: e\in E(G)\}$
where $|S_e| = k-r$, $S_e \cap V(G) = \es$
and $S_e \cap S_{e'} = \es $ for all distinct $e, e' \in E(G)$.
In particular, a $k$-graph $s$-matching
is the $k$-expansion of a graph $s$-matching.

When $G$ is a graph (the case $r=2$),
in the non-degenerate case when $k$ is less than
the chromatic number $\chi(G)$ the Tur\'an numbers
$\ex(n,G^+(k))$ are well-understood
(see \cite[Section 2]{mubayi2016survey}),
so the main focus for ongoing research
is the degenerate case $k \ge \chi(G)$.
Here Frankl and F\"uredi \cite{frankl1987exact}
introduced the following important parameter
and corresponding construction that seems to often
determine the asymptotics of the Tur\'an number.
For any $r$-graph $G$,
we call $S \subset V(G^+)$ a \emph{crosscut}
if $|E \cap S| = 1$ for all $E \in G^+$.
The crosscut $\cc {G}$ of $G$ is the
size of the minimal such set, i.e.\
	\begin{equation*}
		\cc {G} :=
		\min \big \{|S|: S \subset V(G^+) \mbox { with }
		|E \cap S| = 1 \mbox { for all } E \in G^+ \big \}.
	\end{equation*}
It is easy to see that $\cc {G}$ exists for $k\geq r+1$
and is independent of $k$. Clearly,
 \[ {\cal S}_{n,k,\cc {G}-1}^{(1)}
 :=\big \{ A \in \tbinom {[n]}{k}:
 |A\cap [\cc  {G}-1]| = 1\big \} \]
is $G^+$-free. Moreover, this simple construction
determines the asymptotics
of $\ex(n,G^+(k))$ for $n>n_0(k,G)$ for several graphs $G$,
including paths  \cite{furedi2014exact, kostochka2015turan},
cycles \cite{furedi2014hypergraph, kostochka2015turan}
and trees \cite{furedi2014linear, kostochka2017turan}.
Given this phenomenon, according to Mubayi and Verstra\"ete
\cite{mubayi2016survey}, one of the major open problems
on expansions is to decide when the Tur\'an number
is asymptotically determined by the crosscut construction.
Our next result resolves this problem for all bounded degree
$r$-graphs (so in particular for graphs) in a range
of parameters that is optimal up to constant factors.
Moreover, we also obtain a strong structural
approximation for any family that is close to extremal
(see Theorem \ref{thm:junta} below).

\begin{thm}\label{thm: G+ approx}
For any $r, \DD \ge 2$ and $\eps>0$ there is $C>0$
so that the following holds for any
$r$-graph $G$ with $s$ edges,
maximum degree $\DD (G) \leq \DD $
and $\cc {G} \geq 2$.
For any $k,n \in {\mathbb N}$ with
$C \leq k \leq n/Cs$ we have
$\ex(n,G^+(k)) = (1 \pm \eps )
|{\cal S}^{(1)}_{n,k,\cc {G}-1}|$.
\end{thm}

\begin{rem}
Some lower bound on $k$ is necessary to obtain the conclusion
in Theorem \ref{thm: G+ approx}. Indeed, we have already mentioned
that the non-degenerate case $k \le \chi(G)$ when $G$ is a graph
exhibits different behaviour
(a complete partite $k$-graph shows that $\ex(n,G^+) = \OO(n/k)^k$),
and moreover, examples in \cite{mubayi2016survey} show that some
lower bound on $k$ may be necessary even if $G$ is bipartite
(e.g.\ if $G=K_{9,9}$ then consider the $3$-graph of triangles
in a suitably dense random graph made $G$-free by edge deletions).
The upper bound on $k$ in our result is also necessary
up to the constant factor by space considerations, as even the
complete $k$-graph $\tbinom{[n]}{k}$ can only contain $G^+(k)$
if $n \ge |V(G^+)| = |V(G)| + (k-2)s$.
With the exception of Frankl's matching theorem \cite{frankl2013improved},
Theorem \ref{thm: G+ approx} appears to be the only known Tur\'an result
in which both the uniformity $k$ and the size $s$
can vary over such a wide range.
\end{rem}

Next we consider conditions under which we can refine
the asymptotic result of Theorem \ref{thm: G+ approx}
and determine the Tur\'an number $\ex(n,G^+)$ exactly.
One complication here is that crosscuts may be beaten by stars
${\cal S}_{n,k,\tau(G)-1}$, where
\begin{equation*}
		\tau (G) := \min \big \{|S|: |S\cap e|
		\geq 1 \mbox { for all } e \in E(G) \big \}
\end{equation*}
is the transversal number of $G$.
Clearly $\tau(G) \le \cc{G}$. For fixed $s$,
crosscuts cannot be beaten by smaller stars,
but this may not hold when $s$ grows with $n$,
as then edges with more than one vertex
in the base of the star are significant.
Another complication is that lower order correction
terms are necessary for certain $G$, e.g.\ for
$k$-expanded paths $P_\ell^+(k)$ of length $\ell$
for $n>n_0(k,\ell)$ we have
$\ex(n,P_3^+(k)) = \tbinom{n-1}{k-1} = |{\cal S}_{n,k,1}|$,
as predicted by the crosscut/star construction, but
$\ex(n,P_4^+(k)) = \tbinom{n-1}{k-1} + \tbinom{n-3}{k-2}$,
as we can add all sets containing some fixed pair of vertices.
This is analogous to the familiar situation in extremal graph
theory where we only expect exact results for graphs that
are critical with respect to the key parameter
of the extremal construction. Accordingly, we introduce
the following analogous concept of criticality for
expanded hypergraphs with respect to crosscuts and stars:
we say that $G$ is \emph{critical}
if it has an edge $e$ such that
\[ \cc{G \sm e} = \tau(G \sm e) < \tau(G) = \cc{G}. \]
We obtain the following general exact result for Tur\'an numbers.

\begin{thm} \label{thm:critical}
For any $r,\DD \ge 2$ there is $C >0$ such that
for any critical $r$-graph $G$ with $s$ edges, maximum degree
$\DD(G) \le \DD$ and $C \leq k \leq n/Cs$ we have
$\ex(n,G^+(k)) = |{\cal S}_{n,k,\cc{G}-1}|$.
\end{thm}

This result applies to many graphs considered
in the previous literature, such as paths of odd length.
Paths of even length are not critical,
but satisfy a generalised criticality property:
deleting one edge does not reduce the transversal number,
but deleting two edges (whether disjoint or intersecting)
does reduce the crosscut number.
Thus we have the following natural construction
for excluding any expanded path $P_\ell^+$ of length $\ell$.
Let $\f^*_{n,k,\ell} = \s_{n,k,J}$
with $|J|=\cc{P_\ell}-1$ if $\ell$ is odd,
or if $\ell$ is even obtain $\f^*_{n,k,\ell}$ from $\s_{n,k,J}$
by adding $\{A \in \tbinom{[n]}{k}: T \sub A\}$
for some $T \in \tbinom{[n] \sm J}{2}$.
Clearly $\f^*_{n,k,\ell}$ is $P_\ell^+$-free.
F\"uredi, Jiang and Seiver \cite{furedi2014exact} showed
that $\ex(n,P_\ell^+) = |\f^*_{n,k,\ell}|$ provided $n \gg n_0(k,s)$,
and conjectured that this holds provided $n \geq Cks$.
We prove this conjecture.

\begin{cor}\label{cor: expanded paths}	
There is $C>0$ so that if $n,k,\ell\in {\mathbb N}$ and
$C \leq k \leq n/C\ell$ then $\ex(n,P_\ell^+) = |\f^*_{n,k,\ell}|$.
\end{cor}


\subsection{The Junta Method}

In recent years, the Analysis of Boolean functions
has found significant application in Extremal Combinatorics,
via the connection provided by the Margulis-Russo formula
between the sharp threshold phenomenon
and influences of Boolean functions.
This approach was initiated by Dinur and Friedgut
\cite{dinur2009intersecting}, who applied a theorem of
Friedgut \cite{friedgut1998boolean} on Boolean functions
of small influence to prove that large uniform intersecting families
can be approximated by juntas,
i.e.\ families that depend only on a few coordinates.
This connection has since played a key role
in intersection theorems for a variety of settings,
including graphs \cite{ellis2012triangle},
permutations \cite{ellis2011intersecting}
and sets \cite{ellis2016stability, Ellis2017}.

The approach of Dinur and Friedgut was substantially
generalised by Keller and Lifshitz \cite{keller2017junta}
to apply to a variety of Tur\'an problems on expanded hypergraphs.
At a very high level, their Junta Method is a version
of the Stability Method in Extremal Combinatorics,
in that it consists of two steps:
an approximate step that determines the rough structure
of families that are close to optimal,
and an exact step that refines the structure
and determines the optimal construction.
Their approximate step consisted of showing that
any $G^+$-free family is approximately contained
in a $G^+$-free junta. This is also true in our approach,
but the crucial difference is that they required
the number $s$ of edges in $G$ to be fixed,
whereas we allow it to grow as a function of $n$.
Friedgut's theorem can no longer be applied in this setting,
as we require a threshold result for
Boolean functions $f:\{0,1\}^n \to \{0,1\}$
according to the $p$-biased measure $\mu_p$
in the sparse regime where both $p$ and $\mu_p(f)$
may be functions of $n$ that approach zero.
Our new sharp threshold theorems, as in subsection \ref{sec:kahn kalai}, provide the needed
improvement on the analytic side which, when combined with a number of additional combinatorial ideas,
allow us to obtain the following junta approximation theorem.

\begin{restatable}{thm}{junta}
\label{thm:junta}
For any $r,\DD \ge 2$ and $\eps>0$ there are $c,C>0$
so that, given an $r$-graph $G$ with $s$ edges
and maximum degree $\DD (G) \leq \DD $,
for any $G^+$-free ${\cal F} \sub \tbinom {[n]}{k}$
with $C \leq k \leq \tfrac {n}{Cs}$,
there is $J \sub V(G)$ with $|J| \leq \cc {G}-1$
and $|{\cal F} \sm {\cal S}_{n,k,J}|
\leq \eps |{\cal S}_{n,k,\cc {G}-1}|$.
\end{restatable}

We note that Theorem \ref{thm: G+ approx}
is immediate from Theorem \ref{thm:junta},
as for $k \geq C \gg \eps^{-1}$ we have
\[ \ex(n,G^+) \geq |{\cal S}_{n,k,\cc {G}-1}^{(1)}|
\geq (1-\eps )|{\cal S}_{n,k,\cc {G}-1}|.\]

\newpage

\part{Hypercontractivity of global functions} \label{part:hyp}

This part concerns our theory of global hypercontractivity,
which underpins all the results of this paper.
We start in Section \ref{sec:hyp}
by proving Theorem \ref{thm:Hypercontractivity},
which is the form of our result that suffices
for our subsequent applications.
In the remainder of the part
(which could be omitted by a reader
primarily interested in these applications)
we investigate the theory more deeply,
as this has independent interest and further applications.
Section \ref{sec:genhyp} generalises our hypercontractivity result
in two directions: we consider general norms and general product spaces.
We conclude this part in Section \ref{sec:inv}
by proving our $p$-biased version of the
Invariance Principle and remarking on some of its applications
(we omit the details of these for the sake of brevity).

\subsection*{Notation} \label{sec:not}

Here we summarise some notation
and basic properties of Fourier analysis on the cube.
We fix $p \in (0,1)$ and suppress it in much of our notation,
i.e.\ we consider $\{0,1\}^n$ to be equipped with the $p$-biased
measure $\mu_p$, unless otherwise stated. We let $\sigma = \sqrt {p(1-p)}$
(the standard deviation of a $p$-biased bit). For each $i\in [n]$
we define $\chi_{i}\colon\left\{ 0,1\right\} ^{n}\to {\mathbb R}$
by $\chi_{i}\left(x\right)=\frac {x_i-p}{\sigma }$
(so $\chi_i$ has mean $0$ and variance $1$).
We use the orthonormal Fourier basis $\left\{ \chi_{S}\right\}_{S \subset [n]}$
of $L^{2}\left(\left\{ 0,1\right\} ^{n},\mu_{p}\right)$,
where each $\chi_{S}:=\prod_{i\in S}\chi_{i}$.
Any $f: \{0,1\}^n \to {\mathbb R}$ has a unique expression
$f=\sum _{S\sub [n]}\hat{f}(S)\chi _S$
where $\{\hat f(S)\}_{S\subset [n]}$ are the
\emph{$p$-biased Fourier coefficients of }$f$.
Orthonormality gives the Plancherel identity
$\bgen{f,g} = \sum_{S \sub [n]} \hat{f}(S) \hat{g}(S)$.
In particular, we have the Parseval identity
$\mb{E}[f^2] = \|f\|_2^2 = \bgen{f,f}
= \sum_{S \sub [n]} \hat{f}(S)^2$.
For $\mc{F} \sub \{0,1\}^n$ we  define the $\mc{F}$-truncation
$f^{\mc{F}} = \sum_{S \in \mc{F}} \hat{f}(S)\chi _S$.
Our truncations will always be according to some degree threshold $r$,
for which we write $f^{\le r} = \sum_{|S| \le r} \hat{f}(S)\chi _S$.

For $i \in [n]$, the $i$-\emph{derivative} $f_i$
and $i$-\emph{influence} $\mathrm{I}_i(f)$ of $f$ are
\begin{align*}
f_i & =\mathrm{D}_{i}\left[f\right] =\sigma \big (f_{i\to1}-f_{i\to0}\big )
= \sum_{S: i\in S}\hat{f}\left(S\right)\chi_{S\backslash\left\{ i\right\} },
\text{ and} \\
\mathrm{I}_i(f) & = \| f_{i \to 1} - f_{i \to 0} \|_2^2
= \sS^{-2} \mb{E}[f_i^2]
= \tfrac{1}{p(1-p)} \sum_{S: i\in S} \hat{f}(S)^2.
\end{align*}
The \emph{influence} of $f$ is
\begin{align} \label{eq:inf}
\mathrm{I}(f) = \sum_i \mathrm{I}_i(f)
= (p(1-p))^{-1} \sum_S |S| \hat{f}(S)^2.
\end{align}
In general, for $S \sub [n]$, the $S$-\emph{derivative} of $f$ is
obtained from $f$ by sequentially applying $\mathrm{D}_i$ for each $i \in S$, i.e.\
\[
  \mathrm{D}_S(f)
= \sigma ^{|S|} \sum_{x\in\left\{ 0,1\right\} ^{S}}\
 (-1)^{\left|S\right|-\left|x\right|}{f}_{S\to x}
 = \sum _{T: S \subset T} \hat f (T) \chi _{T\setminus S}.
\]
The \emph{$S$-influence} of $f$
(as in  Definition \ref{def:generalised influences}) is
\begin{align}
\label{equation: gen influence Fourier expression}
\mathrm{I}_{S}(f)= \sigma ^{-2|S|}\| \mathrm{D}_S\left(f\right) \|_2^2
= \sigma ^{-2|S|}\sum _{E: S \subset E} \hat f (E)^2.
\end{align}
Recalling that a function $f$ has $\alpha$-small generalised influences
if $\mathrm{I}_{S}(f)\le\alpha \mb{E}[f^2]$ for all $S\subset [n]$,
we see that this is equivalent to
$\mb{E}[\mathrm{D}_S\left(f\right)^2] \le \alpha \sigma^{2|S|} \mb{E}[f^2]$
for all $S \subset [n]$.

\section{Hypercontractivity and generalised influences}
\label{sec:hyp}

In this section we prove our hypercontractive inequality
(Theorem \ref{thm:Hypercontractivity}), which is the fundamental result
that underpins all of the results in this paper.

The idea of the proof is to reduce hypercontractivity in $\mu_p$
to hypercontractivity in $\mu_{1/2}$ via the `replacement method'
(the idea of Lindeberg's proof of the Central Limit Theorem, and of
the proof of Mossel, O'Donnell and Oleszkiewicz
\cite{mossel2010noise} of the invariance principle).
Throughout this section we fix $f:\{0,1\}^n \to \mb{R}$
and express $f$ in
the $p$-biased Fourier basis as $\sum_S \hat{f}(S) \chi^p_S$,
where $\chi^p_S = \prod_{i \in S} \chi^p_i$
and $\chi^p_i(x) = \tfrac{x_i-p}{\sS}$
(the same notation as above, except that we introduce the
superscript $p$ to distinguish the $p$-biased and uniform settings).

For $0 \le t \le n$ we define
$f_t = \sum_S \hat{f}(S) \chi_S^t$,  where
\[\chi^t_S = \prod\limits_{i\in S\cap [t]}{\chi^{1/2}_i(x)}
 \prod\limits_{i\in S\setminus [t]}{\chi^{p}_i(x)}
\in L^2(\{0,1\}^{[t]},\mu_{1/2}) \times
 L^2(\{0,1\}^{[n] \sm [t]},\mu_p). \]
Thus $f_t$ interpolates from $f_0=f \in L^2(\{0,1\}^n,\mu_p)$
to $f_n = \sum_S \hat{f}(S) \chi^{1/2}_S
\in L^2(\{0,1\}^n,\mu_{1/2})$.
As $\{ \chi^t_S: S \sub [n]\}$ is an orthonormal basis
we have $\|f_t\|_2 = \|f\|_2$ for all $t$.

We also define noise operators $\mathrm{T}^t_{\rho',\rho}$
on $L^2(\{0,1\}^{[t]},\mu_{1/2}) \times
 L^2(\{0,1\}^{[n] \sm [t]},\mu_p)$
by $\mathrm{T}^t_{\rho',\rho}(g)(\xb) =
\mb{E}_{\yb \sim N_{\rho',\rho}(\xb)}[f(\yb)]$,
where to sample $\yb$ from $N_{\rho',\rho}(\xb)$,
for $i \le t$ we let $y_i=x_i$ with probability $\rho'$
or otherwise we resample $y_i$ from $\mu_{1/2}$,
and for $i>t$ we let $y_i=x_i$ with probability $\rho$
or otherwise we resample $y_i$ from $\mu_p$.
Thus $\mathrm{T}^t_{\rho',\rho}$ interpolates from
$\mathrm{T}^0_{\rho',\rho}=\mathrm{T}_{\rho}$ (for $\mu_p$) to
$\mathrm{T}^n_{\rho',\rho}=\mathrm{T}_{\rho'}$ (for $\mu_{1/2}$).

We record the following estimate for $4$-norms
of $p$-biased characters:
\[ \lambda := \mb{E}[ (\chi^p_i)^4 ]
= \sS^{-4}( p (1-p)^4 + (1-p) p^4 )
= \sS^{-2}( (1-p)^3 + p^3 ) \le \sS^{-2}. \]

The core of our argument by replacement is the
following lemma which controls the evolution of
$\mb{E}[(\mathrm{T}^t_{2\rho,\rho} f_t)^4]
= \norm{\mathrm{T}^t_{2\rho,\rho} f_t}_4^4$
for $0 \le t \le n$.

\begin{lem} \label{lem:Replacement step}
$\mb{E}[(\mathrm{T}^{t-1}_{2\rho,\rho} f_{t-1})^4]
\le \mb{E}[(\mathrm{T}^t_{2\rho,\rho} f_t)^4]
+ 3\lambda\rho^4 \mb{E}[(\mathrm{T}^t_{2\rho,\rho} ((\mathrm{D}_t f)_t) )^4]$.
\end{lem}

\begin{proof}
We write
\begin{align*}
f_t & = \chi^{1/2}_t g + h \ \ \text{ and } \ \
 f_{t-1} =  \chi^p_t g + h, \ \ \text{ where } \\
g & = (\mathrm{D}_t f)_t
= \sum_{S: t \in S} \hat{f}(S) \chi^t_{S \sm \{t\}}
= \sum_{S: t \in S} \hat{f}(S) \chi^{t-1}_{S \sm \{t\}}
= (\mathrm{D}_t f)_{t-1}, \ \ \text{ and } \\
h & = \mb{E}_{x_t \sim \mu_{1/2}} f_t
= \sum_{S: t \notin S} \hat{f}(S) \chi^t_S
= \sum_{S: t \notin S} \hat{f}(S) \chi^{t-1}_S
= \mb{E}_{x_t \sim \mu_p} f_{t-1}.
\end{align*}
We also write
\begin{align*}
\mathrm{T}^t_{2\rho,\rho} f_t & = 2\rho \chi^{1/2}_t d + e \ \ \text{ and } \ \
\mathrm{T}^{t-1}_{2\rho,\rho} f_{t-1} = \rho \chi^p_t d + e, \ \ \text{ where } \\
d & = \mathrm{T}^t_{2\rho,\rho} g = \mathrm{T}^{t-1}_{2\rho,\rho} g
\ \ \text{ and } \ \
e = \mathrm{T}^t_{2\rho,\rho} h = \mathrm{T}^{t-1}_{2\rho,\rho} h.
\end{align*}
We can calculate the expectations in the statement of the lemma
by conditioning on all coordinates other than $x_t$, i.e.\
$\mb{E}_{\xb}[ \cdot ] = \mb{E}_{\xb'}[ \mb{E}_{x_t}[ \cdot \mid \xb'] ]$
where $\xb'$ is obtained from $\xb=(x_1,\dots,x_n)$ by removing $x_t$.
It therefore suffices to establish the required inequality
for each fixed $\xb'$ with expectations over the choice of $x_t$;
thus we can treat $d$ and $e$ as constants, and it suffices to show
\begin{equation} \label{eq:ets}
\mb{E}_{x_t}[ (\rho d\chi^p_t + e)^4 ]
\le \mb{E}_{x_t}[ (2\rho d\chi^{1/2}_t + e)^4 ]
+ 3\lambda\rho^4 d^4.
\end{equation}
As $\chi^p_t$ has mean $0$, we can expand the left hand side
of \eqref{eq:ets} as
\[ (\rho d)^4 \mb{E}[(\chi^p_t)^4]
 + 4 e (\rho d)^3  \mb{E}[(\chi^p_t)^3]
 + 6 e^2 (\rho d)^2  \mb{E}[(\chi^p_t)^2] + e^4
 \le 3 \lL (d\rho)^4 + 8 (de\rho)^2 + e^4, \]
where we bound the second term
using Cauchy-Schwarz then AM-GM by
\[ 4 \cdot \mb{E}[(d\rho\chi^p_t)^4]^{1/2}
\cdot \mb{E}[(de\rho\chi^p_t)^2]^{1/2}
\le 2 \left( \mb{E}[(d\rho\chi^p_t)^4]
+  \mb{E}[(de\rho\chi^p_t)^2] \right)
= 2( \lL (d\rho)^4 + (de\rho)^2 ). \]
Similarly, as $\mb{E}[\chi^{1/2}_t]
= \mb{E}[(\chi^{1/2}_t)^3] = 0$,
we can expand the first term
on the right hand side of \eqref{eq:ets} as
\[ (2\rho d)^4 \mb{E}[(\chi^{1/2}_t)^4]
 + 6 e^2 (2\rho d)^2  \mb{E}[(\chi^{1/2}_t)^2] + e^4
= (2\rho d)^4 + 6(2\rho de)^2 + e^3
\ge 8 (de\rho)^2 + e^4. \]
The lemma follows.
\end{proof}

Now we apply the previous lemma inductively to prove the following estimate.

\begin{lem} \label{hypinduct}
$\| \mathrm{T}^i_{2\rho,\rho} f_i \|_4^4  \le \sum_{S \sub [n] \sm [i]}
 (3\lL\rho^4)^{|S|} \| \mathrm{T}^n_{2\rho,\rho} ((\mathrm{D}_S f)_n) \|_4^4$
 for all $0 \le i \le n$.
\end{lem}

\begin{proof}
We prove the inequality by induction on $n-i$
simultaneously for all functions $f$.
If $n=i$ then equality holds trivially.
Now suppose that $i<n$.
By Lemma \ref{lem:Replacement step} with $t=i+1$,
and the induction hypothesis applied to $f$ and $\mathrm{D}_t f$
with $i$ replaced by $t$, we have
\begin{align*}
\| \mathrm{T}^i_{2\rho,\rho} f_i \|_4^4
& \le  \| \mathrm{T}^t_{2\rho,\rho} f_t \|_4^4
+ 3\lL\rho^4 \| \mathrm{T}^t_{2\rho,\rho} ((\mathrm{D}_t f)_t) \|_4^4 \\
& \le \sum_{S \sub [n] \sm [t]}
 (3\lL\rho^4)^{|S|} \| \mathrm{T}^n_{2\rho,\rho} ((\mathrm{D}_S f)_n) \|_4^4
+ 3\lL\rho^4 \sum_{S \sub [n] \sm [t]}
 (3\lL\rho^4)^{|S|} \| \mathrm{T}^n_{2\rho,\rho} ((\mathrm{D}_S \mathrm{D}_t f)_n) \|_4^4 \\
& = \sum_{S \sub [n] \sm [i]}
 (3\lL\rho^4)^{|S|} \| \mathrm{T}^n_{2\rho,\rho} ((\mathrm{D}_S f)_n) \|_4^4. & \qedhere
\end{align*}
\end{proof}

In particular, recalling that
$\mathrm{T}^0_{2\rho,\rho}=\mathrm{T}_{\rho}$ on $\mu_p$ and
$\mathrm{T}^n_{2\rho,\rho}=\mathrm{T}_{2\rho}$ on $\mu_{1/2}$,
the case $i=0$ of Lemma \ref{hypinduct} is as follows.

\begin{prop}
\label{prop:Reason for hypercontractivity}
$\| \mathrm{T}_{\rho} f \|_4^4  \le \sum_{S \sub [n]}
 (3\lL\rho^4)^{|S|} \| \mathrm{T}_{2\rho} ((\mathrm{D}_S f)_n) \|_4^4$.
\end{prop}

The $4$-norms on the right hand side of
Proposition \ref{prop:Reason for hypercontractivity}
are with respect to the uniform measure $\mu_{1/2}$,
where we can apply standard hypercontractivity
(the `Beckner-Bonami Lemma') for $\rho \le 1/2\sqrt{3}$
to obtain $\| \mathrm{T}_{2\rho} ((\mathrm{D}_S f)_n) \|_4^4
\le \| (\mathrm{D}_S f)_n \|_2^4 = \| \mathrm{D}_S f \|_2^4 = \sS^{4|S|} \mathrm{I}_S(f)^2$.
Recalling that $\lL \le \sS^{-2}$,
we deduce the following bound for $\| \mathrm{T}_{\rho} f \|_4^4$
in terms of the generalised influences of $f$.

\begin{thm} \label{thm:hypref}
If $\rho\le 1/\sqrt{12}$ then
$\norm{\mathrm{T}_{\rho} f}_4^4  \le \sum_{S \sub [n]}
(3\lL\rho^4)^{|S|} \| \mathrm{D}_S f \|_2^4
\le \sum_{S \sub [n]} (3\sS^2\rho^4)^{|S|} \mathrm{I}_S(f)^2$.
\end{thm}

Now we deduce our hypercontractivity inequality.
It is convenient to prove a
slightly stronger statement, which implies
Theorem \ref{thm:Hypercontractivity} using
$\|\mathrm{D}_S f\|_2^2 = \sS^{2|S|} \mathrm{I}_S(f)
\le \lL^{-|S|} \mathrm{I}_S(f)$ and
$\|\mathrm{T}_{1/5} f\|_4 \le
\|\mathrm{T}_{1/\sqrt{24}}f\|_{4}$
(any $\mathrm{T}_\rho$ is a contraction
in $L^p$ for any $p \ge 1$).

\begin{thm} \label{thm:hyp+}
Let $f\in L^{2}\left(\left\{ 0,1\right\} ^{n},\mu_{p}\right)$
with all $\|\mathrm{D}_S f \|_{2}^2 \le \bB \lL^{-|S|} \mb{E}[f^2]$. Then
$\|\mathrm{T}_{1/\sqrt{24}}f\|_{4} \le \bB^{1/4} \|f\|_{2}$.
\end{thm}

\begin{proof}
By Theorem \ref{thm:hypref} applied to $\mathrm{T}_{1/\sqrt{2}} f$
with $\rho=1/\sqrt{12}$ we have
\[\norm{\mathrm{T}_{1/\sqrt{24}} f}_4^4  \le \sum_{S \sub [n]}
(3\lL\rho^4)^{|S|} \| \mathrm{D}_S \mathrm{T}_{1/\sqrt{2}} f \|_2^4.\]
As $\| \mathrm{D}_S \mathrm{T}_{1/\sqrt{2}} f \|_2^2
= \sum_{E: S \sub E} 2^{-|E|} \hat{f}(E)^2
\le \sum_{E: S \sub E} \hat{f}(E)^2
= \| \mathrm{D}_S f \|_2^2 \le \bB \lL^{-|S|} \mb{E}[f^2]$
we deduce
\begin{align*}
\norm{\mathrm{T}_{1/\sqrt{24}} f}_4^4  \le \sum_{S \sub [n]}
\sum_{E: S \sub E} \bB \mb{E}[f^2] 2^{-|E|} \hat{f}(E)^2
= \bB \mb{E}[f^2] \sum_E \hat{f}(E)^2 = \bB \|f\|_2^4.
& \qedhere
\end{align*}
\end{proof}

\subsection*{Hypercontractivity in practice} \label{sec:practice}
We will mostly use the following application of the hypercontractivity
theorem.
\begin{lem} \label{lem:applying_hypercontractivity}
  Let $f$ be a function of degree $r$. Suppose that
  $\mathrm{I}_S(f)\le \delta$ for all $|S|\le r$. Then \[
    \|f\|_4\le
    5^{\frac{3r}{4}}\delta^{\frac{1}{4}}\left\| f \right\|_2^{0.5}.
    \]
\end{lem}

The proof uses the following lemma, which is immediate from the Fourier expression in \eqref{equation: gen influence Fourier expression}.

\begin{lem} \label{obs}
$\mathrm{I}_S(f^{\le r}) \le \mathrm{I}_S(f)$
for all $S \sub \left[n\right]$ and $\mathrm{I}_S(f^{\le r})=0$ if $\left|S\right|>r$.
\end{lem}

\begin{proof}[Proof of Lemma \ref{lem:applying_hypercontractivity}]
Write $f = \mathrm{T}_{1/5}(h)$, where
$h = \sum_{|T| \le r} 5^{|T|} \hat{f}(T) \chi_T$.
We will bound the 4-norm of $f$ by applying Theorem
\ref{thm:Hypercontractivity} to $h$,
so we need to bound the generalised influences of $h$.

By Lemma \ref{obs}, for $S \sub [n]$ we have
$\mathrm{I}_S(h)=0$ if $|S|>r$. For $|S|\le r$,  we have
\[
 \mathrm{I}_S(h) = \sS^{-2|S|} \sum_{T: S \sub T, |T| \le r}
 5^{2|T|} \hat{f}(T)^2
\le 5^{2r} \mathrm{I}_S(f) \le 5^{2r} \delta
= \alpha \|h\|_2^2,
\]
where
$\alpha = 5^{2r} \delta /\|h\|_2^2$.
By Theorem \ref{thm:Hypercontractivity}, we have
\[
  \|f \|_4 = \|\mathrm{T}_{1/5} h \|_{4}
\le \alpha^{\frac{1}{4}} \| h \|_{2}
= 5^{r/2} \delta^{\frac{1}{4}}  \sqrt{ \|h\|_2}
\le 5^{\frac{3r}{4}} \delta^{\frac{1}{4}}\sqrt{\|f\|_2}.
\]
In the final inequality we used $\|h\|_2\leq 5^r \|f\|_2$,
which follows from Parseval.
\end{proof}

\section{General hypercontractivity} \label{sec:genhyp}

In this section we generalise Theorem \ref{thm:Hypercontractivity} in two
different directions. One direction is showing hypercontractivity
from general $q$-norms to the $2$-norm (rather than merely
treating the case $q=4$); the other is replacing the cube
by general product spaces.

\subsection{Hypercontractivity with general norms}
\label{subsec:multilinear}

We start by describing a more convenient general setting
in which we replace characters on the cube
by arbitrary random variables.
To motivate this setting, we remark that one can extend
the proof of Theorem \ref{thm:hypref} to any
random variable of the form
\begin{equation} \label{eq:fRV}
f  = \sum_{S\subset [n]} a_S\prod_{i\in S} \mathbf{Z}_i,
\end{equation}
where $\mathbf{Z}_1,\ldots,\mathbf{Z}_n$
are independent real-valued random variables having
expectation $0$, variance $1$ and $4$th moment at most
$\sigma^{-2}$. To motivate the analogous setting for general $q>2$,
we note that the characters $\chi_i^p$ satisfy
\[
  \mathbb{E}[|\chi_i^p|^q]\le
  \|\chi_i^p\|_\infty^{q-2}\|\chi_i^p\|_2^2 = \sigma^{2-q}.
  \]
This suggests replacing the $4$th moment condition by
$\|\mathbf{Z}_i\|_q^q\le \sigma^{2-q}$.
Given $f$ as in \eqref{eq:fRV}, we define
the (generalised) derivatives by substituting
the random variables $Z_i$ for the characters $\chi_i^p$
in our earlier Fourier formulas, i.e.\
\[
  \mathrm{D}_i[f]=\sum_{S:\,i\in S}a_S\prod_{j\in S\setminus \{i\}}
  \mathbf{Z}_i
  \quad \text{and} \quad
  \mathrm{D}_T(f)=\sum_{S:\, T\subset
    S}a_S\prod_{j\in S\backslash T}\mathbf{Z}_i,
\]
Similarly, we adopt analogous definitions of the
generalised influences and noise operator, i.e.\
 \[
  \mathrm{I}_S[f]=
  \|\frac{1}{\sigma}\mathrm{D}_S[f]\|_2^2
  \quad \text{and} \quad
  \mathrm{T}_\rho[f]=\sum_S \rho^{|S|}a_S\prod_{i\in S}{\mathbf{Z}_i}.
\]

We prove the following hypercontractive inequality.
\begin{thm}\label{thm:qth moment}
  Let $q \ge 2$ and $\mathbf{Z}_1,\ldots,\mathbf{Z}_n$
  be independent real-valued  random variables satisfying
  \[
    \mathbb{E}[\mathbf{Z}_i]=0, \quad  \mathbb{E}[\mathbf{Z}_i^2]=1,
    \quad \text{and} \quad \mathbb{E} [\left| \mathbf{Z}_i \right|^q] \le \sigma^{2-q}.
  \]
  Let $f=\sum_{S\subset [n]}  a_S\prod_{i\in S}\mathbf{Z}_i$ and
  $\rho< \frac{1}{2q^{1.5}} $. Then
  \[
    \| \mathrm{T}_{\rho}f \|_q^q \le \sum_{S\subset [n]}
    \sigma^{(2-q)|S|}\| \mathrm{D}_S(f) \|_2^{q}.
    \]
  \end{thm}

  Theorem \ref{thm:qth moment} is a qualitative generalisation
  of Theorem \ref{thm:hypref} (with smaller $\rho$, which we do not
  attempt to optimise). The following generalised variant of
  Theorem \ref{thm:Hypercontractivity} follows by repeating
  the proof in Section \ref{sec:hyp}.
  \begin{thm}
    Let $q>2$, let $f=\sum_{S \sub [n]}a_S\prod_{i\in S}\mathbf{Z_i}$ let
    $\delta>0$, and let $\rho \le (2q)^{-1.5}$. Suppose
    that $\mathrm{I}_S[f] \le \beta \|f\|_2^2$ for all $S\subset
    [n]$. Then
    \[
      \| \mathrm{T}_{\rho} [f] \|_q \le \beta^{\frac{q-2}{2q}} \|f\|_2.
      \]
    \end{thm}

We now begin with the ingredients of the
proof of Theorem \ref{thm:qth moment},
following that of Theorem \ref{thm:hypref}.
For $0 \le t \le n$ let
\[ f_t = \sum_S a_S \chi_S^t, \ \ \text{ where }
\chi^t_S = \prod\limits_{i\in S\cap [t]}{\chi^{1/2}_i}
  \prod\limits_{i\in S\setminus [t]}{\mathbf{Z}_i}. \]
Here, just as in Section \ref{sec:hyp},
the function $f_t$ interpolates from the original function
$f_0=f$ to $f_n = \sum_S a_S \chi^{1/2}_S
\in L^2(\{0,1\}^n,\mu_{1/2})$.
As $\{ \chi^t_S: S \sub [n]\}$ are orthonormal
we have $\|f_t\|_2 = \|f\|_2$ for all $t$.

As before, we define the noise operators
$\mathrm{T}^{t}_{\rho',\rho}$ on a function
$f = \sum_S a_S\chi_S^t$ by
\[ \mathrm{T}^t[f]
=\sum_S\rho'^{|S \cap [t]|}\rho^{|S\setminus [t]|}a_S\chi_S^t.\]
Thus $\mathrm{T}^t_{\rho',\rho}$ interpolates from
$\mathrm{T}^0_{\rho',\rho}=\mathrm{T}_{\rho}$ (for the original function) to
$\mathrm{T}^n_{\rho',\rho}=\mathrm{T}_{\rho'}$ (for $\mu_{1/2}$).

Our goal will now be to adjust Lemma \ref{lem:Replacement step}
to the general setting, which is similar in spirit to the 4-norm case,
although somewhat trickier. It turns out that the case $n=1$ poses the
main new difficulties, so we start with this in the next lemma.

\begin{lem} \label{lem:n=1}
Let $q>2$ and
$\mathbf{Z}$ be a random variable satisfying $\mathbb{E}[\mathbf{Z}]=0,
\mathbb{E}[\mathbf{Z}^2]=1, \mathbb{E}[|\mathbf{Z}|^q]\le \sigma^{2-q}.$
Let $e,d\in \mathbb{R}$ and $\rho\in(0,\frac{1}{2q})$.
Then $\|e + \rho d\mathbf{Z}\|_q^q \le \|e + d
\chi^{\frac{1}{2}}\|_q^q+\sigma^{2-q}d^q$.
\end{lem}

\begin{proof}
  If $e = 0$ then the lemma is trivial.
  Therefore we may rescale and assume that $e=1$.
  It will be convenient to consider both sides of the inequality
  as functions of $d$: we write
  \[ f(d) =  \|1+\rho d\mathbf{Z}\|_q^q
   \quad \text{and} \quad
  g(d) = \|1 + d  \chi^{\frac{1}{2}}\|_q^q+\sigma^{2-q}d.
    \]
  As $f(0)=g(0)$, it suffices to show that $f'(0)=g'(0)$
  and $f''\le g''$ everywhere.

  Let us compute the derivatives. We note that the function
  $x\mapsto |x^q|$ has derivative $q|x|^{q-1}\mathrm{sign}(x)$,
  which is in turn continuously differentiable for $q>2$. Thus
  \begin{align*}
f' & =
\mathbb{E}[q\left|1+\rho d\mathbf{Z}\right|^{q-1}\mathrm{sign}(1+\rho d\mathbf{Z})\rho \mathbf{Z}]=
 \rho q\mathbb{E}[|1+\rho d\mathbf{Z}|^{q-1}\mathrm{sign}(1+\rho d\mathbf{Z}) \mathbf{Z}] \ \ \text{ and } \\
 f'' & = (q-1)q\rho^2 \mathbb{E}[|1+\rho d\mathbf{Z}|^{q-2} \mathbf{Z}^2].
\end{align*}
    Differentiating $g$ we obtain
  \begin{align*}
g' & =q\mathbb{E}\Big[\left|1 + d
  \chi^{\frac{1}{2}}\right|^{q-1}\mathrm{sign}(1+d\chi^{\frac{1}{2}})\chi^{\frac{1}{2}}\Big]+q\sigma^{2-q}d^{q-1}
\ \ \text{ and } \\
g'' & = q(q-1) \mathbb{E}\Big[\left| 1+d\chi^{\frac{1}{2}} \right|
      ^{q-2}\left(\chi^{\frac{1}{2}}\right) ^2 \Big]
      +q(q-1)d^{q-2}\sigma^{2-q} \ge
      q(q-1)/2 + q(q-1)d^{q-2}\sigma^{2-q}.
\end{align*}
Thus $g'(0)=f'(0)=0$ and it remains to show $f''\le g''$ everywhere.
Our strategy for bounding $f''$ is to decompose the expectation over
two complementary events $E_1$ and $E_2$,
where $E_1$ is the event that
$|1+ \rho d \mathbf{Z}| \le |d \mathbf{Z}|$
(and $E_2$ is its complementary event).
We write $f''=f''_1 +f''_2$, where each
\[ f''_i = (q-1)q\rho^2 \mathbb{E}[|1+\rho d\mathbf{Z}|^{q-2} \mathbf{Z}^2
       \mathbf{1}_{E_i}].        \]
First we note the bound
\[
f''_1 \le q(q-1) \rho^2 d^{q-2}\mathbb{E}[|\mathbf{Z}|^q]
\le q(q-1) d^{q-2}\sigma^{2-q}.
\]
Given the above lower bound on $g''$,
it remains to show $f''_2\le q(q-1)/2$.
On the event $E_2$ we have
\[
  |d \mathbf{Z}| \le |1+ \rho d \mathbf{Z}|\le 1 +|\rho d \mathbf{Z}|.
\]
Rearranging, we obtain $|\rho d \mathbf{Z}|(\rho^{-1}-1) \le 1.$
Since $\rho^{-1}\geq 2q$, we get
\[
  1+ |\rho d \mathbf{Z}| \le
 1 + \frac{1}{2q-1}.
  \]
    Using $\mathbb{E}[\mathbf{Z}^2]=1$ this yields
    \[
      f''_2
      \le q(q-1)\rho^2 \Big (1+\frac{1}{2q-1}\Big )^{q-2}\le
      e\rho^2q(q-1)\le q(q-1)/2
      .
    \]
    Hence $f''=f''_1 + f''_2 \le g''$ for any value of $d$.
    This completes the proof of the lemma.
    \end{proof}

We are now ready to show the replacement step.

\begin{lem} \label{lem:q-Replacement step}
$\mb{E}[(\mathrm{T}^{t-1}_{2q\rho,\rho} f_{t-1})^q]
\le \mb{E}[(\mathrm{T}^t_{2q\rho,\rho} f_t)^q]
+ \sigma^{2-q} \mb{E}[(\mathrm{T}^t_{2q\rho,\rho} ((\mathrm{D}_t f)_t) )^q]$.
\end{lem}

\begin{proof}
We write
\begin{align*}
f_t & = \chi^{1/2}_t g + h \ \ \text{ and } \ \
 f_{t-1} =  \chi^p_t g + h, \ \ \text{ where } \\
g & = (\mathrm{D}_t f)_t
= \sum_{S: t \in S} \hat{f}(S) \chi^t_{S \sm \{t\}}
= \sum_{S: t \in S} \hat{f}(S) \chi^{t-1}_{S \sm \{t\}}
= (\mathrm{D}_t f)_{t-1}, \ \ \text{ and } \\
h & = \mb{E}_{x_t \sim \mu_{1/2}} f_t
= \sum_{S: t \notin S} \hat{f}(S) \chi^t_S
= \sum_{S: t \notin S} \hat{f}(S) \chi^{t-1}_S
= \mb{E}_{\mathbf{Z}_t} f_{t-1}.
\end{align*}
We also write
\begin{align*}
\mathrm{T}^t_{2q\rho,\rho} f_t & = 2q\rho \chi^{1/2}_t d + e \ \ \text{ and } \ \
\mathrm{T}^{t-1}_{2q\rho,\rho} f_{t-1} = \rho \mathbf{Z}_t d + e, \ \ \text{ where } \\
d & = \mathrm{T}^t_{2q\rho,\rho} g = \mathrm{T}^{t-1}_{2q\rho,\rho} g
\ \ \text{ and } \ \
e = \mathrm{T}^t_{2q\rho,\rho} h = \mathrm{T}^{t-1}_{2q\rho,\rho} h.
\end{align*}
As before, we can calculate the expectations in the statement of the lemma
by conditioning on all coordinates other than $\mathbf{Z}_t$ and
$\chi_t^{\frac{1}{2}}$, so the lemma follows from Lemma
\ref{lem:n=1}, with $2qd$ in place of $d$.
\end{proof}

From now on, everything is similar to Section \ref{sec:hyp}.
We may apply the previous lemma inductively to obtain.

\begin{lem} \label{qhypinduct}
$\| \mathrm{T}^i_{2q\rho,\rho} f_i \|_q^q  \le \sum_{S \sub [n] \sm [i]}
\sigma^{(2-q)|S|}\| \mathrm{T}^n_{2 q\rho,\rho} ((\mathrm{D}_S f)_n) \|_q^q$
 for all $0 \le i \le n$.
\end{lem}

In particular, recalling that
$\mathrm{T}^0_{2 q \rho,\rho}=\mathrm{T}_{\rho}$ on the original function and
$\mathrm{T}^n_{2 q \rho,\rho}=\mathrm{T}_{2 q \rho}$ on $\mu_{1/2}$,
the case $i=0$ of Lemma \ref{qhypinduct} is as follows.

\begin{prop}
\label{prop:q-Reason for hypercontractivity}
$\| \mathrm{T}_{\rho} f \|_q^q  \le \sum_{S \sub [n]}
 \sigma^{(2-q)|S|} \| \mathrm{T}_{2 q \rho} ((\mathrm{D}_S f)_n) \|_q^q$.
\end{prop}

The $q$-norms on the right hand side of
Proposition \ref{prop:q-Reason for hypercontractivity}
are with respect to the uniform measure $\mu_{1/2}$,
where we can apply standard hypercontractivity with noise rate $ \le 1/\sqrt{q-1}$
to obtain \[
  \| \mathrm{T}_{2 q \rho} ((\mathrm{D}_S f)_n) \|_q^q
\le \| (\mathrm{D}_S f)_n \|_2^q = \| \mathrm{D}_S f \|_2^q.
\]

This completes the proof of Theorem \ref{thm:qth moment}.

In the case where the $\mathbf{Z}_i$ have different $q$th
moments, the proof can be adjusted to give a better upper bound.
We write
\begin{equation} \label{eq:Zgen}
\mathbb{E}[\mathbf{Z}_i^q] = \sigma_i^{2-q}, \quad
\sigma_S=\prod_{i\in S}\sigma_i
\ \ \text{ and } \ \
\mathrm{I}_S[f]=\|\frac{1}{\sigma_S}\mathrm{D}_S[f]\|_2^2.
\end{equation}
The proof of Theorem \ref{thm:qth moment}
yields the following variant of Theorem \ref{thm:hypref}.

\begin{thm}\label{thm:qth moment with different norms}
  Let $q\ge 2$, let $\rho\le (2q)^{-1.5}$, and let $f =\sum
  a_S\prod_{i\in S}\mathbf{Z}_i$ with $Z_i$ as in \eqref{eq:Zgen}. Then
  \[
    \|\mathrm{T}_{\rho}f\|_q^q \le \sum
    _{S\subset [n]}\sigma_S^{2-q}\|\mathrm{D}_S[f]\|_2^q.
    \]
  \end{thm}

  The following variant of Theorem \ref{thm:Hypercontractivity}
  follows from Theorem \ref{thm:qth moment with different norms}. The proof is
  similar to the one given in Section \ref{sec:hyp}, where Theorem
  \ref{thm:Hypercontractivity} is deduced from Theorem \ref{thm:hypref}.

  \begin{thm}\label{thm:Hypercontractivity with different norms}
Let $q>2$, $\beta>0$ and $\rho \le (2q)^{-1.5}$.
Suppose $f=\sum_{S \sub [n]}a_S\prod_{i\in S}\mathbf{Z}_i$
 with $Z_i$ as in \eqref{eq:Zgen}
has $\mathrm{I}_S[f] \le \beta \|f\|_2^2$
for all $S\subset    [n]$. Then
    \[
      \| \mathrm{T}_{\rho}f \|_q\le \beta^{\frac{q-2}{2q}}\|f\|_2.
      \]
    \end{thm}

    Finally, we state the following variant of Lemma
    \ref{lem:applying_hypercontractivity}, which is easy to deduce
    from Theorem \ref{thm:Hypercontractivity with different norms}.

    \begin{lem}\label{lem:q-applying_hypercontractivity}
     Let $q>2$ and $\delta>0$.
     Suppose $f=\sum_{S \sub [n]}a_S\prod_{i\in S}\mathbf{Z}_i$
   with $Z_i$ as in \eqref{eq:Zgen}
   has $\mathrm{I}_S[f] \le \delta $ for
    all $|S| \le r$. Then
    \[
      \| f \|_q \le (2q)^{1.5 r}\delta^{\frac{q-2}{2q}}\|f\|_2^{\frac{2}{q}}.
      \]
      \end{lem}

\subsection{A hypercontractive inequality for product spaces}
Now we consider the setting of a general discrete
product space $(\Oo,\nu) = \prod_{t=1}^n (\Oo_t,\nu_t)$.
We assume
$p_t=\min_{\oO_t \in \Oo_t} \nu_t(\oO_t) \in (0,1/2)$
for each $t \in [n]$, and we write $p=\min_t{p_t}$.
We recall the projections $\mathrm{E}_J$ on $L^2(\Oo, \nu)$ defined
by $(\mathrm{E}_J f)(\oO) = \mb{E}_{\oO_J}[ f(\oO) \mid \oO_{\ov{J}} ]$,
the generalised Laplacians $\mathrm{L}_S$
defined by composing $\mathrm{L}_t$ for all $t \in S$,
where $\mathrm{L}_t f = f - \mathrm{E}_t f$,
and the generalised influences
$\mathrm{I}_S(f) = \mb{E}[\mathrm{L}_S(f)^2] \prod_{i \in S} \sS_i^{-2}$,
where $\sS_i^2=p_i(1-p_i)$.

We will require the theory of orthogonal decompositions
in product spaces, which we summarise following the
exposition in \cite[Section 8.3]{o2014analysis}.
For $f\in L^2(\Oo,\nu)$ and $J,S \sub [n]$
we write $f^{\sub J} = \mathrm{E}_{\ov{J}} f $ and define
$f^{=S} = \sum_{J \sub S} (-1)^{|S \sm J|} f^{\sub J}$
(inclusion-exclusion for
$f^{\sub J} = \sum_{S \sub J} f^{=S}$).
This decomposition is known as the Efron--Stein decomposition
\cite{efron1981jackknife}. The key properties of $f^{=S}$ are that it only depends
on coordinates in $S$ and it is orthogonal to any function
that depends only on some set of coordinates not containing $S$;
in particular, $f^{=S}$ and $f^{=S'}$ are orthogonal for $S \ne S'$.
We note that $f = f^{\sub [n]} = \sum_S f^{=S}$.
We have similar Plancherel / Parseval relations
as for Fourier decompositions, namely
$\bgen{f,g} =  \sum_S f^{=S} g^{=S}$,
so $\mb{E}[f^2] = \sum_S (f^{=S})^2$.

Our goal in this section is to prove an hypercontractive inequality
for the Efron--Stein decomposition in the spirit of Theorem
\ref{thm:hypref}. The noise operator is defined by
$\mathrm{T}_\rho[f]=\sum_{S\sub [n]}\rho^{|S|}f^{=S}$. It also has a
combinatorial interpretation, which is similar to the usual one on the
$p$-biased setting. Given $x\in\Omega$, a sample $\mathbf{y}\sim
N_\rho(x)$ is chosen by
independently setting $y_i$ to $x_i$ with probability $\rho$ and
resampling it from $(\Omega_i,\nu_i)$ with probability $1-\rho$.
In the general product space setting there are no good analogs to
$\mathrm{D}_i[f]$ and $\mathrm{D}_S(f)$, and we
instead work with the Laplacians,
which have similar Fourier formulas:
$\mathrm{L}_i[f]=\sum_{S:\,i\in S}f^{=S}$, and $\mathrm{L}_T[f]=\sum
_{S:\,T\subset S}f^{=S}$. In the special case where $\Omega_i=\{0,1\}$
we have $\|\mathrm{L}_S[f]\|_2 =\|\mathrm{D}_S[f]\|_2$. It will be
convenient to write
$\sigma_S=\prod_{i\in S} \sigma_i$.

The main result of this section is the following theorem.
\begin{thm}\label{thm:es}
Let $f\in L^2(\Omega, \nu )$, let $q>2$ be an even integer, and let $\rho \le
\frac{1}{8q^{1.5}}$. Then
  \[
    \| \mathrm{T}_{\rho}f \|^q_q  \le
    \sum_{S\sub [n]} \sigma_S^{2-q}
    \| \mathrm{L}_S[f] \|_2^q.
    \]
  \end{thm}

 The idea of the proof is as follows.
 We encode our function $f\in L^2(\Omega,\nu)$
 as a function $\tilde{f}:=\sum_S\|f^{=S}\|_2\chi_S$
 for appropriate $\chi_S=\prod_{i\in S}\chi_i$
 (in fact, these will be biased characters on the cube).
 We then bound $\| \mathrm{T}_\rho f\|_q $
 by $\| \mathrm{T}_\rho \tilde{f} \|_q$
 and use Theorem \ref{thm:Hypercontractivity with
   different norms} to bound the latter norm.

The main technical component of the theorem
is the following proposition.

\begin{prop} \label{prop:reduction}
  Let $g\in L^2(\Omega, \nu)$ let $\chi_S=\prod_{i\in S}\chi_i$, where
  $\chi_i$ are independent random variables having expectation $0$,
  variance $1$, and
  satisfying $\mathbb{E}[\chi_S^j]\ge \sigma_S^{2-j}$ for each integer
  $j\in \left(2, q \right]$. Let $\tilde{g}=\sum_{S\subset [n]} \|g^{=S}\|_2 \chi_S$.
  Then
  \[
    \|g\|_q\leq \|\tilde{g}\|_q.
    \]
  \end{prop}

  Below, we fix $\chi_S$ as in the proposition, and
  let $\tilde{\circ}$ denote the operator mapping a function
  $g\in L^2(\Omega, \nu)$ to the function
  $\sum_{S\subset[n]}g^{=S}\chi_S$.

To prove the proposition, we will expand out
$\|g\|_q^q$ and $\|\tilde{g}\|_q^q$ according to their definitions
and compare similar terms: namely, we show that a term of the form
$\mathbb{E}[\prod_{i=1}^q g^{=S_i}]$ is bounded by the corresponding term
in $\|\tilde{g}\|_q^q$, i.e. $\prod^q_{i=1}\|g^{=S_i}\|_2\mathbb{E}[\prod_{i=1}^q\chi_{S_i}]$. We now establish such a bound.

We begin with identifying cases in which both terms are equal to $0$,
and for that we use the orthogonality of the decomposition
$\{g^{=S}\}_{S\subset [n]}$. Afterwards, we only rely on the fact
that $g^{=S}$ depends only on the coordinates in $S$.

\begin{lem}\label{lem:upper bound on f=S}
  Let $q$ be some integer,  let $g\in L^2(\Omega, \nu)$, and let
  $S_1,\ldots, S_q \subset [n]$ be some sets.
  Suppose that some $j\in [n]$ belongs to exactly one of the sets $S_1,\ldots, S_q$. Then
  \[
    \mathbb{E}\left[\prod_{i=1}^q g^{=S_i}\right]=0 \quad \text{and} \quad
    \mathbb{E}\left[\prod_{i=1}^q\chi_{S_i}\right]=0.
    \]
  \end{lem}
  \begin{proof}
    Assume without loss of generality that $j\in S_1$.
    The second equality $\mathbb{E}\left[\prod_{i=1}^q\chi_{S_i}\right]=0$
    follows by taking expectation over $\chi_j$,
    using the independence between the random variables $\chi_i$.
    For the first equality, observe that the function
    $\prod^{q}_{i=2}g^{=S_i}$ depends only on coordinates in $S_2\cup
    \cdots, S_q \subset [n]\setminus \{j\}$. Hence the properties of
    the Efron--Stein decomposition imply
    \begin{align*}
      0 = \left\langle g^{=S_1}, \prod_{i=2}^q g^{=S_i}\right\rangle
      =\mathbb{E}\left[\prod_{i=1}^q g^{=S_i}\right]. & \qedhere
    \end{align*}
    \end{proof}

    Thus we only need to consider terms corresponding to $S_1,\ldots,S_q$
    in which each coordinate appears in at least two sets.
    To facilitate our inductive proof we work with general functions
    $f_i$ that depend only on coordinates of $S_i$
    (rather than only with the functions of the form $g^{=S_i}$).
    \begin{lem}\label{lem:tough upper bound}
      Let $f_1,\ldots, f_q\in L^2(\Omega, \nu)$ be functions that
      depend on sets $S_1,\ldots, S_q$ respectively. Let $T_i$ for
      $i=3,\ldots, q$ be the set of coordinates covered by the sets
      $S_1,\ldots, S_q$ exactly $i$ times. Then
      \[
        \left|\mathbb{E}\left[\prod_{i=1}^q f_i \right]\right|\le
        \prod_{i=1}^q \|f_i\|_2
        \cdot
        \prod_{j=3}^q\sigma_{T_j}^{2-j}.
        \]
      \end{lem}
      \begin{proof}
        The proof is by induction on $n$, simultaneously for all functions.
        We start with the case $n=1$, which we prove
        by reducing to the case that all $f_i$ are equal.
        \subsection*{The case $n=1$.}
        Here each $f_i$ either depends on a single input
        or is constant and depends only on the empty set.
        We may assume that none of the $f_i$'s is constant,
        as otherwise we may eliminate it from the inequality
        by dividing by $|f_i|$.
        By the generalised H\"{o}lder inequality we have
        \[
         \left|\mathbb{E}\left[\prod_{i=1}^q f_i \right]\right| \le \prod_{i=1}^q\| f_i \|_q. \label{holder}
         \]
        Hence the case $n=1$ of the lemma will follow once
        we prove it assuming all the $f_i$ are equal.
        \subsection*{The $n=1$ case with equal $f_i$'s}
        We show that if $(\Omega, \nu)$ is a discrete probability
        space in which any atom has probability at least
        $p$, then  $\|f\|_q^q \le
        \|f\|_2^q\sigma^{2-q}$, where $\sigma=\sqrt{p(1-p)}$.

        While the inequality $\|f\|_2\le
        \|f\|_q$ holds in any probability space,
        the reverse inequality holds in any measure space
        where each atom has measure at least $1$.
        Accordingly, we consider the measure $\tilde{\nu}$
        on $\Omega$ defined by $\tilde{\nu}(x)=\nu(x) p^{-1}$. Then
        \[
          \|f\|_{q,\nu}^q=
        p\|f\|_{q,\tilde{\nu}}^q \le
        p\|f\|^q_{2, \tilde{\nu}} =
        p^{1-\frac{q}{2}}\|f\|_{2,\nu}^q \le
        \sigma^{2-q} \|f\|_{2,\nu}^{q}.
      \]
      This completes the proof of the $n=1$ case.
       \subsection*{The inductive step}
       Let $f_1,\ldots,f_q \in L^2(\Omega, \nu)$ be functions. Let
       $\mathbf{x} \sim \prod_{i=1}^{n-1} (\Omega _i,\nu_i)$. By the
       $n=1$ case we have:
       \[
        \left|\mathbb{E}\left[\prod_{i=1}^q f_i\right]\right|
        =\left|\mathbb{E}_{\mathbf{x}}
        \left[\mathbb{E}
        \left[\prod_{i=1}^q (f_i)_{[n-1]\to \mathbf{x}}
        \right]\right]\right|
        \le
        \mathbb{E}_{\mathbf{x}}
        \left[\prod_{i=1}^q\| (f_i)_{[n-1]\to
          \mathbf{x}}\|_2\sigma_n^{j}
          \right],
       \]
       where $j$ is $2-i$ if $n\in T_i$ for $i\geq 3$,
       and otherwise $0$.
       The lemma now follows by applying the inductive
       hypothesis on the functions
       $\mathbf{x}\rightarrow \|(f_i)_{[n-1]\to \mathbf{x}}\|$
       and using
       $\left\| \left\|(f_i)_{[n-1]\to \mathbf{x}} \right\|_2
       \right\|_{2,\mathbf{x}}= \|f_i\|_2$.
        \end{proof}

    \begin{proof}[Proof of Proposition \ref{prop:reduction}]
     We wish to upper bound
     \[
       \mathbb{E}[g^q]=\sum_{S_1,\ldots,S_q}\mathbb{E}\left[\prod_{i=1}^q
       g^{=S_i}\right]
     \]
     by
     \[
       \sum_{S_1,\ldots,S_q}\mathbb{E}\left[\prod_{i=1}^q
       \chi_{S_i}\right]\prod_{i=1}^q \|g^{=S_i}\|_2.
     \]
     We upper bound each term participating in the expansion of $g^q$
     by the corresponding term in $\tilde{g}^q$.
     In the case the sets $S_i$ cover some element exactly once,
     Lemma \ref{lem:upper bound on f=S} implies that both terms are
     $0$. Otherwise, the sets $S_i$ cover each element either $0$
     times or at least $2$ times; let $T_i$ be the set of elements of $S_1,\ldots,S_q$
     appearing in exactly $i$ of the sets (as in Lemma \ref{lem:tough upper bound}).
     By the  assumption of the proposition, we have $\mathbb{E}\left[\prod_{i=1}^q
     \chi_{S_i}\right]\ge \prod_{i=3}^q\sigma_{T_i}^{2-|T_i|}$.
     The proof is concluded by combining this with the upper bound
     on $\mathbb{E}\left[\prod_{i=1}^q g^{={S_i}}\right]$
     following from Lemma \ref{lem:tough upper bound} with $f_i = g^{=S_i}$.
   \end{proof}

   \begin{proof}[Proof of Theorem \ref{thm:es}]
 Let $\sigma_i'=\sqrt{p_i/4(1-p_i/4)}$.
 We choose $\chi_i$ to be the
 $\frac{p_i}{4}$-biased character,
 $\chi_i = \frac{x_i-p_i/4}{\sigma_i'}$.
 Clearly $\chi_i$ has mean $0$ and variance
 $1$, and a direct computation shows that
 $\mathbb{E}\left[\chi_i^j\right]
 \geq (\sigma_i)^{2-j}$ for all integer $j> 2$,
 hence all of the conditions of Proposition
 \ref{prop:reduction} hold.

 Denote $\sigma'_S=\prod_{i\in S}\sigma'_i$ and set $h=T_{\frac{1}{4}}f$, $g=\mathrm{T}_{\frac{1}{2q^{1.5}}} h$.
By Proposition \ref{prop:reduction} and Theorem
\ref{thm:qth moment with different norms} we have

\[
\| \mathrm{T}_{\frac{1}{8q^{1.5}  }   } f \|_q^q = \| g \|_q^q
\le \|\tilde{g} \|_q^q \le \sum_S (\sigma'_S)^{2-q} \|\mathrm{D}_S[\tilde{h}]\|_2.
\]
We note that by Parseval, the $2$-norm of $\tilde{h}$ and its derivatives are
equal to the $2$-norm of $h$ and its Laplacians,
and thus the last sum is equal to
\[
\sum_{S}(\sigma'_S)^{2-q} \| \mathrm{L}_S[h]\|_2^q \le
\sum_{S}(\sigma_S)^{2-q} \| \mathrm{L}_S[f]\|_2^q.
\]
In the last inequality we used $\sigma'_S\geq 2^{-|S|} \sigma_S$ and
$\| \mathrm{L}_S[h]\|^{q}\leq 2^{-q|S|} \| \mathrm{L}_S[f]\|_2^q$
(which follows from Parseval).
This completes the proof of the theorem.
\end{proof}

\section{An invariance principle for global functions}
\label{sec:inv}

Invariance (also known as Universality) is a fundamental paradigm
in Probability, describing the phenomenon that many random processes
converge to a specific distribution
that is the same for many different instances of the process.
The prototypical example is the Berry-Esseen Theorem,
giving a quantitative version of the Central Limit Theorem
(see e.g.\ \cite[Section 11.5]{o2014analysis}). More sophisticated instances of
the phenomenon that have been particularly influential on recent research
in several areas of Mathematics include the universality of
Wigner's semicircle law for random matrices (see \cite{mehta2004book})
and of Schramm--Loewner evolution (SLE)
e.g.\ in critical percolation (see \cite{smirnov2006survey}).

In the context of the cube, the Invariance Principle is a powerful tool
developed by Mossel, O'Donnell and Oleszkiewicz \cite{mossel2010noise}
while proving their `Majority is Stablest' Theorem,
which can be viewed as an isoperimetric theorem for the noise operator.
Roughly speaking, the result (in a more general form due to
Mossel \cite{mossel2010gaussian}) is that `majority functions'
(characteristic functions of Hamming balls) minimise noise sensitivity
among functions that are `far from being dictators'.
The Invariance Principle converts many problems on the cube
to equivalent problems in Gaussian Space;
in particular, `Majority is Stablest' is converted
into an isoperimetric problem in Gaussian Space
which was solved by a classical theorem
of Borell \cite{borell1985geometric}
(half-spaces are isoperimetric).

In this section we will establish an invariance principle
for global functions that has several applications
analogous to those of the classical invariance principle,
such as the following variant of `majority is stablest'.
We define the \emph{$p$-biased $\alpha$-Hamming ball} on $\{0,1\}^n$
as the function $H_{\alpha}$ whose value is $1$ on an input $x$
if and only if $x$ has at least $t$ coordinates equal to $1$,
and $t$ is chosen so that $\mu_p(H_{\alpha})$ is as close to $\alpha$ as possible.
\begin{cor}\label{majority is stablest}
  For each $\epsilon>0$, there exists
  $\delta>0$, such that the following holds. Let $\rho\in (\epsilon,
  1-\epsilon)$, let $n>\delta^{-1}$, and let
  $f, g\in L^2(\{0,1\}^n,\mu_p)$. Suppose that $\mathrm{I}_S[f]\le
  \delta$ and that $\mathrm{I}_S[g]\le \delta$
  for each set $S$ of at most $\delta^{-1}$ coordinates. Then
  \[
    \left\langle \mathrm{T}_\rho{f},g \right\rangle \le \left
      \langle \mathrm{T}_\rho H_{\mu_p(f)},
      H_{\mu_p(g)}\right\rangle +\epsilon.
    \]
  \end{cor}
We omit the proof of this result as it is very similar
to that in \cite{mossel2010gaussian}.
For the sake of brevity we also omit discussion
of other applications of our invariance principle,
including a sharp threshold for almost monotone Boolean functions,
which is analogous to results of Lifshitz \cite{lifshitz2018hypergraph}.

In the basic form (see \cite[Section 11.6]{o2014analysis})
of the Invariance Principle,
we consider a multilinear real-valued polynomial $f$ of degree $\le k$
and wish to compare $f(\xb)$ to $f(\yb)$,
where $\xb$ and $\yb$ are random vectors
each having independent coordinates,
according a smooth (to third order) test function $\phi$.
(Comparison of the cumulative distributions requires $\phi$ to be
a step function, but this can be handled by smooth approximation.)
The version of \cite[Remark 11.66]{o2014analysis}
 shows that if the coordinates $x_i$
have mean $0$, variance $1$ and are suitably hypercontractive
(satisfy $\|a + \rho b x_i\|_3 \le \|a + bx_i\|_2$
for any $a,b \in \mb{R}$), and similarly for $y_i$, then
\begin{equation} \label{eq:rem}
\big| \mb{E}[\phi(f(\xb))] - \mb{E}[\phi(f(\yb))] \big|
\le \tfrac{1}{3} \|\phi'''\|_\infty \rho^{-3k}
\sum_{i \in [n]} \mathrm{I}_i(f)^{3/2}.
\end{equation}

The hypercontractivity assumption applies e.g.\
if the coordinates are standard Gaussians or $p$-biased bits
(renormalised to have mean $0$ and variance $1$) with $p$ bounded away
from $0$ or $1$, but if $p=o(1)$ then we need $\rho=o(1)$,
in which case their theorem becomes ineffective.
We will apply our hypercontractivity inequality to obtain
an invariance principle that is effective for small
probabilities and functions with small generalised influences.
We adopt the following setup.

\begin{setup} \label{set:inv}
Let $\sigma_1,\ldots,\sigma_n>0$, let
$\mathbf{X}=(\mathbf{X}_1,\dots, \mathbf{X}_n)$ and
$\mathbf{Y}=(\mathbf{Y}_1,\ldots, \mathbf{Y}_n)$ be random vectors
with independent coordinates, where each
$X_i$ and $Y_i$ are real-valued random variable with mean $0$, variance $1$,
and satisfy $\|X_i\|_3^3 \le \sigma_i^{-1}$ and $\|Y_i\|_3^3\le \sigma_i^{-1}$.
Let $f \in \mb{R}[v]$ be a multilinear polynomial
of degree $d$ in $n$ variables $v=(v_1,\dots,v_n)$.
Let $\phi:\mb{R} \to \mb{R}$ be smooth.
\end{setup}

For $S \sub [n]$ we write $\hat{f}(S)$ for the
coefficient in $f$ of $v_S = \prod_{i \in S} v_i$.
We write $W_{S}(f)= \sum_{J: S \sub J} \hat{f}(J)^2$
and similarly to Section \ref{subsec:multilinear} we define the
generalised influences by
$\mathrm{I}_S(f) = W_{S}(f) \prod_{i \in S} \sS_{i}^{-2}$.

We write $\mathrm{T}_\rho[f]=\sum_{S\sub [n]}\rho^{|S|}\hat{f}(S)v_S$.

Now we state our invariance principle,
which compares $f(\mathbf{X})$ to $f(\mathbf{Y})$.

\begin{thm}\label{thm:inv}
Under Setup \ref{set:inv}, if $\mathrm{I}_S[f]\le \epsilon$ for each
nonempty set $S$, then
\[
  \left| \mb{E}[\phi(f(\mathbf{X}))] - \mb{E}[\phi(f(\mathbf{Y}))]
  \right|
  \le 2^{5d} \|\phi'''\|_\infty W_{\emptyset}(f) \sqrt{\epsilon}.
\]
\end{thm}

The term $W_{\emptyset}(f)$ can be replaced by either
$\mathbb{E}[f(\mathbf{X})^2]$ or
$\mathbb{E}[f(\mathbf{Y})^2]$ as they are all equal.

\skipi

Theorem \ref{thm:inv} can be informally interpreted as saying that if
a multilinear, low degree polynomial $f$ is global, then the distribution
of $f(\mathbf{X})$ does not really depend on the distribution of $\mathbf{X}$
except for the mean and variance of each coordinate.

In particular, it implies that plugging in the $p$-biased characters into
$f$ results in a fairly similar distribution to the one resulting from plugging in
the uniform characters into $f$. A posteriori, this may be seen as an intuitive explanation
for Theorem \ref{thm:Hypercontractivity}, as the standard hypercontractivity
theorem holds in the uniform cube.

\skipi
Next, we set up some notations and preliminary observations for the proof of Theorem
\ref{thm:inv}.
Throughout we fix
$\mathbf{X}$, $\mathbf{Y}$, $f$, and $\phi$ as in Setup \ref{set:inv}. We write
$\mathbf{X}_S=\prod_{i\in S}\mathbf{X}_i$, and similarly for $\mathbf{Y}$.
Recall that $f = \sum_S \hat{f}(S) v_S$ is a (formal) multilinear polynomial in $\mb{R}[v]$ of degree $d$.
Note that $f(\mathbf{X}) = \sum_S \hat{f}(S)\mathbf{X}_S$
has $\mb{E}[f(\mathbf{X})^2] = \sum_S \hat{f}(S)^2$,
as $\mb{E}\mathbf{X}_S^2 = 1$
and $\mb{E}[\mathbf{X}_S \mathbf{X}_T] = 0$ for $S \ne T$.
The random variable $f(\mathbf{X})$ has the orthogonal decomposition
$f = \sum_S f^{=S}$ with each $f^{=S} = \hat{f}(S)\mathbf{X}_S$.
Further note that $\mathrm{L}_S f(\mathbf{X}) = \sum_{J: S \sub J} \hat{f}(J)\mathbf{X}_J$
so we have the identities
\[
  \mathrm{I}_S(f) \prod_{i \in S} \sS_i^2 =
  \mb{E}[(\mathrm{L}_S f(\mathbf{X}))^2]=
  \mb{E}[(\mathrm{L}_S f(\mathbf{Y}))^2]=
  \sum_{J: S \sub J} \hat{f}(J)^2 = W^{S^\ua}(f).
  \]

We apply the replacement method
as in Section \ref{sec:hyp} (and as in the proof
of the original invariance principle by
Mossel, O'Donnell and Oleszkiewicz \cite{mossel2010noise}).
For $0 \le t \le n$, define
$\mathbf{Z}^{:t} = (\mathbf{Z}^{:t}_1,\dots,\mathbf{Z}^{:t}_n)
= (\mathbf{Y}_1,...,\mathbf{Y}_t,\mathbf{X}_{t+1},...,\mathbf{X}_n)$,
and note that $f(\mathbf{Z}^{:t})$ has the orthogonal decomposition
$f(\mathbf{Z}^{:t}) = \sum_S f(\mathbf{Z}^{:t})^{=S}$ with
\[
  f(\mathbf{Z}^{:t})^{=S} = \hat{f}(S) \mathbf{Z}_S = \hat{f}(S)
  \mathbf{Y}_{S\cap [t]} \mathbf{X}_{S\setminus [t]}.
  \]

\begin{proof}[Proof of Theorem \ref{thm:inv}]
We adapt the exposition in \cite[Section 11.6]{o2014analysis}.
As $\mathbf{Z}^{:0}=\mathbf{X}$ and $\mathbf{Z}^{:n}=\mathbf{Y}$ we have
by telescoping and the triangle inequality
\[
  |\mb{E}[\phi(f(\mathbf{X}))] - \mb{E}[\phi(f(\mathbf{Y}))] |
\le \sum_{t=1}^n
| \mb{E}[\phi(f(\mathbf{Z}^{:t-1}))] - \mb{E}[\phi(f(\mathbf{Z}^{:t}))] |.
\]
Consider any $t \in [n]$ and write
\[ f(\mathbf{Z}^{:t-1}) = U_t + \DD_t \mathbf{Y}_t
\ \ \text{ and } \ \
f(\mathbf{Z}^{:t}) = U_t + \DD_t \mathbf{X}_t,
\ \ \text{ where } \]
\[ U_t = \mathrm{E}_t f(\mathbf{Z}^{:t-1}) = \mathrm{E}_t f(\mathbf{Z}^{:t})
\ \ \text{ and } \ \
\DD_t = \mathrm{D}_t f(\mathbf{Z}^{:t-1}) = \mathrm{D}_t f(\mathbf{Z}^{:t}).
\]
Both of the functions $U_t$ and $\DD_t$ are independent of the random
variables $X_t$ and $Y_t$.

By Taylor's Theorem,
\begin{align*}
\phi(f(\mathbf{Z}^{:t-1})) & = \phi(U_t) + \phi'(U_t) \DD_t \mathbf{Y}_t
+ \tfrac{1}{2} \phi''(U_t) (\DD_t \mathbf{Y}_t)^2
+ \tfrac{1}{6} \phi'''(A) (\DD_t \mathbf{Y}_t)^3,
\ \ \text{ and }  \\
\phi(f(\mathbf{Z}^{:t})) & = \phi(U_t) + \phi'(U_t) \DD_t \mathbf{X}_t
+ \tfrac{1}{2} \phi''(U_t) (\DD_t \mathbf{X}_t)^2
+ \tfrac{1}{6} \phi'''(A') (\DD_t \mathbf{X}_t)^3,
\end{align*}
for some random variables $A$ and $A'$.
As $\mathbf{X}_t$ and $\mathbf{Y}_t$ have mean $0$ and variance $1$ we have
$0 = \mb{E}[\phi'(U_t) \DD_t \mathbf{Y}_t]
= \mb{E}[\phi'(U_t) \DD_t \mathbf{X}_t]$ and
$ \mb{E}[ \phi''(U_t) (\DD_t)^2 ]
= \mb{E}[ \phi''(U_t) (\DD_t \mathbf{Y}_t)^2 ]
= \mb{E}[ \phi''(U_t) (\DD_t \mathbf{X}_t)^2 ]$, so
\[
  | \mb{E}[\phi(f(\mathbf{Z}^{:t-1}))] - \mb{E}[\phi(f(\mathbf{Z}^{:t}))] |
\le \tfrac{1}{6} \|\phi'''\|_\infty
( \mb{E}[|\DD_t \mathbf{X}_t|^3] + \mb{E}[|\DD_t \mathbf{Z}_t|^3] )
\le \tfrac{1}{3} \|\phi'''\|_\infty\sigma_t^{-1} \|\DD_t \|_3^3.
\]

The function $\DD_t$ is the function $\mathrm{D}_t[f]$ applied on random
variables satisfying the hypothesis of Lemma \ref{lem:q-applying_hypercontractivity}. Moreover, $\mathrm{I}_S[\mathrm{D}_{t}[f]]$ is
either 0 when $t\in S$, or $\sigma_t^{2}\mathrm{I}_{S\cup\{t\}}[f]$ when $t\notin
S$, in which case $\mathrm{I}_S[f]\le \sigma_t^{2}\epsilon$. Hence, by
Lemma \ref{lem:q-applying_hypercontractivity} (with $q=3$), we obtain
\[
  \|\DD_t \|_3^3\le
6^{4.5 d}\sigma_t \sqrt{\epsilon}
\|\DD_t \|_2^2=
6^{4.5 d}\sigma_t \sqrt{\epsilon} \cdot \sum_{S\ni t}\hat{f}(S)^2.
\]

Hence,
\[
  \sum_{t=0}^n\tfrac{1}{3} \|\phi'''\|_\infty\sigma_t^{-1} \|\DD_t
  \|_3^3
  \le 6^{4.5 d}\sqrt{\epsilon} \tfrac{1}{3} \|\phi'''\|_\infty
  \sum_{S}|S|\hat{f}(S)^2
  \le 6^{4.5 d}\sqrt{\epsilon}\tfrac{d}{3} \|\phi'''\|_\infty
  W_{\emptyset}(f).
  \]
  This completes the proof of the theorem since
  $6^{4.5d}\frac{d}{3}\le 2^{12 d}$.
\end{proof}

\newpage

\part{Sharp thresholds} \label{part:sharp}

In this part we apply the hypercontractivity results
of the previous part to obtain new results
in the theory of sharp thresholds.
To prepare for the analysis, we start
in Section \ref{sec:equiv} by establishing
the equivalence between the two notions of globalness
introduced earlier,
namely control of generalised influences
and insensitivity of the measure
under restriction to a small set of coordinates.
Section \ref{sec:noise+sharp} concerns the total influence
of global functions, and includes the proofs of
our stability results for the isoperimetric inequality
(Theorems \ref{thm:Bourgain+} and \ref{thm:Variant of Kahn kalai})
and our first sharp threshold result
(Theorem \ref{thm:Sharp threshold result}).
In Section \ref{sec:noise} we prove our result
on noise sensitivity and apply this to deduce
an alternative sharp threshold result.

\section{Characterising global functions} \label{sec:equiv}

Above we have introduced two notions of what it means for
a Boolean function $f$ to be global.
The first globalness condition,
which appears e.g.\ in Theorem \ref{thm:Bourgain+},
is that the measure of $f$ is not sensitive
to restrictions to small sets of coordinates.
The second condition is a bound on generalised
influences $I_S(f)$ for small sets $S$.
In this section we show that we can move
freely between these notions for two classes
of Boolean functions:
namely sparse ones and monotone ones.

Throughout we assume $p \le 1/2$, which does not
involve any loss in generality in our main results;
indeed, if $p>1/2$ we can consider
the dual $f^*(x) = 1-f(1-x)$ of any Boolean function $f$,
for which $\mu_{1-p}(f^*) = 1-\mu_p(f)$
and $\mathrm{I}_{\mu_{1-p}}(f^*) = \mathrm{I}_{\mu_p}(f)$.

We start by formalising our first notion of globalness.

\begin{defn} \label{def:global}
 We say that a Boolean function $f$
 is \emph{$\left(r,\delta\right)$-global} if
 $\mu_p\left(f_{J\to 1}\right)\le \mu_p\left(f\right)+\delta$
 for each set $J$ of size at most $r$.
\end{defn}

We remark that Definition \ref{def:global}
is a rather weak notion of globalness,
so it is quite surprising that it suffices
for Theorems \ref{thm:Variant of Kahn kalai}
and \ref{thm:Noise sensitivity},
where one might have expected to need
the stricter notion that
$\mu_p(f_{J\to 1})$ is close to $\mu_p(f)$.

The following lemma shows that if a sparse Boolean function
is global in the sense of Definition \ref{def:global}
then it has small generalised influences.

\begin{lem}
\label{lem: gen influence control for global functions}
Suppose that $f: \{0,1\}^n \to \{0,1\}$
is an  $\left(r,\delta\right)$-global Boolean function
with $\mu_p(f) \le \delta$.
Then $\mathrm{I}_{S}\left(f^{\le r}\right) \le
\mathrm{I}_{S}\left(f\right) \le 8^{r} \delta$
for all $S\sub [n]$ with $|S| \le r$.
\end{lem}

\begin{proof}
The first inequality is from Lemma \ref{obs}. Next, we estimate
\begin{align} \label{eq:s}
\sqrt{\mathrm{I}_{S}\left(f\right)}  =
\left \| \sum_{x \in \{0,1\}^S}\left(-1\right)^{\left|S\right | -
  |x|}f_{S\to x} \right \|_2 \le \sum_{x\in \{0,1\}^S}\left\|f_{S\to
  x}\right \|
  _2=\sum_{x\in\{0,1\}^S}\sqrt{\mu_p(f_{S\to x})}.
\end{align}
Next we fix $x \in \{0,1\}^S$ and claim that
$\mu_p(f_{S\to x}) \le 2^r\delta$.
By substituting this bound in \eqref{eq:s}
we see that this suffices to complete the proof.
Let $T$ be the set of all $i\in S$ such that $x_i=1$.
Since $f$ is nonnegative, we have
$\mu_p(f_{T\to 1})\ge \left(1-p\right)^{\left|S\backslash
    T\right|}\mu_p(f_{S\to x})$.
As $f$ is $\left(r,\delta\right)$-global
and $\mu_p(f) \le \delta$, we have
$\mu_p\left(f_{T\to1}\right)\le 2\delta$,
so $\mu_p(f_{S\to x})\le (1-p) ^{|T|-r}2\delta \le 2^r\delta$,
where for the last inequality we can assume $T \ne \es$,
as $\mu_p\left(f_{T\to 1}\right) = \mu_p(f) \leq\delta \le 2^r\delta$.
This completes the proof.
\end{proof}

Next we show an analogue of the previous lemma
replacing the assumption that $f$ is sparse
by the assumption that $f$ is monotone.

\begin{lem} \label{rem}
Let $f\colon\{0,1\}^n\to\{0,1\}$
be a monotone Boolean $\left(r,\delta\right)$-global function.
Then $\mathrm{I}_S\left(f\right) \le 8^{r}\delta$
for every nonempty $S$ of size at most $r$.
\end{lem}

The proof is based on the following lemma
showing that globalness is inherited
(with weaker parameters) under restriction of a coordinate.

\begin{lem} \label{restrict}
    Suppose that $f$ is a monotone
    $\left(r,\delta\right)$-global function. Then for each  $i$:
    \begin{enumerate}
    \item{$f_{i\to 1}$ is $\left(r-1,\delta\right)$-global,} 
    \item{$\mu_p\left(f_{i\to 0}\right) \ge
        \mu_p\left(f\right)-\frac{p\delta}{1-p}$,}
    \item{$f_{i\to 0}$ is $\left( r-1, \frac{\delta}{1-p}\right)$-global}.
     \end{enumerate}
   \end{lem}

    \begin{proof}
 To see (1), note that for any $J$ with $|J| \le r-1$ we have
 $\mu_p((f_{i \to 1})_{J \to 1}) = \mu_p(f_{J \cup \{i\} \to 1})
 \le \mu_p(f) + \delta \le \mu_p(f_{i \to 1}) + \delta$,
 where the last inequality holds as $f$ is monotone.
 Statement (2) follows from the upper bound $\mu_p\left(f_{i\to
        1}\right)\le \mu_p\left(f\right)+\delta$
      and    $\mu_p\left(f_{i\to
          0}\right)=\frac{\mu_p\left(f\right)-p\mu_p\left(f_{i\to1}\right)}{\left(1-p\right)}.$

   For (3), we note that by monotonicity
   $\mu_p\left(\left(f_{i\to
          0}\right)_{S\to 1}\right)
          \le \mu_p\left(f_{\{i\}\cup S\to
        1}\right).$
        As $f$ is $\left(r,\delta\right)$-global,
    \[ \mu_p\left(f_{S\cup\{i\}\to 1}\right)
    \le \mu_p\left(f\right)+\delta
    \le \mu_p\left(f_{i\to 0}\right)+\delta+\frac{p\delta}{1-p}=\mu_p\left(f_{i\to 0}\right)+\frac{\delta}{1-p}, \]
    using (2). Hence, $f_{i\to 0}$ is
    $\left(r,\frac{\delta}{1-p}\right)$-global.
  \end{proof}

\begin{proof}[Proof of Lemma \ref{rem}]
We argue by induction on $r$. In the case where  $r=1$,
Lemma \ref{restrict} and monotonicity of $f$
imply (using $p \le 1/2$)
 \[\mathrm{I}_{i}\left(f\right)=\mu_p\left(f_{i\to 1}\right) -
   \mu_p\left(f_{i\to 0}\right)\le \delta
   +\frac{p\delta}{1-p}\le 2\delta. \]
Now we bound $\mathrm{I}_{S\cup\left\{i\right\}}\left(f\right)$
for $r>1$ and $S$ of size $r-1$ with $i \notin S$.

Note that $\mathrm{D}_{S\cup\{i\}}\left(f\right)=\mathrm{D}_{S}
\left[\mathrm{D}_i (f)\right]$. By the triangle inequality,
we have
 \[
   \sqrt{\mathrm{I}_{S\cup\left\{i\right\}} \left(f\right)} =
   \sigma^{-r}\|\mathrm{D}_{S\cup\left\{i\right\}}(f)\|_2 =
   \sigma^{1-r}\|\mathrm{D}_{S}(f_{i\to 1})-\mathrm{D}_S(f_{i\to 0})\|_2
   \le
   \sqrt{ \mathrm{I}_S\left(
       f_{i\to1} \right) } +
   \sqrt{ \mathrm{I}_S{\left(f_{i\to
       0} \right)}}.
\]
By the induction hypothesis and Lemma \ref{restrict}
 the right hand side is at most
\[
  \sqrt{8^{r-1}\delta}
  +\sqrt{8^{r-1}2\delta}
  \le \sqrt{8^r\delta}.
\]
Taking squares, we obtain  $\mathrm{I}_{S\cup\left\{i\right\}}\left(f\right)\le
8^r\delta.$
\end{proof}

We conclude this section by showing the converse direction of
the equivalence between our two notions of globalness,
i.e.\ that if the generalised influences of a function $f$
are small then $f$ is global in the sense of its measure
being insensitive to restrictions to small sets.
(We will not use the lemma in the sequel
but include the proof for completeness.)

\begin{lem}
  Let $f\colon\{0,1\}^n\to \{0,1\}$ be a Boolean function and let $r>0$. Suppose
  that $\mathrm{I}_S[f] \le \delta$ for each nonempty set $S$ of at most
  $r$ coordinates. Then $f$ is $\left(r,4^{r}\delta\right)$-global.
\end{lem}
\begin{proof}
To facilitate a proof by induction on $r$
we prove the slightly  stronger statement that
$f$ is $(r,\sum_{i=1}^{r}4^{i-1}\delta)$-global.
Suppose first that $r=1$. Our
goal is to show that if $\mathrm{I}_i[f]<\delta$, then
$\mu_p(f_{i\to1})-\mu_p(f_{i\to 0})<\delta$, and indeed,
\[
  \mu_p(f_{i\to1})-\mu_p(f_{i\to 0})
  \le \Pr[f_{i\to 1}\ne f_{i\to
    0}]= \|f_{i\to 1}-f_{i\to 0}\|_2^2
    =\|\mathrm{D}_i[f]\|_2^2=\mathrm{I}_{i}[f] < \delta.
\]
Now suppose that $r>1$ and that the lemma holds with $r-1$ in place of
$r$. The lemma will follow once we show that for all $i$ and all
nonempty sets
$S$ of size at most $r-1$, we have $\mathrm{I}_{S}[f_{i\to 1}]\le 4\delta$.
Indeed, the induction hypothesis and the $n=1$ case
will imply that for each set $S$ of size at most $r$ and each $i\in S$
we have $\mu_p(f_{S\to 1}) \le \mu_p(f_{i\to 1}) +
\sum_{j=1}^{r-1}4^{j-1}\cdot 4\delta \le \mu_p(f) +
\sum_{j=1}^{r}4^{j-1}\delta$.

We now turn to showing the desired upper bound on the generalised
influences of  $f_{i\to 1}$. Let $S$ be a set of size at
most $r-1$. Recall that $\mathrm{I}_S[f_{i\to 1}]=\|
\mathrm{D}_{S}[f_{i\to 1}] \|_2^2$. We may assume that $i\notin S$ for
otherwise the generalised influence $\mathrm{I}_S[f_{i\to 1}]$ is
0. We make two observations. Firstly, we have
\[
  \mathrm{D}_{S\cup \{i\}}[f]=\mathrm{D}_{S}[f_{i\to 1}]
  -\mathrm{D}_S[f_{i\to 0}].
\]
Secondly, conditioning on the ouput of the coordinate $i$ we have
\[
  \|\mathrm{D}_S[f]\|_2^2=p\| \mathrm{D}_S[f_{i\to
    1}]\|_2^2 + (1-p)\|\mathrm{D}_S[f_{i\to 0}]\|_2^2,
\]
which implies
$\|\mathrm{D}_S[f_{i\to 0}]\|_2 \le \sqrt{2}\|\mathrm{D}_S[f]\|_2$. We
may now apply the triangle inequality on the first
observation and use the second observation to obtain
\[
  \sqrt{\mathrm{I}_S[f]} = \| \mathrm{D}_S[f_{i\to 1}]\|_2 \le
  \|\mathrm{D}_{S\cup \{ i\}}[f]\|_2 + \|\mathrm{D}_{S}[f_{i\to
    0}]\|_2\le \sqrt{\delta} + \sqrt{2} \|\mathrm{D}_S[f]\|_2\le 2\sqrt{\delta}.
\]

Taking squares, we obtain the desired upper bound on the generalised
influences of $f_{i\to 1}$.
\end{proof}

\section{Total influence of global functions}
\label{sec:noise+sharp}

In this section we show that our hypercontractive inequality
(Theorem \ref{thm:Hypercontractivity})
implies our stability results for the isoperimetric inequality, namely
Theorems \ref{thm:Bourgain+} and \ref{thm:Variant of Kahn kalai}. We also deduce our first sharp threshold result,  Theorem
\ref{thm:Sharp threshold result}.

\subsection{The spectrum of sparse global sets}\label{spectrum}

The key step in the proofs of Theorems \ref{thm:Variant of Kahn
  kalai} and \ref{thm:Noise sensitivity} is to show that the Fourier
spectrum of global sparse subsets of the $p$-biased cube
is concentrated on the high
degrees. We recall first a proof that in the uniform cube
(i.e.\ cube with uniform measure), \emph{all} sparse sets
have this behaviour (not just the global ones).
Our proof is based on ideas from Talagrand
\cite{talagrand1994approximate} and Bourgain and Kalai
\cite{BourgainKalai19}.

\begin{thm} \label{thm:warm-up}
Let $f$ be a Boolean function on the uniform cube, and let $r>0$. Then
\[
  \left\|f^{\le r} \right\|_2^2\le 3^r \mu_{1/2}\left(f\right)^{1.5}.
  \]
\end{thm}

  The idea of the proof is to  bound $\left\|f^{\le r}
  \right\|_2^2=\left\langle f^{\le r},f\right\rangle$ via H\"older
  by $\left\|f^{\le r}\right\|_4 \left\|f\right\|_{4/3}$,
  bound the $4$-norm via hypercontractivity
  and express the $4/3$-norm in terms of the measure of $f$
  using the assumption that $f$ is Boolean.
  For future reference, we decompose the argument into two lemmas,
  the first of which applies also to the $p$-biased settting
  and the second of which requires hypercontractivity,
  and so is specific to the uniform setting.
  Theorem \ref{thm:warm-up} follows immediately from
  Lemmas \ref{lem:nt} and \ref{lem:upper bound on 4th moment} below.

In the following lemma we consider $\{-1,0,1\}$-valued functions
so that it can be applied to either a Boolean function
or its discrete derivative.

  \begin{lem}\label{lem:nt}
Let $f\colon\power{n}\to\{0,1,-1\}$, let $\mc{F}$ be a family of
subsets of $\left[n\right]$,
and let $g(x) = f^{\mc{F}} = \sum_{S\in\mc{F}}{\hat{f}(S)\chi_S(x)}$.
Then $\|g\|_2^2 \le \|g\|_4 \|f\|_2^{1.5}$, where the norms
can be taken with respect to an arbitrary $p$-biased measure.
\end{lem}

\begin{proof}
By Plancherel and H{\"o}lder's inequality,
$\mb{E}[g^2] = \inner{f}{g}
\le \norm{f}_{4/3}\norm{g}_4$,
where $\norm{f}_{4/3} = \mb{E}[f^2]^{3/4}=\|f\|_2^{1.5}$
as $f$ is $\{-1,0,1\}$-valued.
\end{proof}

  Applying Lemma \ref{lem:nt} with $g=f^{\le r}$, we obtain a lower bound on
  the 4-norm of $g$. We now upper bound it by appealing to the
  Hypercontractivity Theorem.
  \begin{lem}\label{lem:upper bound on 4th moment}
    Let $g$ be a function of degree $r$ on the uniform cube. Then
    $\left\|g\right\|_4 \le \sqrt{3}^{r}\left\|g\right\|_2$.
  \end{lem}
  \begin{proof}
    Let $h$ be the function, such that $\mathrm{T}_{1/\sqrt{3}}h=g$,
    i.e.\ $h=\sum_{|S|\le r}\sqrt{3}^{|S|}\hat{g}\left(
      S\right)\chi_S$. Then the Hypercontractivity
    Theorem implies that $\|g\|_4\le \|h\|_2$, and by Parseval
    $\|h\|_2\le \sqrt{3}^r\|g\|_2$.
    \end{proof}

We shall now adapt the proof of Theorem \ref{thm:warm-up} to global
functions on the $p$-biased cube. The only part in the above proof
that needs an adjustment is Lemma \ref{lem:upper bound on 4th
  moment}, and in fact
  we have already provided the required adjustment in Section
\ref{sec:hyp} in the form of Lemma \ref{lem:applying_hypercontractivity}.

\begin{thm} \label{lem:normsense0}
Let $r \ge 1$,
and let  $f\colon\power{n}\to\{0,1,-1\}$. Suppose that
$\mathrm{I}_S[f^{\le r}]\le \delta$ for each set $S$ of size at most $r$.
Then $\mb{E}[(f^{\le r})^2]
\le 5^r \delta^{\frac{1}{3}}  \mb{E}\left[f^2\right]$.
\end{thm}

\begin{proof}
Applying Lemma \ref{lem:applying_hypercontractivity} with $g=f^{\le r}$, we
obtain the upper bound $\|g\|_4\le 5^{\frac{3r}{4}}\delta^{\frac{1}{4}}\|g\|_2^{0.5}$. Since
the function $f$ takes values only in the set $\{0,1,-1\}$,
we may apply Lemma \ref{lem:nt}. Combining it with the upper bound on
the 4-norm of $g$, we obtain
\[
\|g\|_2^2
\le \|g\|_4\|f\|_2^{1.5}\le 5^{\frac{3r}{4}}\delta^{\frac{1}{4}}
\|g\|_2^{0.5}\|f\|_2^{1.5}.
\]
Rearranging, and raising everything to the power $\frac{4}{3}$, we obtain
$\|g\|_2^2\le 5^{r}\delta^{\frac{1}{3}}\left\|f\right\|_2^2$.
\end{proof}

Let us say that $f$ is \emph{$\epsilon$-concentrated} above degree $r$ if
$\|f^{\le r}\|_2^2 \le \epsilon \|f\|_2^2$. The significance of
Theorem \ref{lem:normsense0} stems from the fact that it implies
the following result showing that for each $r,\epsilon>0$
there exists a $\delta>0$ such that any
sparse $(r,\delta)$-global function
is $\epsilon$-concentrated above degree $r$.

\begin{cor} \label{lem:normsense}
Let $r \ge 1$.
Suppose that  $f$ is an
$\left(r,\delta\right)$-global Boolean function
with $\mu_p\left( f\right)<\delta$.
Then $\mb{E}[(f^{\le r})^2]
\le 10^{r} \delta^{\frac{1}{3}}\mu_p(f)$.
\end{cor}
\begin{proof}
  By Lemma \ref{lem: gen influence
    control for global functions}, for each $S$ of size $r$ we have
  $ \mathrm{I}_S\left(f^{\le r}\right)\le\mathrm{I}_S\left(f\right)< 8^r \delta$.
  Then Theorem \ref{lem:normsense0} implies
  $\|f^{\le r}\|_2^2\le 10^r\delta^{\frac{1}{3}} \|f\|_2^2$,
  where since $f$ is Boolean
  we have $\|f\|_2^2 = \mu_p(f)$.
\end{proof}

\subsection{Isoperimetric stability}

We are now ready to prove our variant of the Kahn-Kalai Conjecture
and sharp form of Bourgain's Theorem, both of which can be thought
of as isoperimetric stability results. Both proofs closely follow
existing proofs and substitute our new hypercontractivity inequality
for the standard hypercontractivity theorem:
for the first we follow a proof of the isoperimetric inequality,
and for the second the proof of KKL given
by Bourgain and Kalai \cite{BourgainKalai19}
(their main idea is to apply the argument we gave in Theorem
\ref{thm:warm-up} for each of the derivatives of $f$).

\begin{proof}[Proof of Theorem \ref{thm:Variant of Kahn kalai}]
 We prove the contrapositive statement that for
 a sufficiently large absolute constant $C$, if $f$ is a Boolean function
 such that $\mu_p(f_{J\to 1})\le e^{-{CK}}$  for all $J$ of size at most $CK$, then $p\mathrm{I}[f]>K\mu_p(f)$.
 Let $f$ be such a function, and set $\delta=e^{-CK}$.
 Provided that $C>2$, $f$ is  $\left(2K,\delta\right)$-global, and has
 $p$-biased measure at most $\delta$.
 By Corollary \ref{lem:normsense}, we have
\[
\|f^{\le 2K}\|_2^2 \le 10^{2K}\delta^{\frac{1}{3}}\mu_p\left(f\right)\leq \mu_p\left(f\right)/2,
\]
provided that $C$ is sufficiently large. Hence,
\[
  \|f^{>2K}\|_2^2 = \|f\|_2^2-\|f^{\le 2K}\|_2^2
  \geq
  \mu_p\left(f\right)/2.
  \]
By \eqref{eq:inf} we obtain $p(1-p)\mathrm{I}[f]\ge
2K\|f^{>2K}\|_2^2$, so $p\mathrm{I}[f]> K\mu_p(f)$.
\end{proof}

Next we require the following lemma which bounds the norm of
a low degree truncation in terms of the total influence.
\begin{lem} \label{lem:normtruncate}
  Let $r \ge 0$. Suppose that for each nonempty
  set $S$ of size
  at most $r$, $\mathrm{I}_S\left(f^{\le r}\right)\le \delta.$ Then
  \[
    \|f^{\le r}\|_2^2\le \mu_p(f)^2 + 5^{r-1}\delta^{\frac{1}{3}}\sigma^2\mathrm{I}[f].
    \]
\end{lem}

\begin{proof}
Let $g_i := f_{i \to 1} - f_{i \to 0}$.
Then for each $S$ of size at most $r-1$, with $i\notin S$ we have
\[
  \mathrm{I}_S(g_{i}^{\le r})=\mathrm{I}_{S\cup\{i\}}(f^{\le r}) \le
  \delta,
\]
and for each $S$ containing $i$ we have $\mathrm{I}_S((g_i)^{\le r})=0$.
By Lemma \ref{lem:normsense0},
$\mb{E}[((g_i)^{\le r})^2]
\le 5^{r-1}\delta^{\frac{1}{3}}  \mb{E}[g_i^2]$.
The lemma now follows by summing over all $i$, using $\sum_i \mb{E}[g_i^2] = \mathrm{I}(f)$:
\begin{align*}
\| f^{\le r} \|_2^2
& = \sum_{|S| \le r} \hat{f}(S)^2
\le \hat{f}(\es)^2 + \sum_{|S| \le r} |S|\hat{f}(S)^2 \\
& = \mu_p(f)^2 + \sS^2 \sum_i \mb{E}[ ((g_i)^{\le r})^2 ]
\le \mu_p(f)^2 + 5^{r-1} \dD^{1/3} \sS^2 \mathrm{I}(f). & \qedhere
\end{align*}
\end{proof}

We now establish a variant of Bourgain's Theorem
for general Boolean functions, in which we replace
the conclusion on the measure of a restriction
by finding a large generalised influence.

\begin{thm}\label{thm:Bourgain++}
 Let $f\colon\{0,1\}^n\to\{0,1\}.$ Suppose that
$p\mathrm{I}[f] \le
  K \mu_p\left(f \right)(1-\mu_p(f))$. Then there exists
  an $S$ of size $2K$, such that $\mathrm{I}_S(f)\ge 5^{-8K}.$
\end{thm}
\begin{proof}
Let $r = 2K$ and let
$\dD = 5^{-8K}$.
Suppose for contradiction that $\mathrm{I}_S(f)\le\delta$ for each
set $S$ of size at most $r$. By Lemma \ref{lem:normtruncate},
\[
  \| f^{\le r} \|_2^2- \mu_p(f)^2
\le 5^{r-1} \dD^{1/3} \sS^2 I(f)
< p\mathrm{I}[f]/2K\le \mu_p(f)(1-\mu_p(f))/2.
\]

On the other hand, by Parseval \[
  \| f - f^{\le r} \|_2^2
= \sum_{|S| \ge r} \hat{f}(S)^2
\le r^{-1} \sum_{|S| \ge r} |S| \hat{f}(S)^2
\le r^{-1} p(1-p) \mathrm{I}(f)\le
\mu_p(f)(1-\mu_p(f))/2.
\]
However, these bounds contradict the fact that
\begin{align*}
  \mu_p(f)(1-\mu_p(f))=\|f\|_2^2 -\mu_p(f)^2= \| f^{\le r} \|_2^2-
  \mu_p(f)^2 + \| f - f^{\le r} \|_2^2. & \qedhere
\end{align*}
 \end{proof}

\begin{proof}[Proof of Theorem \ref{thm:Bourgain+}]
The theorem follows immediately from Theorem \ref{thm:Bourgain++} and
Lemma \ref{rem}.
\end{proof}

\subsection{Sharpness examples}
We now give two examples showing
sharpness of the theorems in this section,
both based on the tribes function of Ben-Or
and Linial \cite{ben1990collective}.

\begin{example} \label{eg1}
We consider the anti-tribes function $f=f_{s,w}:\{0,1\}^n \to \{0,1\}$
defined by $s$ disjoint sets $T_1,\dots,T_s \sub [n]$ each of size $w$,
where $f(x) = \prod_{j=1}^s \max_{i \in T_j} x_i$,
i.e.\ $f(x)=1$ if for every $j$ we have $x_i=1$ for some $i \in T_j$,
otherwise $f(x)=0$. We have $\mu_p(f) = (1-(1-p)^w)^s$
and $\mathrm{I}[f] = \mu_p(f)' = sw(1-p)^{w-1} (1-(1-p)^w)^{s-1}$.
We choose $s,w$ with $s(1-p)^w = 1$
(ignoring the rounding to integers) so that
$\mu_p(f)=(1-s^{-1})^s$ is bounded away from $0$ and $1$,
and $K = (1-p)p\mathrm{I}[f] = pw(1-s^{-1})^{-1} \mu_p(f) = \Tt(pw)$.
Thus $\log s = w\log (1-p)^{-1} = \Tt(K)$.
However, for any $J \sub [n]$ with $|J| = t \le s$ we have
$\mu_p(f_{J\to1}) \le (1-s^{-1})^{s-t} \le 2^{t/s} \mu_p(f)$,
so to obtain a density bump of $e^{-o(K)}$
we need $t = e^{-o(K)} s = e^{\Oo(K)} \gg K$.
Thus Theorem \ref{thm:Bourgain+} is sharp.
\end{example}

\begin{example} \label{eg2}
Let $f(x) = f_{s,w}(x) \prod_{i \in T} x_i$
with $f_{s,w}$ as in Example \ref{eg1}
and $T \sub [n]$ a set of size $t$ disjoint from $\cup_j T_j$.
We have $\mu_p(f) = p^t (1-(1-p)^w)^s$ and
$I[f] = \mu_p(f)' = tp^{t-1} (1-(1-p)^w)^s
+ p^t sw(1-p)^{w-1} (1-(1-p)^w)^{s-1}$.
We fix $K>1$ and choose $s,w$ with $s(1-p)^w = K$, so that
$\mu_p(f) = p^t (1-K/s)^s = p^t e^{-\Tt(K)}$ for $s>2K$
and $p(1-p)I[f] = \mu_p(f) ( (1-p)t + pw K (1-K/s)^{-1} )
= \mu_p(f) \Tt(K) $ if $pw = \Tt(1)$ and $t=O(K)$.
For any $J \sub [n]$ with $|J|=t+u \le t+s$
we have  $\mu_p(f_{J\to1}) \le (1-K/s)^{s-u}
\le e^{-K(1-u/s)} \le e^{-K/2}$ unless $u > s/2 = \Tt(K)$.
Thus Theorem \ref{thm:Variant of Kahn kalai} is sharp.
\end{example}

\subsection{Sharp thresholds: the traditional approach}\label{subsec:ST2}
In this section we deduce Theorem \ref{thm:Sharp threshold result}
from our edge-isoperimetric stability results
and the Margulis--Russo Lemma.
Recall that a monotone Boolean function
is $M$-global in an interval if
$\mu_p\left(f_{J\to 1}\right)\le \mu_p\left(f\right)^{0.01}$
for each $p$ in the interval and set $J$ of size $M$.
We prove the following slightly stronger version of Theorem
\ref{thm:Sharp threshold result}.
\begin{thm} \label{thm:trad}
 There exists an absolute constant $C$ such that
 the following holds for any monotone Boolean function $f$
 that is $M$-global in some interval  $\left[p,q\right]$:
 if $q\le p_c$ and  $\mu_p\left(f\right)\ge e^{-M/C}$ then
  \begin{equation} \label{eq:trad}
    \mu_q\left(f \right) \ge \mu_p(f)^{\left(\frac{p}{q}\right)^{1/C}}.
  \end{equation}
 In particular, $q\le M^{C} p$.
\end{thm}
\begin{proof}
 By Theorem \ref{thm:Variant of Kahn kalai}, since $f$ is $M$-global throughout the interval, there exists a constant
 $C$ such that  $\mathrm{I}_{x}\left[f\right]
 \ge \frac{\mu_{x}(f) \log(\frac{1}{\mu_{x}(f)})}{Cx}$
 for all $x$ in
 the interval $\left[p,q\right]$.
 By the Margulis-Russo lemma,
  \[
   \frac{d}{dx} \log \left(-\log
   (\mu_{x}\left(f\right))\right)
   =\frac{\mu_{x}(f)'}{\mu_{x}(f)\log(\mu_{x}\left(f\right))}
   =\frac{I_{x}[f]}{\mu_{x}(f)\log(\mu_{x}\left(f\right))}
   \le \frac{-1}{Cx}
 \]
 in all of the interval $\left[p,q\right]$. Hence,
\[
 \log\left(-\log
 (\mu_q(f))\right)\le \log(-\log(\mu_p(f))) - \frac{\log(\frac{q}{p})}{C}.
\]
The first part of the theorem follows by taking exponentials,
multiplying by $-1$ then taking exponentials again.
To see the final statement, note that $q \le p_c$
implies $\mu_q\left(f\right)\le\frac{1}{2}$.
We cannot have $q \ge M^c p$,
as then the right hand side in \eqref{eq:trad}
would be larger than $e^{-\frac{1}{C}} > 1/2$ for large $C$.
To obtain Theorem \ref{thm:Sharp threshold result}
we substitute $q=p_c$.
  \end{proof}

\section{Noise sensitivity and sharp thresholds} \label{sec:noise}

We start this section by showing that sparse global functions
are noise sensitive; Theorem \ref{thm:Noise sensitivity}
follows immediately from
Theorem \ref{thm:quantitative noise sensitivity}.

\begin{thm}\label{thm:quantitative noise sensitivity}
 Let $\rho\in\left(0,1\right)$, and let $\epsilon>0$. Let
 $r=\frac{\log(2/\epsilon)}{\log(1/\rho)}$,
 and let $\delta = 10^{-3r-1}\epsilon^3$.
 Suppose that $f$ is an
  $\left(r,\delta\right)$-global Boolean
  function with $\mu_p\left(f\right)<\delta$.  Then
  \[
    \mathrm{Stab}_\rho\left(f\right)\le \epsilon \mu_p\left(f\right).
    \]
  \end{thm}
\begin{proof}
We have
\begin{align*}
\left\langle \mathrm{T}_{\rho}f,f\right\rangle
& \le \sum_{\left|S\right|\le r}\hat{f}\left(S\right)^{2}
+ \rho^{r}\sum_{\left|S\right|>r}\hat{f}\left(S\right)^{2}
\le \mathbb{E}\left[\left(f^{\le r}\right)^{2}\right]
+ \frac{\eps}{2}\mu_{p}(f).
\end{align*}
The statement now follows from Corollary \ref{lem:normsense},
which gives $\mb{E}[(f^{\le r})^2]
\le 10^{r} \delta^{1/3} \mb{E}[f^2]
< \eps \mu_p(f)/2$.
\end{proof}

In the remainder of this section,
following \cite{lifshitz2018hypergraph},
we deduce sharp thresholds from noise sensitivity
via the following \emph{directed noise operator},
which is implicit in the work of
Ahlberg, Broman, Griffiths and Morris \cite{ahlberg2014noise}
and later studied in its own right by
Abdullah and Venkatasubramanian \cite{abdullah2015directed}.
\begin{defn}
Let $D\left(p,q\right)$
denote the unique distribution on pairs
$\left(\boldsymbol{x,y}\right)\in
\left\{ 0,1\right\} ^{n}\times\left\{ 0,1\right\} ^{n}$
such that $\xb\sim\mu_{p}$, $\yb\sim\mu_{q}$,
all $\xb_{i}\le\yb_{i}$ and
$\{ (\xb_{i},\yb_{i}) : i \in [n]\}$ are independent.
We define a linear operator
$\mathrm{T}^{p \to q}: L^2(\{0,1\}^n, \mu _p) \to  L^2(\{0,1\}^n, \mu _q)$ by
\[ \mathrm{T}^{p\to q}\left(f\right)\left(y\right)
=\mathbb{E}_{(\xb,\yb)\sim D\left(p,q\right)}
\left[f\left(\xb\right)|\,\yb=y\right]. \]
\end{defn}

The directed noise operator $\mathrm{T}^{p \to q}$
is a version of the noise operator
where bits can be flipped only from $0$ to $1$.
The associated notion of directed noise stability,
i.e.\ $\left\langle f,\art f\right\rangle _{\mu_{q}}$,
is intuitively a measure of how close
a Boolean function $f$ is to being monotone.
Indeed, for any $(\mathbf{x},\mathbf{y})$ with all $x_i \le y_i$
we have $f\left(\xb\right)f\left(\yb\right)\le f\left(\xb\right)$,
with equality if $f$ is monotone, so
\[ \left\langle f,\art f\right\rangle
= \mathbb{E}_{(\mathbf{x},\mathbf{y})\sim D\left(p,q\right)}
  \left[f\left(\xb\right)f\left(\yb\right)\right]
\le \mathbb{E}_{(\mathbf{x},\mathbf{y})
\sim D\left(p,q\right)}\left[f\left(\xb\right)\right]
=\mu_{p}\left(f\right),\]
with equality if $f$ is monotone\footnote{
The starting point for \cite{lifshitz2018hypergraph}
is the observation that this inequality is close
to an equality if $f$ is almost monotone.}.
We note that the adjoint operator
$\left(\mathrm{T}^{p\to q}\right)^{\star}:
L^2(\{0,1\}^n, \mu _q) \to  L^2(\{0,1\}^n, \mu _p)$
defined by $\bgen{\mathrm{T}^{p\to q}f,g}
= \bgen{f,\left(\mathrm{T}^{p\to q}\right)^{\star}g} $
 satisfies $\left(\mathrm{T}^{p\to q}\right)^{\star}
= \mathrm{T}^{q\to p}$, where
\[ \mathrm{T}^{q\to p}\left(g\right)\left(x\right)
=\mathbb{E}_{(\xb,\yb)\sim D\left(p,q\right)}
\left[g\left(\yb\right)|\,\xb=x\right]. \]
The following simple calculation relates these
operators to the noise operator.

\begin{lem} \label{calcrho}
Let $0<p<q<1$ and $\rho = \frac {p(1-q)}{q(1-p)}$. Then
$\left(\mathrm{T}^{p\to q}\right)^{\star}\mathrm{T}^{p\to q}
= \mathrm{T}_{\rho}$ on $L^2(\{0,1\}^n, \mu _p)$.
\end{lem}

\begin{proof}
We need to show that the following distributions
on pairs of $p$-biased bits $(\mathbf{x},\mathbf{x'})$ are identical:
(a) let $\mathbf{x}$ be a $p$-biased bit, with probability $\rho$
let $\mathbf{x'}=\mathbf{x}$,
otherwise let $\mathbf{x'}$ be an independent $p$-biased bit,
(b) let $(\mathbf{x},\mathbf{y}) \sim D(p,q)$ and then
$(\mathbf{x'},\mathbf{y}) \sim D(p,q) \mid y$.
It suffices to show $\mb{P}(x \ne x')$
is the same in both distributions.
We condition on $x$. Consider $x=1$. In distribution (a)
we have $\mb{P}(\mathbf{x'}=0) = (1-\rho)(1-p)$.
In distribution (b) we have $\mb{P}(\mathbf{y}=1)=1$
and then $\mb{P}(\mathbf{x'}=0) = 1-p/q = (1-\rho)(1-p)$, as required.
Now consider $\mathbf{x}=0$. In distribution (a)
we have $\mb{P}(\mathbf{x'}=1) = (1-\rho)p$.
In distribution (b) we have $\mb{P}(\mathbf{y}=1)=\tfrac{q-p}{1-p}$
and then $\mb{P}(\mathbf{x'}=1 \mid \mathbf{y}=1) = p/q$, so
$\mb{P}(\mathbf{x'}=1) = \tfrac{p(q-p)}{q(1-p)} = (1-\rho)p$, as required.
\end{proof}

We now give an alternative way to deduce sharp threshold results,
using noise sensitivity, rather than the traditional approach
via total influence (as in the proof of Theorem \ref{thm:trad}).
Our alternative approach has the following additional
nice features, both of which have been found useful in
Extremal Combinatorics (see \cite{lifshitz2018hypergraph}).
\begin{enumerate}
\item To deduce a sharp threshold result in an interval
$\left[p,q\right]$ it is enough to show that $f$ is
global only according to the $p$-biased distribution.
This is a milder condition than the one in the traditional approach,
that requires globalness throughout the entire interval.
\item The monotonicity requirement may be relaxed
to ``almost monotonicity''.
\end{enumerate}

\begin{prop}
\label{prop:Noise sensetivity implies sharp threshold}
Let $f\colon\left\{ 0,1\right\} ^{n}\to\left\{ 0,1\right\} $
be a monotone Boolean function. Let $0<p<q<1$
and $\rho=\frac{p\left(1-q\right)}{q\left(1-p\right)}$.
Then $\mu_q(f) \ge \mu_p(f)^2 / \mathrm{Stab}_{\rho}\left(f\right)$.
\end{prop}
\begin{proof}
By Cauchy\textendash Schwarz and Lemma \ref{calcrho},
\begin{align*}
\mu_{p}\left(f\right)^2=\left\langle \art f,f\right\rangle _{\mu_{q}}^2\le \left\langle \art f,\art f\right\rangle _{\mu_{q}}\left\langle f,f\right\rangle _{\mu_{q}}
=\left\langle \mathrm{T}_{\rho}f,f\right\rangle _{\mu_{p}}\mu_{q}\left(f\right).
& \qedhere
\end{align*}
\end{proof}

The above proof works not only for monotone functions,
but also for functions where the first equality above
is replaced by approximate equality (which is a natural
notion for a function to be ``almost monotone'').
We conclude this part by recalling
the following sharp threshold theorem for global functions,
and noting that its proof
is immediate from Theorem \ref{thm:quantitative noise sensitivity}
and Proposition \ref{prop:Noise sensetivity implies sharp threshold}.


\qrsharpthreshold*

\newpage

\part{Pseudorandomness and junta approximation} \label{part:pr+junta}

The first main result proved in this part
will be our junta approximation theorem, Theorem \ref{thm:junta},
which we will now restate,
using the notation $\g(r,s,\DD)$ for the family
of all $r$-graphs $G$ with $s$ edges
and maximum degree $\DD (G) \leq \DD$.
We recall that $S \subset V(G^+)$ is a \emph{crosscut}
if $|E \cap S| = 1$ for all $E \in G^+$,
and $\cc {G}$ denotes the minimum size of a crosscut.

\begin{thm} \label{thm:junta'}
Let $G \in \g(r,s,\DD)$ and $C \gg r\DD\eps^{-1}$.
Then for any $G^+$-free $\f \sub \tbinom {[n]}{k}$
with $C \leq k \leq n/Cs$,
there is $J \sub V(G)$ with $|J| \leq \cc {G}-1$
and $|\f \sm \s_{n,k,J}| \le \eps |\s_{n,k,\cc {G}-1}|$.
\end{thm}

The set $J$ in Theorem \ref{thm:junta'}
will consist of all vertices of suitably large degree.
Thus ${\cal F}^\es_J := {\cal F} \sm {\cal S}_{n,k,J}$
does not have any vertices of large degree, which we think
of a pseudorandomness property, called `globalness',
due to its interpretation as globalness
of the corresponding characteristic Boolean function.

An important theme of this part, treated in its first section,
will be the interplay between two pseudorandomness notions:
globalness and another, called uncapturability.
We will see that globalness implies uncapturability,
and that uncapturability can be `upgraded'
to globalness by taking appropriate restrictions.

In the proof of Theorem \ref{thm:junta'} we will consider
separately the two steps of showing $|J| \leq \cc {G}-1$ and
$|{\cal F} \sm {\cal S}_{n,k,J}| \leq \eps |S_{n,k,\cc {G}-1}|$.
For both steps we consider a two step embedding strategy for $G^+$,
where in the first step we embed\footnote{
For simplicity in this overview we are only describing the embedding
strategy used to bound $|{\cal F}^\es_J|$; the strategy for bounding $|J|$
is similar, but adapted so that $J$ can play the role of a crosscut in $G$.} 
$G$ in the `fat shadow' of ${\cal F}$
(meaning that the image of every edge
has many extensions to an edge of ${\cal F}$)
and in the second step we `lift' edges
from the fat shadow to the original family.

This proof strategy is implemented at the end of the first section,
assuming results that will be proved in later sections.
The lifting step requires results on cross matchings
presented in Section \ref{sec:match},
which will also be used for the proof
of the Huang--Loh--Sudakov Conjecture in Section \ref{sec:hls}.
The analysis of fat shadows and the embedding steps
will be carried out in Section \ref{sec:shadow}.

After proving Theorem \ref{thm:junta},
in Section \ref{sec:refine} we prove
the following refined junta approximation result,
in which we improve the bound on $|{\cal F}^\es|$;
besides being of interest in its own right,
this bound is needed for the proofs of
our exact Tur\'an results in the next part.

\begin{thm} \label{thm:juntarefined}
Let $G \in \g(r,s,\DD)$,
$0 < C^{-1} \ll \dD \ll \eps \ll (r\DD)^{-1}$
and $C \leq k \leq n/Cs$.
Then for any $G^+$-free $\f \sub \tbinom {[n]}{k}$
with $|\f| > |{\cal S}_{n,k,\cc {G}-1}| - \dD \tbinom{n-1}{k-1}$
there is $J \in \tbinom{[n]}{\cc {G}-1}$ with
$|\f \sm {\cal S}_{n,k,J}| \leq \eps \tbinom{n-1}{k-1}$.
\end{thm}

Throughout the remainder of the paper
it will often be convenient to assume
that $G$ belongs to the subset $\g'(r,s,\DD)$
of $\g(r,s,\DD)$ consisting of its $r$-partite $r$-graphs.
There is no loss of generality in this assumption,
as $G^+(r\DD)$ is $r\DD$-partite for any $G \in \g(r,s,\DD)$.
To see this, consider a greedy algorithm in which we assign
vertices of $G$ sequentially to $r\DD$ parts, ensuring for
every edge that all of its vertices are in distinct parts.
Clearly this algorithm can be completed. Then the expansion vertices
can be assigned so that each edge of $G^+$ has one vertex in each part.

\section{Globalness and uncapturability}

This section introduces the key concepts that will underpin
this part of the paper. After introducing some basic definitions
that run throughout the paper in the first subsection,
we will define and analyse our two pseudorandomness notions
in the second subsection. We conclude in the third section
by proving our junta approximation theorem, assuming two embedding
lemmas that will be proved in Section \ref{sec:shadow}.

\subsection{Definitions}

Given $m,n \in {\mathbb N}$ with $m\leq n$
we let $[n] = \{1,2,\ldots, n\}$
and $[m,n] = \{m, m+1,\ldots, n\}$.
We write $\{0,1\}^X$ for the power set
(set of subsets) of a set $X$
(identifying sets with their characteristic 0/1 vectors)
and $\tbinom {X}{k} = X^{(k)}
= \big \{A \subset X: |A| = k \big \}$.
We call ${\cal F} \subset \{0,1\}^X$ a family
or a hypergraph on the vertex set $X$,
and the elements of ${\cal F}$ are called edges.
We say ${\cal F}$ is $k$-uniform
if ${\cal F} \subset \tbinom {X}{k}$;
we also call ${\cal F}$ a $k$-graph on $X$.

Given a family ${\cal F} \subset \{0,1\}^X$ and
$B \subset J \subset X$ we write $\f_{J}^{B}$ for the family
	\begin{equation*}
\f _J^B := \big \{ A\in \{0,1\}^{X\sm J} :
		A \cup B \in \f \big \} \subset \{0,1\}^{X\sm J}.
	\end{equation*}
Clearly ${\cal F}^{B}_J$ is $(k-|B|)$-uniform if ${\cal F}$ is $k$-uniform.
If either $B$ or $J$ has a single element $\{j\}$
then we will often suppress the bracket,
e.g.\ ${\cal F}^v_v = {\cal F}^{\{v\}}_{\{v\}}$.

We refer to ${\cal F}^v_v$ as the \emph{exclusive link} of $v$ in ${\cal F}$.
The \emph{inclusive link} of $v$ in ${\cal F}$
is ${\cal F}*v := \{ E \in {\cal F}: v \in E\}$.
The \emph{degree} of a vertex $v$ in ${\cal F}$
is $d_{\cal F}(v) = |{\cal F}^v_v| = |{\cal F}*v|$.
The minimum and maximum degrees of ${\cal F}$
are $\dD ({\cal F}) = \min _{v\in V({\cal F})} d_{\cal F}(v)$
and $\DD ({\cal F}) = \max _{v\in V({\cal F})} d_{\cal F}(v)$.

Let $\h_1,\dots,\h_s \subset \{0,1\}^V$.
We say that $\f_1,\dots,\f_s \sub \{0,1\}^X$
\emph{cross contain} $\h_1,\dots,\h_s$
if there is an injection $\phi:V \to X$ such that
$\phi(\h_i) \sub \f_i$ for all $i\in [s]$.
Here we write $\phi(\h_i) = \{ \phi(e): e \in \h_i \}$
with each $\phi(e) = \{ \phi(x): x \in e\}$.

We simply say that ${\cal F}_1,\ldots, {\cal F}_s$
cross contain ${\cal H}$ if they cross contain
any ordering of the edges of ${\cal H}$.
Thus a single hypergraph ${\cal F}$ contains ${\cal H}$
if ${\cal F}_1,\ldots, {\cal F}_s$ cross contain ${\cal H}$,
where ${\cal F}_i = {\cal F}$ for all $i\in [s]$.

Given an $r$-graph $G$ and $k \geq r$,
we recall that the $k$-expansion
$G^+ = G^+(k)$ is the $k$-uniform hypergraph
obtained from $G$ by adding $k-r$ new vertices to each edge,
i.e.\ $G^+$ has edge set $\{e \cup S_e: e\in E(G)\}$
where $|S_e| = k-r$, $S_e \cap V(G) = \es$
and $S_e \cap S_{e'} = \es $ for all distinct $e, e' \in E(G)$.

When embedding expanded hypergraphs in uniform families,
we may allow the uniformity of our families to vary,
defining cross containment of $G^+$ in the obvious way:
the edge of $G^+$ embedded in the family
$\f _i \subset\tbinom{\left[n\right]}{k_i}$
is obtained from an edge of $G$
by adding $k_i-r$ new vertices.

A family ${\cal F} \subset \{0,1\}^X$ is said to be \emph{monotone}
if given $F \in {\cal F}$ and $F \subset F'\subset X$ we also have $F' \in {\cal F}$.
Given $\f \subset \{0,1\}^X$ the \emph{up closure} of ${\cal F}$ is the monotone family
$\f^{\uparrow} = \{B \subset X: A \subset B \mbox { for some }A \in {\cal F}\} \subset \{0,1\}^X$.
The $\ell $-shadow of ${\cal F}$ is
$\pl^\ell({\cal F}) := \{F \in \tbinom {X}{\ell }:
 F \subset G \mbox { for some } G \in {\cal F}\}$.

Given $\f\subset \tbinom{X}{k}$ we will write $\mu(\f)=|\f|/\tbinom{|X|}{k}$.
Some of our results are more naturally stated with $|{\cal F}|$
and others with $\mu ({\cal F})$, so we will freely move between these settings.
Given $p\in [0,1]$ we will use $\mu_{p}$
to denote the $p$-biased measure on $\{0,1\}^n$,
where a set $\boldsymbol{A}\sim\mu_{p}$ is selected
by including each $i\in [n]$ independently with probability $p$.
We extend this notation to families $\f\subset \{0,1\}^n$
by $\mu_{p}\left(\f\right) :=\Pr_{\boldsymbol{A}\sim\mu_{p}}\left[\boldsymbol{A}\in\f\right]$.
We often identify a family $\mc{F}$ with its characteristic Boolean function $f:\{0,1\}^n \to \{0,1\}$
and apply the above terminology freely in either setting,
e.g.\ we call $f$ monotone if $\mc{F}$ is monotonem
and write $\mu _p(f)$ for the expectation of $f$ under $\mu _p$.

To pass between these measures we note the following simple properties
that will be henceforth used without further comment.
For any $\f \subset \{0,1\}^n$ and $J \sub [n]$,
we have the union bound estimate
\[ \mu_p(\f) \le \mu_p(\f_J^\es)
+ p \sum_{j \in J} \mu_p(\f^j_j)
\le \mu_p(\f_J^\es) + |J|p. \]
The same estimate holds replacing $\mu_p$
by uniform measures $\mu$
for $\f \sub \tbinom{[n]}{k}$ with $k=pn$,
remembering to use the correct normalisations:
we have $\mu(\f) = |\f| \tbinom{n}{k}^{-1}$
and $\mu(\f^j_j) = |\f^j_j| \tbinom{n-1}{k-1}^{-1}$.

In the other direction, we have the bounds
\begin{align*}
\mu_p(\f) & \ge (1-p)^{|J|}\mu_p(\f_J^\es)
\text{ for } \f \sub \{0,1\}^n, \text{ and } \\
\mu(\f) & \ge \tbinom{n}{k}^{-1}
\tbinom{n-|J|}{k} \mu(\f_J^\es)
\ge \big(1-\tfrac{|J|}{n-k}\big)^k \mu(\f_J^\es)
\text{ for } \f \sub  \tbinom{[n]}{k}.\\
\end{align*}

Throughout $a \ll b$ or $a^{-1} \gg b^{-1}$
will mean that the following statement holds
provided $a$ is sufficiently small
as a function of $b$.

\subsection{Pseudorandomness}

Here we define our two key notions of pseudorandomness
for set systems, namely uncapturability and globalness,
and explore some of their basic properties.

\begin{defn} \label{def:pr}
Let $\f \subset \{0,1\}^n$
and $\mu$ be a measure on $\{0,1\}^n$.

We say $\f$ is $(\mu,a,\eps)$-uncapturable
if $\mu(\f^\es_J) \ge \eps$ whenever
$J \subset [n]$ with $|J| \le a$.

We say $\f$ is $(\mu,a,\eps)$-global
if $\mu(\f^J_J) \le \eps$ whenever
$J \subset [n]$ with $|J| \le a$.

We say $\f$ is $(\mu,a,\eps)$-capturable
if it is not $(\mu,a,\eps)$-uncapturable,
or $(\mu,a,\eps)$-local if it is not $(\mu,a,\eps)$-global.
We omit $\mu$ from the notation if it is clear from the context,
i.e.\ if $\f \subset \tbinom{[n]}{k}$ with uniform measure
or $\f \subset \{0,1\}^n$ with $p$-biased measure $\mu_p$,
where $p$ is clear from the context.
\end{defn}

We now establish some basic properties
of these definitions. For each property
we state two lemmas that apply
when $\mu$ is uniform or $\mu=\mu_p$.
We only give proofs in the uniform setting,
as those in the $p$-biased setting are essentially the same.
The following pair of lemmas shows that
globalness is preserved by restrictions.

\begin{lem} \label{lem:globalrestrict}
If $\f \subset \tbinom{[n]}{k}$ is $(a,\eps)$-global
and $I \sub J \sub [n]$ with $|I|<a$ and $|J|<n/2k$
then $\f^I_J$ is $(a-|I|,2\eps)$-global.
\end{lem}

\begin{lem} \label{lem:globalrestrict'}
If $\f \subset \{0,1\}^n$ under $\mu_p$ is $(a,\eps)$-global
and $I \sub J \sub [n]$ with $|I|<a$ and $|J|<1/2p$
then $\f^I_J$ is $(a-|I|,2\eps)$-global.
\end{lem}

\begin{proof}[Proof of Lemma \ref{lem:globalrestrict}]
For any $K \sub [n] \sm J$ with $|K| \le a-|I|$,
we have $\mu(\f^{I \cup K}_{I \cup K}) \le \eps$,
so $\mu((\f^I_J)^K_K) \le
\big(1-\tfrac{|J \sm I|}{n-k}\big)^{-k} \eps < 2\eps$.
\end{proof}

The next pair shows that globalness implies uncapturability.

\begin{lem} \label{lem:globaluncap}
If $\f \subset \tbinom{[n]}{k}$ is $(1,\eps)$-global
with $\eps =  \mu(\f) n/2ak$
then $\f$ is $(a,\mu(\f)/2)$-uncapturable.
\end{lem}

\begin{lem} \label{lem:globaluncap'}
If $\f \subset \{0,1\}^n$ under $\mu_p$ is $(1,\eps)$-global
with $\eps =  \mu_p(\f)/2ap$
then $\f$ is $(a,\mu_p(\f)/2)$-uncapturable.
\end{lem}

\begin{proof}[Proof of Lemma \ref{lem:globaluncap}]
If $|J| \le a$ then
$\mu(\f^\es_J) \ge \mu(\f) - \eps ak/n \ge \mu(\f)/2$.
\end{proof}

Uncapturability does not imply globalness,
but we do have a partial converse:
by taking restrictions we can upgrade uncapturable families
to families that are global or large.

\begin{lem} \label{lem:upgrade1}
Suppose $\bB \in (0,.1)$ and
$\f_i \subset \tbinom{[n]}{k_i}$ with $2r < k_i < \bB n/2rm$
are $(rm,\dD_i)$-uncapturable for $i \in [m]$.
Then there are pairwise disjoint $S_1,\dots,S_m$
with each $|S_i| \le r$ such that,
setting $\g_i = (\f_i)^{S_i}_S$ where $S = \bigcup_i S_i$,
whenever $\mu(\g_i) < \bB$ we have
$S_i = \es$ and $\g_i$ is $(r,2\bB)$-global
with $\mu(\g_i) > \dD_i$.
\end{lem}

\begin{lem} \label{lem:upgrade1'}
Suppose $\bB \in (0,.1)$ and
$\f_i \subset \tbinom{[n]}{k_i}$ with $k_i < \bB n/2rm$
are $(rm,\dD_i)$-uncapturable for $i \in [m]$.
Then there are pairwise disjoint $S_1,\dots,S_m$
with each $|S_i| \le r$ such that,
setting $\g_i = (\f_i^\ua)^{S_i}_S$ where $S = \bigcup_i S_i$
and $p_i = k_i/(n-|S|)$,
whenever $\mu_{p_i}(\g_i) < \bB$ we have
$S_i = \es$ and $\g_i$ is $(r,2\bB)$-global
with $\mu_{p_i}(\g_i) > \dD_i/4$.
\end{lem}

\begin{proof}[Proof of Lemma \ref{lem:upgrade1}]
Let $(S_i: i \in I)$ be a maximal collection
of pairwise disjoint sets with $|S_i| \le r$
and $\mu((\f_i)^{S_i}_{S_i}) > 1.5\bB$.
Let $S = \bigcup_{i \in I} S_i$ and
$\g_i = (\f_i)^{S_i}_S$ for each $i \in [m]$,
where $S_i=\es$ for $i \in [m] \sm I$.
For any $i \in I$ we have $\mu(\g_i)
> \mu((\f_i)^{S_i}_{S_i}) - |S \sm S_i|k_i/n > \bB$.
Now consider $i$ with $\mu(\g_i) < \bB$.
Then $i \notin I$, so $S_i = \es$ and
$\mu(\g_i) > \dD_i$ by uncapturability.
Furthermore, for any $R \sub [n] \sm S$ with $|R| \le r$
we have $\mu((\f_i)^R_R) \le 1.5\bB$,
so $(\g_i)^R_R = ((\f_i)^R_R)^\es_S$
has $\mu((\g_i)^R_R) \le
\big(1-\tfrac{|S|}{n-k_i}\big)^{-k_i} \mu((\f_i)^R_R) < 2\bB$.
\end{proof}

We conclude this subsection with a lemma on
decomposing any family according to its vertex degrees,
where to make an analogy with the regularity method
we think of high degree vertex links as `structured'
and the low degree remainder as `pseudorandom'.

\begin{lem} \label{lem:decomp}
Let $\f \subset \tbinom{[n]}{k}$
and $J = \{i: \mu(\f^i_i)>\eps\}$.
If $|J|<n/2k$ then $\g=\f^\es_J$ is $(1,2\eps)$-global,
and so $(a,\mu(\g)/2)$-uncapturable with $a = \mu(\g) n/4k\eps$,
\end{lem}

\begin{proof}
If $j \in [n] \sm J$ then
$\mu(\f^j_j) \le \eps$ by definition of $J$,
so $\mu(\g^j_j) \le \big(1-\tfrac{|J|}{n-k}\big)^{-k} \mu(\f^j_j) < 2\eps$.
The lemma follows by Definition \ref{def:pr}
and Lemma \ref{lem:globaluncap}.
\end{proof}

\subsection{Embeddings}

Here we will prove Theorem \ref{thm:junta}
assuming two fundamental embedding results,
which will be proved in Section \ref{sec:shadow}.
The first of these shows that sufficiently large families
contain a cross copy of any expanded hypergraph $G^+$.
Our bound on $\mu(\f_{i})$ is sharper for larger $k_i$:\
when $k_i=O(1)$ it is a constant,
which is relatively weak (but still useful),
whereas when $k_i \gg \log n$
it is $O(sk_i/n) = O(\cc{G} k_i/n)$,
which is tight up to the constant factor.

\begin{lem} \label{lem:embedlarge}
Given $G \in \g(r,s,\DD)$, $C \gg r\DD$
and $C \leq k_i \leq n/Cs$ for all $i \in [s]$,
any $\f_i \subset\tbinom{[n]}{k_i}$
with all $\mu(\f_{i}) \ge e^{-k_i/C}+ Csk_i/n$
cross contain $G^+$.
\end{lem}

When the uniformities $k_i$ are small we cannot improve
this cross containment result, as below density $e^{-\OO(k_i)}$
the families $\f_i$ may have disjoint supports.
However, when finding $G^+$ in a single family $\f$
we can get a much better bound on the density,
and moreover it suffices to assume
that $\f$ is sufficiently uncapturable, as follows.

\begin{lem} \label{lem:embeduncap}
Given $G \in \g(r,s,\DD)$, $C \gg C_1 \gg C_2 \gg r\DD$
and $C \leq k \leq n/Cs$,
any $(C_1 s, sk/C_2 n)$-uncapturable
$\f \subset\tbinom{[n]}{k}$ contains $G^+$.
\end{lem}

We conclude this section by deducing
our junta approximation theorem
from the above lemmas.

\begin{proof}[Proof of Theorem \ref{thm:junta}]
Let $G \in \g(r,s,\DD)$ and $C \gg C_1 \gg C_2 \gg r\DD \eps^{-1}$.
Consider any $G^+$-free ${\cal F} \sub \tbinom {[n]}{k}$
with $C \leq k \leq \tfrac {n}{Cs}$.
Let $J = \{ i \in [n]: \mu(\mc{F}^i_i) \ge \bB \}$,
where $\bB := e^{-k/C_1} + C_1 sk/n$.
We need to show $|J| \leq \cc {G}-1$
and $|{\cal F}^\es_J| \leq \eps |{\cal S}_{n,k,\cc {G}-1}|$.

The bound on $|J|$ follows from Lemma \ref{lem:embedlarge}.
Indeed, supposing for a contradiction $|J| \ge \cc{G}$,
we may fix a minimal crosscut $S$ of $G^+$
and distinct $j_s \in J$ for each $s \in S$.
Let $I = \{i_s: s \in S\}$ and
$\mc{F}_s = \mc{F}^{i_s}_I$ for $s \in S$.
By definition of $J$, for each $s \in S$ we have
$\mu(\mc{F}_s) > \bB - |I|k/n > \bB/2$,
so by Lemma \ref{lem:embedlarge}
the families $(\mc{F}_s: s \in S)$ cross contain
the exclusive links $((G^+)^s_s: s \in S)$.
However, this contradicts ${\cal F}$ being $G^+$-free.

As $|J|<s \le n/Ck$ we can apply Lemma \ref{lem:decomp}
to see that $\g=\f^\es_J$ is
$(a,\mu(\g)/2)$-uncapturable with $a = \mu(\g) n/4k\bB$.
However, by Lemma \ref{lem:embeduncap}
$\g$ is $(C_1 s, sk/C_2 n)$-capturable,
so we must have $\mu(\g)/2 < sk/C_2 n$,
or $a < C_1 s$, so again
$\mu(\g) < 4\bB C_1 sk/n < sk/C_2 n$.
As $\mu({\cal S}_{n,k,\cc {G}-1}) > .9(\cc {G}-1)k/n$
and $s \le \DD\cc{G}$ we deduce
$|{\cal F}^\es_J| = |\g| < \eps |{\cal S}_{n,k,\cc {G}-1}|$.
\end{proof}

\section{Matchings} \label{sec:match}

The main result of this section is the following lemma
on cross containment of matchings in uncapturable families,
which will be used for `lifting' (as described in the previous section)
and also in the proof of the Huang--Loh--Sudakov Conjecture.

\begin{lem} \label{lem:matchuncap}
Let $C \gg C_1 \gg C_2 \gg 1$ and $\f_i \sub \tbinom{[n]}{k_i}$
with $k_i \le n/Cs$ for $i \in [s]$.
Suppose $\f_i$ is $(C_1 m, mk_i/C_2 n)$-uncapturable for $i \in [m]$
and $\mu(\f_i) > C_1 sk_i/n$ for $i>m$.
Then $\f_1,\dots,\f_s$ cross contain a matching.
\end{lem}

We start in the first subsection by recalling some basic probabilistic tools,
and also our new sharp threshold result from Part \ref{part:sharp}.
Next we present some extremal results on cross matchings in the second subsection.
We conclude by proving the uncapturability result in the third subsection.

\subsection{Probabilistic tools and sharp thresholds}

We start with the following lemma that will be used to pass between
the uniform and $p$-biased measures.

\begin{lem}
	\label{lem:bin}
	Let $n,k \in {\mathbb N}$ with $k=pn \leq n$.
	Then ${\mathbb P}\big (\Bin (n,p) \geq k \big ) \geq 1/4$.
	Thus if ${\cal A} \subset \tbinom {[n]}{k}$ we have
	$\mu _p({\cal A}^{\uparrow }) \geq \mu ({\cal A})/4$.
\end{lem}

\begin{proof}
	The first statement appears in \cite{greenberg}. The second
	holds as $\big |{\cal A}^{\uparrow }
	\cap \tbinom {[n]}{j} \big | \geq \aA \tbinom {n}{j}$
	for $j\geq k$ by the LYM inequality, and so
	$\mu _p({\cal A}^{\uparrow }) \geq \sum
	_{j = k}^n {\mathbb P}\big (\Bin (n,p) = j\big )
	\mu \big ({\cal A}^{\uparrow } \cap \tbinom {[n]}{j} \big )
	\geq {\mathbb P}\big ( \Bin (n,p)
	\geq k \big ) \aA \geq \aA /4$.
\end{proof}

We will also need the following well-known Chernoff bound
(see \cite[Theorem 2.8]{JLR}),
as applied to sums of Bernoulli random variables,
i.e.\ random variables which take values in $\{0, 1\}$;
if these are identically distributed
then we obtain a binomial variable.
The inequality can also be applied to a hypergeometric
random variable (see \cite[Remark 2.11]{JLR}),
i.e.\ $|S \cap T|$ with $S \in \tbinom{X}{s}$
and uniformly random $T \in \tbinom{X}{t}$
for some $X$, $s$ and $t$.

\begin{lem} \label{conc}
Let $X$ be a sum of independent Bernoulli random variables and $0<a<3/2$.
Then $\mb{P}\big[|X - \mb{E} X| \geq a\mb{E} X\big]
\le 2e^{-\frac{a^2}{3}\mb{E} X}$.
\end{lem}

Next we recall our sharp threshold result
for global functions that we proved in Part \ref{part:sharp}
which will play a crucial role in this section,
and so for all subsequent applications
of Lemma \ref{lem:matchuncap}.

\qrsharpthreshold*

We will apply the following two consequences of this result.

\begin{thm} \label{thm:sharpset}
Suppose $\f \subset \{0,1\}^n$ is monotone with $\mu_p(\f)=\mu$.
\begin{enumerate}
\item If $\mu \ll r^{-1} \ll \eps$ then
there is $R\subset [n]$ with $|R| \leq r$
and $\mu_{2p}(\f^R_R) \ge \mu/\eps$.
\item If $p \ll K^{-1} \ll \eta \ll 1$ then
there is $R \subset [n]$ with $|R| \le K\log \mu^{-1}$
and $\mu_{Kp}(\f^R_R) \ge \mu^\eta$.
\end{enumerate}
\end{thm}

\begin{proof}
Let $f$ be the monotone Boolean characteristic function of $\f$.

For (1) we apply
Theorem	\ref{theorem: quasirandom sharp threshold theorem}
with $\aA=1$ and the same $\eps$,
If $f$ is not $(r,\dD)$-global then for some $R$ with $|R| \ge r$
we have $\mu_{2p}(\f^R_R) \ge \mu_p(\f^R_R) = \mu_p(f_{R\to 1})
\ge \dD \ge \mu/\eps$. On the other hand,
if $f$ is $(r,\dD )$-global then we can take $R=\es$, as
Theorem	\ref{theorem: quasirandom sharp threshold theorem}
gives $\mu_{2p}(\f) \ge \mu/\eps$.

For (2), we repeatedly apply
Theorem	\ref{theorem: quasirandom sharp threshold theorem}
with $\aA=1$ and $\eps = \mu^{\eta^2}$,
so $r=C\log \eps^{-1} = C\eta^2\log \mu^{-1}$
and $\dD=10^{-3r-1}\eps^3 \ge \mu^\eta$,
as we may assume $\eta \ll C^{-1}$.
We can assume that $f$ is $(r,\dD)$-global,
otherwise we immediately obtain $R$ as required,
so $\mu_{2p}(\f) \ge \mu/\eps = \mu^{1-\eta^2}$.
Repeating the argument, if we do not find $R$
then after $t \le \eta^{-2}$ iterations
we reach $\mu_{2^t p}(\f) \ge \dD \ge \mu^\eta$,
so we can take $R=\es$.
\end{proof}

\subsection{Extremal results}

In this subsection we adapt the method
of \cite[Lemma 3.1]{huang2012size}
to prove a variant form of the following result
of Huang, Loh and Sudakov \cite{huang2012size}.

\begin{lem} \label{lem:hls}
Let $k_1,\ldots, k_s,n \in {\mathbb N}$
with $\sum _{i\in [s]} k_i \leq n$.
Suppose $\f _i \subset \tbinom {[n]}{k_i}$ for $i\in [s]$
do not cross contain a matching.
Then $\mu (\f _i) \leq k_i(s-1)/n$ for some $i\in [s]$.
\end{lem}

We will prove the following variant that allows
a few families to be significantly smaller.

\begin{lem}\label{lem:matchlarge}
Let $1 \le m \leq s$, $k_1,\ldots, k_s \geq 0$
and $n \geq \sum _{i\in [s]} k_i$. Suppose
${\cal F}_i \subset \tbinom {[n]}{k_i}$ with
$\mu ( {\cal F}_i ) > 2k_im/n$ for
$i\in [m]$ and
$\mu ( {\cal F}_i ) > 2k_is/n$ for
$i\in [m+1,s]$.
Then $\{ {\cal F}_i \}_{i\in [s]}$
cross contain a matching.
\end{lem}

We also require the following version for the $p$-biased measure,
which we will deduce from Lemma \ref{lem:matchlarge}
by a limit argument similar to those in
\cite{dinur2005hardness, frankl2003weighted}.

\begin{lem} \label{lem:matchlargelimit}
Let $m \leq s$ and $p_1,\ldots ,p_s> 0$ with $\sum _{i\in [s]} p_i
\leq 1/2$. Suppose that ${\cal F}_1,\ldots {\cal F}_s \subset
\{0,1\}^n$ are monotone families with
$\mu _{p_i}\big ( {\cal F}_i \big ) \geq 3mp_i$ for $i\in [m]$
and $\mu _{p_i}\big ( {\cal F}_i \big ) \geq 3sp_i$
for $i\in [m+1,s]$. Then $\{{\cal F}_i\}_{i\in [s]}$
cross contain a matching.
\end{lem}

We introduce the following terminology.
Given ${\bf a} = (a_1,\ldots, a_s) \in {\mathbb R}^s$
and $n,k_1,\ldots, k_s \geq 0$ we say ${\bf a}$ is
\emph{forcing} for $(n,k_1,\ldots, k_s)$
if any families ${\cal F}_1,\ldots , {\cal F}_s$
with  ${\cal F}_i \subset \tbinom {[n]}{k_i}$ and
$\mu ({\cal F}_i) > \frac {a_ik_i}{n}$
for all $i\in [s]$ cross contain an $s$-matching.
We say ${\bf a} = (a_1,\ldots, a_s) \in {\mathbb R}^s$
is \emph{forcing} if it is forcing
for $(n,k_1,\ldots, k_s)$ whenever $n\geq \sum _{i\in [s]}k_i$
and \emph{exactly forcing} if it is forcing
for $(n,k_1,\ldots, k_s)$ whenever $n = \sum _{i\in [s]}k_i$.
Any forcing sequence is clearly exactly forcing;
we establish the converse.

\begin{lem} \label{lem: forcing sequences}
A sequence ${\bf a} \in {\mathbb R}^s$
is forcing if and only if it is exactly forcing.
\end{lem}

We require the following compression operators.
Given distinct $i,j \in [n]$ and $F\subset [n]$, we let
	\begin{align*}
		C_{i,j}(F) :=
			\begin{cases}
               (F\sm \{j\}) \cup \{i\}   & \mbox {if }
               j\in F, i \notin F;\\
               F              & \mbox{otherwise.}
             \end{cases}
	\end{align*}
Given ${\cal F} \subset \{0,1\}^n$, we let
$C_{i,j}({\cal F}) = \{C_{i,j}(F): F \in {\cal F}\} \cup
\{F \in {\cal F}: C_{i,j}(F) \in {\cal F}\}$.
We say ${\cal F}$ is $C_{i,j}$-compressed
if $C_{i,j}({\cal F}) = {\cal F}$.

\begin{proof}[Proof of Lemma \ref{lem: forcing sequences}]
A forcing sequence is clearly exactly forcing,
so it remains to prove the converse.
We argue by induction on $s$; the base case $s=1$ is clear.
Suppose that ${\bf a} \in {\mathbb R}^s$ is exactly forcing.
We fix $k_1,\ldots, k_s \geq 0$ and show by induction
on $n\geq \sum _{i\in [s]} 	k_i$ that ${\bf a}$
is forcing for $(n,k_1,\ldots, k_s)$,
i.e.\ any families ${\cal F}_1,\ldots , {\cal F}_s$
with  ${\cal F}_i \subset \tbinom {[n]}{k_i}$ and
$\mu ({\cal F}_i) > \frac {a_ik_i}{n}$
for all $i\in [s]$ cross contain an $s$-matching.
The base case $n=\sum _{i\in [s]} 	k_i$
holds as ${\bf a}$ is exactly forcing.

First suppose $k_i = 0$ for some $i\in [s]$;
without loss of generality $i = s$.
Then ${\bf a}' = (a_1,\ldots, a_{s-1})$
is exactly forcing, and so forcing by induction on $s$.
Thus ${\cal F}_1,\ldots, {\cal F}_{s-1}$
cross contain an $(s-1)$-matching.
Combined with $\es \in {\cal F}_s$
we find a cross $s$-matching
in ${\cal F}_1,\ldots, {\cal F}_s$, as required.
	
We may now assume $k_i \geq 1$ for all $i\in [s]$.
We suppose for contradiction that
${\cal F}_1,\ldots, {\cal F}_s$
do not cross contain an $s$-matching.
Let ${\cal G}_1,\ldots, {\cal G}_s$ be obtained
from ${\cal F}_1, \ldots, {\cal F}_s$
by successively applying the compression operators
$C_{1,n}, C_{2,n},\ldots, C_{n-1,n}$.
As is well-known (e.g.\ see \cite[Lemma 2.1 (iii)]{huang2012size}),
${\cal G}_1,\ldots, {\cal G}_s$
do not cross contain an $s$-matching
and are $C_{j,n}$-compressed for all $j\in [n-1]$.
For each $i\in [s]$ let
		\begin{align*}
			{\cal G}_i(n) &:= \big \{ A \subset [n-1]: A \cup \{n\} \in
			{\cal G}_i \big \} \subset \tbinom {[n-1]}{k_i-1};\\ \quad
			{\cal G}_i({\overline n}) &:= \big \{ A \subset [n-1]: A \in
			{\cal G}_i \big \} \subset \tbinom {[n-1]}{k_i}.
		\end{align*}
We now claim that if $I \subset [s]$ then
$\{ {\cal H}_i\}_{i\in [s]}$ are cross free of an $s$-matching,
where ${\cal H}_i = {\cal G}_i(n)$ for $i\in I$
and ${\cal H}_i = {\cal G}_i({\overline n })$ for $i\notin I$.
For contradiction, suppose $\{A_i\}_{i\in [s]}$
is such a cross matching in $\{ {\cal H}_i\} _{i\in [s]}$.
Then $A_i \cup \{n\} \in {\cal G}_i$ for all $i\in I$
and $A_i \in {\cal G}_i$ for $i\notin I$.
However, as ${\cal G}_i$ is $C_{j,n}$-compressed for all $j\in [n-1]$
and $n > \sum _{i\in [s]}k_i$, there are distinct $j_i \in [n] \sm
\big ( \cup _{i\in [s]}A_i \big )$ for all $i\in I$
such that $A_i \cup \{j_i\} \in {\cal G}_i$.
Then $\{A_i \cup \{j_i\}\}_{i\in I} \cup \{A_i\}_{i\in [s]\sm I}$
is a cross $s$-matching in $\{ {\cal G}_i\}_{i\in [s]}$, a contradiction.
Thus the claim holds.
	
By induction on $n$, it now suffices to show that
for each $i\in [s]$ either
$\mu ({\cal G}_i(n)) > a_i(k_i-1)/(n-1)$
or $\mu ({\cal G}_i({\overline n}) ) > a_i k_i/(n-1)$;
indeed, we then obtain the required contradiction by
setting $I = \{i\in [s]: \mu ({\cal G}_i(n)) > a_i(k_i-1)/(n-1)\}$
in the above claim. But this is clear, as otherwise
		\begin{equation*}
			\frac {a_i k_i}{n} < \mu ({\cal G}_i) =
			\Big ( \frac {n-k_i}{n} \Big )
			\mu ({\cal G}_i({\overline n}))
			+ \Big ( \frac {k_i}{n} \Big ) \mu
			({\cal G}_i (n)) \leq \Big ( \frac {n-k_i}{n} \Big ) \Big ( \frac {a_ik_i}{n-1} \Big ) + \Big ( \frac {k_i}{n} \Big ) \Big ( \frac {a_i(k_i-1)}{n-1} \Big ) = \frac {a_ik_i}{n},
		\end{equation*}
	a contradiction. This completes the proof.
\end{proof}

We conclude this subsection by deducing
Lemmas \ref{lem:matchlarge} and \ref{lem:matchlargelimit}.

\begin{proof}[Proof of Lemma \ref{lem:matchlarge}]
By Lemma \ref{lem: forcing sequences}
it suffices to prove the statement
under the assumption $n = \sum _{i\in [s]} k_i$.
Note first that if $n=0$ then ${\cal F}_i=\{\es\}$
for all $i \in [s]$ which clearly
cross contain an $s$-matching.
Thus we may assume $n>0$.
For any $i \in [m]$ we have
$2k_im/n < \mu ( {\cal F}_i ) \leq 1$,
so $k_i < n/2m$, and similarly
$k_i < n/2s$ for $i\in [m+1,s]$. But now
$n = \sum _{i\in [s]} k_i
< m \cdot n/2m + (s-m) \cdot n/2s < n$
is a contradiction.
 \end{proof}

\begin{proof}[Proof of Lemma \ref{lem:matchlargelimit}]
Let $N^{-1} \ll \eps \ll \min _{i\in [s]} p_i$ and
${\cal G}_i = {\cal F}_i \times \{0,1\}^{[N]\sm [n]}
\sub \{0,1\}^N$ for each $i \in [s]$.
Then each $\mu _{p_i}({\cal G}_i) = \mu _{p_i}({\cal F}_i)$.
Writing $I_i = \big [(1-\eps )Np_i, (1+\eps )Np_i \big ]$,
by Lemma \ref{conc} each
$\mu_{p_i} \big( \cup_{k \notin I_i} \tbinom{[N]}{k} \big) < \eps$,
so there are $k_i \in I_i$ such that each
$\mu \big( {\cal G}_i \cap \tbinom{[N]}{k_i} \big)
> \mu _{p_i}({\cal F}_i) - \eps $,
which is at least $2mk_i /N$ for $i\in [m]$
and $2sk_i /N$ for $i\in [m+1,s]$.	
The result now follows from Lemma \ref{lem:matchlarge}.
\end{proof}

\subsection{Capturability}

In this subsection we conclude this section by proving its main lemma
on cross matchings in uncapturable families. The idea of the proof
is to take suitable restrictions that boost the measure of the families
so that we can apply the extremal result from the previous subsection.
However, uncapturability is not preserved by restrictions,
so we first upgrade to globalness, which is preserved by restrictions.
We also pass from the setting of uniform families to that of biased measures,
which allows us to apply our sharp threshold result, and also has the
technical advantage that we do not need to assume any lower bound
on the uniformity of our families.

\begin{proof}[Proof of Lemma \ref{lem:matchuncap}]
Let $C \gg C_1 \gg C_2 \gg 1$ and $\f_i \sub \tbinom{[n]}{k_i}$
with $k_i \le n/Cs$ for $i \in [s]$.
Suppose $\f_i$ is $(C_1 m, mk_i/C_2 n)$-uncapturable for $i \in [m]$
and $\mu(\f_i) > C_1 sk_i/n$ for $i>m$. We need to show that
$\f_1,\dots,\f_s$ cross contain a matching.

We start by upgrading uncapturability
to globalness and moving to biased measures.
By Lemma \ref{lem:upgrade1'} with $r=C_1$ and $\bB = C_1^{-2}$
there are pairwise disjoint $S_1,\dots,S_m$
with each $|S_i| \le r$ such that,
setting $\g_i = (\f_i^\ua)^{S_i}_S$ where $S = \bigcup_i S_i$
and $p_i = k_i/(n-|S|)$,
whenever $\mu_{p_i}(\g_i) < C_1^{-2}$ we have
$S_i = \es$ and $\g_i$ is $(C_1,2C_1^{-2})$-global
with $\mu_{p_i}(\g_i) > mk_i/4C_2 n > mp_i/5C_2$.
We note by Lemma \ref{lem:globaluncap'} that
$\g_i$ is $(a,mp_i/10C_2)$-uncapturable,
where $a = (mp_i/5C_2)/(4p_iC_1^{-2}) > C_1 m$.

Next we will choose pairwise disjoint $R_1,\dots,R_m \sub [n] \sm S$
with each $|R_i| < C_1/8$, write $R_{{<}j} = \bigcup_{i<j} R_i$,
and define families $\g_i^j$ by
$\g_i^j = (\g_i)^\es_{R_{{<}j}}$ for $i \ge j$
or $\g_i^j = (\g_i)^{R_i}_{R_{{<}j}}$ for $i < j$.

We claim that we can choose each $R_i$
to ensure $\mu_{2p_i}(\g_i^i) \ge 7mp_i$.
To see this, first note that $\g_i^{i-1} = (\g_i)^\es_{R_{{<}i}}$
has $\mu_{p_i}(\g_i^{i-1}) \ge mp_i/10C_2$ by uncapturability.
If $\mu_{p_i}(\g_i^{i-1}) \ge 7mp_i$ we let $R_i=\es$ to obtain
$\mu_{2p_i}(\g_i^i) = \mu_{2p_i}(\g_i^{i-1})
\ge \mu_{p_i}(\g_i^{i-1}) \ge 7mp_i$.
Otherwise, as $mp_i < 2C^{-1} \ll C_1^{-1} \ll C_2^{-1}$
we can apply Theorem \ref{thm:sharpset}.1
with $\eps^{-1} = 70C_2$ and $r=C_1/8$ to choose $R_i$ with $|R_i | \leq r$
so that $\g_i^i = (\g_i^{i-1})^{R_i}_{R_i}$ has
$\mu_{2p_i}(\g_i^i) > \mu_{p_i}(\g_i^{i-1})/\eps \ge 7mp_i$.
Either way the claim holds.

By Lemma \ref{lem:globalrestrict'}
each $\g_i^i$ with $i \in [m]$ is $(C_1/2,4C_1^{-2})$-global,
so  $\g_i^m = (\g_i^i)^\es_{\bigcup_{j>i} R_j}$ has $\mu_{2p_i}(\g_i^m)
\ge \mu_{2p_i}(\g_i^i) - m(C_1/8)  \cdot 4C_1^{-2} \cdot 2p_i \ge 3m(2p_i)$.
For $i>m$ we have $\mu(\f_i) > C_1 sk_i/n$,
so $\mu_{p_i}(\g_i^i) > \mu_{p_i}(\f_i)/4 - m(C_1/8) p_i > 3sp_i$.
By Lemma \ref{lem:matchlargelimit},
$\g_1^m,\dots,\g_s^m$ cross contain a matching;
hence so do $\f_1,\dots,\f_s$.
\end{proof}

\section{Shadows and embeddings} \label{sec:shadow}

In this section we will complete the proof
of our junta approximation theorem
by implementing the strategy described above
of finding embeddings in fat shadows.
We start in the first subsection
by defining and analysing fat shadows.
In the second subsection we find shadow embeddings.
We then conclude in the final subsection with lifted embeddings
(using the lifting result from the previous section)
that prove Lemmas \ref{lem:embedlarge} and \ref{lem:embeduncap},
thus proving Theorem \ref{thm:junta'}.

\subsection{Fat shadows}

In this subsection we present various lower bounds
on the density of fat shadows, defined as follows.

\begin{defn}
The $c$-fat $r$-shadow of $\f\subset \tbinom{[n]}{k}$ is
$\pl^r_c \f := \{ A \in \tbinom{[n]}{r} : \mu(\f^A_A) \ge c\}$.

The $c$-fat shadow of $\f$ is
$\pl_c \f := \bigcup_{r \le k} \pl^r_c \f$.
\end{defn}

The following simple `Markov' bound is useful
when $\f$ is nearly complete.

\begin{lem} \label{lem:markov}
If $\mu(\f) \ge 1-cc'$ then
$\mu(\pl^r_{1-c} \f) \ge 1-c'$.
\end{lem}

\begin{proof}
Consider uniformly random $A \sub B \sub [n]$
with $|A|=r$ and $|B|=k$. For any $A \notin \pl^r_{1-c} \f$
we have $\mb{P}(B \notin \f \mid A) \ge c$,
so $cc' \ge \mb{P}(B \notin \f)
\ge c \cdot \mb{P}(A \notin \pl^r_{1-c} \f)
= c(1-\mu(\pl^r_{1-c} \f))$.
\end{proof}

Another bound is given
the following Fairness Proposition
of Keller and Lifshitz \cite{keller2017junta}.

\begin{prop}[Fairness Proposition] \label{prop:Fair}
Let $C \gg r/\eps$ and $\f\subset \tbinom{[n]}{k}$
with $k \ge r$ and $\mu\left(\f\right)\ge e^{-k/C}$.
For $c=(1-\eps)\mu(\f)$ we have
$\mu(\pl^r_c \f ) \ge 1-\eps$.
\end{prop}

When the above bounds are not applicable we rely on the following
lemma, whose proof will occupy the remainder of this subsection.

\begin{lem} \label{lem:fat}
Let $\f \subset \tbinom{[n]}{k}$, $r < \ell \le k$
and $\h = \{B \in \tbinom{[n]}{\ell}: \pl^r B \sub \pl^r_c \f \}$,
where $c = \mu(\f)/2\tbinom{\ell}{r}$.
Then $\mu(\h) \ge \mu(\f)/2$.
Thus $\mu(\pl^r_c \f) \ge (\mu(\f)/2)^{r/\ell}$.
Furthermore, if $G \in \g'(r,s,\DD)$, $C \gg r\DD$
and $\pl^r_c \f$ is $G$-free then $\mu(\pl^r_c \f) \ge
 \big( (\mu(\f)/2-(s/n)^{\ell/C}) n/s\ell^2 \big)^{r/(\ell-1)}$.
\end{lem}

We require several further lemmas
for the proof of Lemma \ref{lem:fat}.
We start by stating a consequence
of the Lov\'asz form \cite{Lovasz2}
of the  Kruskal--Katona theorem
\cite{katona2009theorem, kruskal1963number}.

\begin{lem} \label{lem:KK}
If $1 \leq \ell \leq k \leq n$ and
${\cal A} \subset \tbinom {[n]}{k}$
then $\mu ( \pl^\ell({\cal A}) )
\geq \mu ({\cal A})^{\ell /k}$.
\end{lem}

\begin{proof}
We define $\bB \in [0,1]$ by
$|{\cal A}| = \tbinom{\bB n}{k}$, so that
$\mu({\cal A}) = \prod_{i=0}^{k-1} (\bB-i/n)$.
By the Lov\'asz form of Kruskal--Katona, we have
$| \pl^\ell {\cal A}| \ge \tbinom{\bB n}{\ell}$,
so $\mu ( \pl^\ell({\cal A}) )^k
\ge \prod_{i=0}^{\ell-1} (\bB-i/n)^k
\ge \mu({\cal A})^\ell$.
\end{proof}

Next we require an estimate on the Tur\'an numbers of
$r$-partite $r$-graphs, which follows from
\cite[Theorem 2]{CFS} due to Conlon, Fox and Sudakov.
(Recall that $\g'(r,s,\DD)$ is the family
of $r$-partite $r$-graphs with $s$ edges
and maximum degree $\DD$.)

\begin{thm} \label{thm:CFS}
Let $F \in \g'(r,s,\DD)$  and $C \gg r\DD$.
Then any $F$-free $\h \sub \tbinom{[n]}{r}$ with $n > Cs$
has $\mu(\h) < (s/n)^{1/C}$.
\end{thm}

We note that the following lemma is immediate
from Theorem \ref{thm:CFS} and Lemma \ref{lem:KK}.

\begin{lem} \label{lem:DRC}
Let $G \in \g'(r,s,\DD)$, $C \gg r\DD$, $C \le k \le n/Cs$
and  $\f\subset \tbinom{[n]}{k}$.
If $\pl^r \f$ is $G$-free then $\mu(\f)\le (s/n)^{k/C}$.
\end{lem}

Our next lemma is an adaptation of one due to
Kostochka, Mubayi and Verstra{\"e}te \cite{kostochka2015turan}.

\begin{lem} \label{lem:kmv}
Suppose $G \in \g'(r,s,\DD)$, $C \gg r\DD$
and $\f$ is a $G^+$-free $k$-graph on $[n]$.
Then $\mu(\pl \f) \ge (\mu(\f)-(s/n)^{k/C}) n/sk^2$.
\end{lem}

\begin{proof}
We define $\g \sub \f$ by starting with $\g=\f$
and then repeating the following procedure:
if there is any $A \in \pl\g$ with $|\g^A_A| \le ks$
then remove from $\g$ all edges containing $A$.
This terminates with some $\g$
such that $|\g^A_A|>rs$ for all $A \in \pl\g$
and $|\g| \ge |\f|-ks|\pl\f|$, so
$\mu(\pl\f) \ge (\mu(\f)-\mu(\g))n/sk^2$.

We will now show that $\pl_r \g$ is $G$-free,
which will complete the proof due to Lemma \ref{lem:DRC}.
To see this, we suppose that $\phi(G)$ is a copy of $G$ in $\pl_r\g$
and will obtain a contradiction by finding a copy of $G^+$ in $\g$.
To do so, we start by fixing for each edge $A$ of $G$
an edge $e_A$ of $\g$ containing $\phi(A)$.
Then we repeat the following procedure:
while some $e_A$ contains some $\phi(x)$ with $x \notin A$,
replace $e_A$ by some edge $(e_A \sm \{\phi(x)\}) \cup \{v\}$
with $v \notin \im\phi$. As $|\g^A_A|>ks$ for all $A \in \pl\g$
we can always choose $v$ as required. The procedure terminates
with a copy of $G^+$, so the proof is complete.
\end{proof}

We conclude this subsection with the proof of its main lemma.

\begin{proof}[Proof of Lemma \ref{lem:fat}]
Consider uniformly random $(A,B,C)$
with $C \subset B \subset A \subset [n]$
and $|C|=r$, $|B|=\ell$, $|A|=k$.
Write $p = \mb{P}(A \in \f, C \notin \pl^r_c \f)$
and $q = \mb{P}(A \in \f, B \notin \h)$.

For any $C \notin \pl^r_c \f$ we have
$\mb{P}(A \in \f \mid C) = \mu(\f^C_C) \le c$, so $p \le c$.
On the other hand, $p \ge q \tbinom{\ell}{r}^{-1}$,
as for any $A \in \f$ and $B \notin \h$ we have
$\mb{P}(C \notin \pl^r_c \f \mid A,B) \ge \tbinom{\ell}{r}^{-1}$.
We deduce $q \le \tbinom{\ell}{r} c = \mu(\f)/2$.

Thus $\mu(\h) = \mb{P}(B \in \h)
\ge \mb{P}(A \in \f) - q \ge \mu(\f)/2$.

As $\pl^r \h \subset \pl^r_c \f$, Lemma \ref{lem:KK} gives
$\mu(\pl^r_c \f) \ge (\mu(\f)/2)^{r/\ell}$.

Now suppose $G \in \g'(r,s,\DD)$ and  $\pl^r_c \f$ is $G$-free.
Then $\h$ is $G^+$-free, so Lemma \ref{lem:kmv} gives
$\mu(\pl \h) \ge (\mu(\h)-(s/n)^{\ell/C}) n/s\ell^2$.
As $\pl^r \pl \h \subset \pl^r_c \f$,
Lemma \ref{lem:KK} gives the required bound.
\end{proof}

\subsection{Shadow embeddings}

The following lemma implements a simple greedy algorithm
for cross embedding any bounded degree $r$-graph
in a collection of nearly complete $r$-graphs
(more generally, we also allow smaller edges).

\begin{lem} \label{lem:nearcomplete}
Let $0 < \eta \ll (r\DD)^{-1}$ and
$G = \{e_1,\dots,e_s\}$ be a hypergraph
of maximum degree $\DD$ with each $|e_i|=r_i \le r$.
Suppose for each $i \in [s]$ that
$\g_i$ is an $r_i$-graph on $[n]$,
where  $n \ge 2rs$ and $\mu(\g_i)>1-\eta$.
Then $\g_1,\dots,\g_s$ cross contain $G$.
\end{lem}

\begin{proof}
Write $V(G) = \{v_1,\dots,v_m\}$.
We may assume that $G$ has no isolated vertices,
so $m \le \sum_i d_G(v_i) \le rs \le n/2$.
We will construct an injection $\phi:V(G) \to [n]$
such that each $\phi(e_j) \in \g_j$.
To do so, we define $\phi$ sequentially so that,
for each $0 \le t \le m$ the definition of $\phi$
on $V_t := \{ v_i: i \le t\}$ is \emph{$t$-good},
meaning that for each edge $e_j$ we have
\begin{equation} \label{eq:notbad}
\phi(e_j \cap V_t) \in \pl_{c_{jt}} \g_j,
\text { where }
c_{jt} = 1 - \eta (2\DD)^{|e_j \cap V_t|}.
\end{equation}
Note that \eqref{eq:notbad} holds
whenever $e_j \cap V_t = \es$,
as $\mu(\g_j) > 1- \eta $;
in particular, \eqref{eq:notbad} holds when $t=0$.

It remains to show for any $0 \le t < m$ that we can
extend any $t$-good embedding $\phi$ to a $(t+1)$-good embedding.
To see this, first note that we only need to check
\eqref{eq:notbad} when $e_j$ is one of at most $\DD$
edges containing $v_{t+1}$. Fix any such edge $e_j$,
let $f = \phi(e_j \cap V_t)$,
and let $B_j$ be the set of $x \in [n]$
such that choosing $\phi(v_{t+1})=x$
would give $\phi(e_j \cap V_{t+1})
= f \cup \{x\} \notin \pl_{c_{j(t+1)}} \g_j$. Then
\[ |B_j| \eta (2\DD)^{|f|+1} \le \sum_{x \in B}
\left( 1-\mu \big( (\g_j)^{f \cup \{x\}}_{f \cup \{x\}}
\big) \right)
\le n(1-\mu((\g_j)^f_f)) < n \eta (2\DD)^{|f|},\]
so $|B_j| < n/2\DD$. Summing over at most $\DD$ choices
of $j$ forbids fewer than $n/2$ choices of $x$.
The requirement that $\phi$ be injective also forbids
fewer than $n/2$ vertices, so we can extend $\phi$ as required.
\end{proof}

\subsection{Lifted embeddings}

We conclude this section by proving
the two embedding lemmas assumed above,
thus completing the proof of Theorem \ref{thm:junta'}.

\begin{proof}[Proof of Lemma \ref{lem:embedlarge}]
Suppose $n, s, k_1,\ldots, k_s \in {\mathbb N}$
with $C \leq k_i \leq \frac {n}{Cs}$ for all $i\in [s]$,
and $\f _i \subset\tbinom{\left[n\right]}{k_i}$
with each $\mu(\f_{i}) \ge e^{-{k_i}/{C}}+ C sk_i/n$.
Let $\eta$ be as in Lemma \ref{lem:nearcomplete}.
We can assume $C$ is large enough so that Proposition \ref{prop:Fair}
gives $\mu\left({\cal G}_i\right)\ge 1-\eta$  for each $i\in [s]$,
where ${\cal G}_i$ is the $r$-graph on $[n]$ consisting
of all $e \in \tbinom{[n]}{r}$ with
$\mu((\f_{i} )_e^e ) \ge Cs k_i/2n$.
By Lemma \ref{lem:nearcomplete} we can find
$R_1,\ldots ,R_s$ forming a copy of $G$
with $R_{i} \in {\cal G}_i$ for all $i\in [s]$.
Let $R=R_{1}\cup\cdots\cup R_{s}$.
By the union bound,
each $\mu \big ( (\f _i)^{R_i}_{R} \big ) \geq
\mu \big ( (\f _i)^{R_i}_{R_i} \big )
- |R| k_i/n \geq Cs k_i/4n$ for $C \ge 8$,
so Lemma \ref{lem:hls}
gives a cross matching $E_1,\ldots, E_s$
in $(\f _1)^{R_1}_R,\ldots ,(\f _s)^{R_s}_R$.
Now $\f_{1},\ldots,\f_{s}$ cross contain a copy of $G^+$
with edges $R_1 \cup E_1, \ldots, R_s \cup E_s$.
\end{proof}

\begin{proof}[Proof of Lemma \ref{lem:embeduncap}]
Let $G \in \g(r,s,\DD)$ and $C \gg C_1 \gg C_2 \gg r\DD$.
Suppose for a contradiction that
$\f \subset\tbinom{[n]}{k}$ with $C \leq k \leq n/Cs$
is $(C_1 s, sk/C_2 n)$-uncapturable but $G^+$-free.

Let $\b$ be a maximal collection of pairwise disjoint sets
where each $B \in \b$ has $|B| \le r+1$
and $\mu(\f^B_B) > \bB := e^{-k/C_1} + C_1 sk/n$.
We claim that $|\b|<s$. To see this, suppose for a contradiction
that we have distinct $B_1,\dots,B_s$ in $\b$.
Let $B = \bigcup_{i=1}^s B_i$ and
$\f_i = \f^{B_i}_B$ for $i \in [s]$.
Then each $\mu(\f_i) > \bB - |B|k/n  > e^{-k/C_1} + C_1 sk/2n$.
Now Lemma \ref{lem:embedlarge} gives a cross copy of $G^+$
in $\f_1,\dots,\f_s$, contradicting $\f$ being $G^+$-free,
so $|\b|<s$, as claimed.

Now let $\g = \f^\es_B$ with $B = \bigcup \b$.
Then $\g$ is $(r+1,2\bB)$-global by definition of $\b$
and $\mu(\g) > sk/C_2 n$ by uncapturability of $\f$.
Let $\h = \{B \in \tbinom{[n]}{C_2}: \pl^r B \sub \pl^r_c \g \}$,
where $c = \mu(\g)/2\tbinom{C_2}{r} > sk/nC_2^{2r}$.
We have $\mu(\h) \ge \mu(\g)/2$ by Lemma \ref{lem:fat}.
We will show that $\pl^r \h$ is $G$-free. Then Lemma \ref{lem:DRC}
with $C_2/2 \gg r\DD$ in place of $C$ will give the contradiction
$sk/C_2 n < \mu(\g) \le 2\mu(\h) \le (s/n)^2$.

It remains to show that $\pl^r \h$ is $G$-free.
Suppose for a contradiction
that $A_1,\dots,A_s$ is a copy of $G$ in $\pl^r \h$.
Let $A = \bigcup_{i=1}^s A_i$ and
$\g_i = \g^{A_i}_A$ for $i \in [s]$.
Then each $\g_i$ is $(1,4\bB)$-global
by Lemma \ref{lem:globalrestrict}
with $\mu(\g_i) > c - |A| \cdot 2\bB k/n > c/2$.
Now each $\g_i$ is $(C_1 s,c/4)$-uncapturable
by Lemma \ref{lem:globaluncap},
so $\g_1,\dots,\g_s$ cross contain a matching
by Lemma \ref{lem:matchuncap} with $m=s$.
However, this contradicts $\f$ being $G^+$-free.
\end{proof}

\section{Refined junta approximation} \label{sec:refine}

In this final section of the part
we will prove Theorem \ref{thm:juntarefined},
our refined junta approximation result,
which will play a key role in the proofs
of our results in the next part.
We start in the first subsection by setting out
the strategy of the proof and implementing it
assuming an embedding lemma, whose proof will
then occupy the remainder of the section.

\subsection{Strategy}

Our embedding strategy considers a setup below
that blends the two embedding strategies used in
the proof of Theorem \ref{thm:junta}:
it has elements of Lemma \ref{lem:embedlarge}
(mapping a crosscut to a junta)
and of Lemma \ref{lem:embeduncap}
(embedding in the fat shadow
and lifting via uncapturability).

\begin{setup} \label{set:embed}
Let $G \in \g'(r,s,\DD)$.
Let $S$ be a crosscut in $G^+(r+1)$ with $|S|=\sS:=\cc{G}$.
Suppose $S_1 \sub S$ with $|S_1|=\sS_1 \le \sS$
and $\{ G^x_x: x \in S_1 \}$ vertex disjoint.
Let $H_1,\dots,H_{\sS_1}$ be the inclusive links
$G*x = \{e \in G: x \in e\}$ for $x \in S_1$
and $H_{\sS_1+1},\dots,H_\sS$ be the exclusive links
$G^x_x$ for $x \in S \sm S_1$.
Let $V_1 = \bigcup_{i=1}^{\sS_1} V(H_i)$ and suppose
$\{j: V(H_j) \cap V_1 \ne \es \} = [\sS_2]$.
Let $H'_i = H_i$ for $i \in [\sS_1]$ and
$H'_i = \{ e \cap V_1: e \in H_i \}$ for $i \in [\sS_1+1,\sS_2]$.
\end{setup}

We note that $\sS \le s \le \DD\sS$. To use Setup \ref{set:embed}
for embedding $G^+$ in $\f \subset\tbinom{[n]}{k}$
it suffices to find $J = \{j_{\sS_1+1},\dots,j_{\sS}\} \sub [n]$ and
a cross copy of $H_1^+,\dots,H_{\sS}^+$ in $\f_1,\dots,\f_{\sS}$,
where $\f_i = \f^\es_J$ for $i \in [\sS_1]$
and $\f_i = \f^{j_i}_J$ for $i \in [\sS_1+1,\sS]$.
This will be achieved by the following lemma.

\begin{lem} \label{lem:embedmix}
Let $C \gg C_1 \gg \tT^{-1} \gg \eps^{-1} \gg r\DD$ and $C < k < n/Cs$.
Let $G,H_1,\dots,H_\sS$ be as in Setup \ref{set:embed} with $\sS_1 \le \tT \sS$.
Let $\f_i \subset\tbinom{[n]}{k}$ for $i \in [\sS_1]$ and
$\f_i \subset\tbinom{[n]}{k-1}$ for $i \in [\sS_1+1,\sS]$.
Suppose $\f_i$ is $(C_1 \sS_1, \eps \sS_1 k/n)$-uncapturable for $i \in [\sS_1]$,
that $\mu(\f_i) \ge 1-\tT$ for $i \in [\sS_1+1,\sS_2]$, and
$\mu(\f_i) \ge \bB := e^{-k/C_1} + C_1 sk/n$ for $i \in [\sS_2+1,\sS]$.
Then $\f_1,\dots,\f_\sS$ cross contain $H_1^+,\dots,H_\sS^+$.
\end{lem}

Next we deduce Theorem \ref{thm:juntarefined} from Lemma \ref{lem:embedmix}.

\begin{proof}[Proof of Theorem \ref{thm:juntarefined}]
Let $G \in \g(r,s,\DD)$ with $\cc{G}=\sS$ and
$C \gg C_1 \gg \tT^{-1} \gg \dD^{-1} \gg \eps^{-1} \gg r\DD$.
Suppose $\f \sub \tbinom {[n]}{k}$
with $C \leq k \leq n/Cs$ is $G^+$-free
with $|\f| > |\s_{n,k,\sS-1}| - \dD \tbinom{n-1}{k-1}$.
We need to find $J \in \tbinom{[n]}{\sS-1}$ with
$|\f^\es_J| \leq \eps \tbinom{n-1}{k-1}$.

As in the proof of Theorem \ref{thm:junta} we let
$J = \{ i \in [n]: \mu(\f^i_i) \ge \bB \}$,
where $\bB := e^{-k/C_1} + C_1 sk/n$.
We recall that $|J| \leq \sS-1$ and
$\f^\es_J$ is $(a,\mu(\f^\es_J)/2)$-uncapturable
with $a = \mu(\f^\es_J) n/4k\bB$.
Replacing `$\eps$' in that proof by $.1\tT^2$ we obtain
$|\f^\es_J| \leq .1\tT^2 |\s_{n,k,\sS-1}|
\le .2\tT^2 (\sS-1) \tbinom{n-1}{k-1}$.
We may assume $\sS \ge 2\tT^{-1}$, otherwise
$|\f^\es_J| \leq \tT \tbinom{n-1}{k-1}$.
As $|\f^\es_J| \ge |\f| - |\s_{n,k,J}|
\ge (.9(\sS-1-|J|) - \dD) \tbinom{n-1}{k-1}$
we deduce $|J| > (1-.3\tT^2)(\sS-1)$, so
$1 \le \sS_1 := \sS - |J| \le 1 + .3\tT^2 \sS \le \tT \sS$.

Now we let $S,S_1,H_1,\dots,H_\sS$ be as in Setup \ref{set:embed},
where we can greedily choose $S_1 \sub S$ with $|S_1|=\sS_1$
such that $\{ G^x_x: x \in S_1 \}$ are vertex disjoint,
as any partial choice of $S_1$ forbids at most
$\sS_1 (\DD r)^2 < \sS$ vertices of $S$.
We write $J = \{j_{\sS_1+1},\dots,j_{\sS}\}$,
let $\f_i = \f^\es_J$ for $i \in [\sS_1]$
and $\f_i = \f^{j_i}_J$ for $i \in [\sS_1+1,\sS]$,
where we can assume $|\f_{\sS_1+1}| \ge \dots \ge |\f_\sS|$.
We note that $\mu(\f_{\sS_2}) > 1-\tT$,
as otherwise we would have the contradiction $|\f| < |\f^\es_J| +
\big( \sS_2-\sS_1 + (\sS-\sS_2)(1-\tT) \big) \tbinom{n-1}{k-1}
< \big( (1+.2\tT^2)\sS - \sS_1 - \tT(\sS-\sS_2) \big) \tbinom{n-1}{k-1}
< |\s_{n,k,\sS-1}| - \dD \tbinom{n-1}{k-1}$.

Now we must have $\mu(\f^\es_J) \le \eps \sS_1 k/n$;
otherwise $\f^\es_J$ is $(C_1 \sS_1, \eps \sS_1 k/2n)$-uncapturable,
so $\f_1,\dots,\f_\sS$ cross contain $H_1^+,\dots,H_\sS^+$
by Lemma \ref{lem:embedmix}, contradicting $\f$ being $G^+$-free.
As $|\f^\es_J| \ge |\f| - |\s_{n,k,J}|
\ge (.9(\sS_1-1) - \dD) \tbinom{n-1}{k-1}$
we deduce $.9(\sS_1-1) - \dD \le \eps \sS_1$,
so $\sS_1=1$ and $\mu(\f^\es_J) \le \eps k/n$.
\end{proof}

The remainder of the section will be devoted
to the proof of Lemma \ref{lem:embedmix}.
Similarly to the proofs of our previous embedding results
(Lemmas \ref{lem:embedlarge} and \ref{lem:embeduncap}),
the strategy will be to find shadow embeddings
and then lifting embeddings. However, there are further
technical challenges to overcome in the current setting,
particularly when the uniformity $k$ of our families is small,
when we need to `pause' the shadow embedding
after embedding $H'_i = H_i$ for $i \in [\sS_1]$,
then lift this part of the embedding,
then complete the shadow embedding,
and finally lift the remainder of the embedding.
The shadow embedding lemma will be presented
in the next subsection. The third subsection
contains further results on upgrading
uncapturability to globalness,
which we call `enhanced upgrading',
as they obtain globalness parameters that are
significantly stronger than one might expect,
and this will be a crucial technical ingredient of the proof.
In the fourth subsection we establish an improved lifting result
that allows for a much weaker uncapturability assumption
than that in Lemma \ref{lem:matchuncap}.
We conclude with the proof of Lemma \ref{lem:embedmix}
in the final subsection.

\subsection{Shadow embeddings}

Here we extend the argument used in Lemma \ref{lem:nearcomplete}
to prove the following lemma that will be applied
to show that the fat shadows of $\f_1,\dots,\f_\sS$
as in Lemma \ref{lem:embedmix} cross contain $H_1,\dots,H_\sS$.
Whereas before we were embedding into nearly complete hypergraphs,
now many of our hypergraphs will be quite sparse,
which makes the embedding more challenging:
the idea is to replace the naive greedy arguments
by Theorem \ref{thm:CFS}, here making key use of
our observation that we can assume $G$ is $r$-partite.

\begin{lem} \label{lem:embedfat}
Let $C \gg \eta^{-1} \gg K \gg r\DD$ and $0<\tT<\eta$.
Let $G,H_1,\dots,H_\sS$ be as in Setup \ref{set:embed}
and $\g_1,\dots,\g_\sS \subset\tbinom{[n]}{r}$ with $n > C\sS$.
Suppose $\mu(\g_i) \ge 1-\eta$ for $i \in [\sS_2+1,\sS]$,
$\mu(\g_i) \ge 1-\tT$ for $i \in [\sS_1+1,\sS_2]$
and $\mu(\g_i) \ge \tT^{1/2r} + n^{-1/K} + r\DD\sS_1/n$
 for $i \in [\sS_1]$. Let $c = 1 - \tT^{1/r}$.
Then $\pl_c \g_1, \dots, \pl_c \g_{\sS_2}$
cross contain $H'_1,\dots,H'_{\sS_2}$
and $\g_1,\dots,\g_\sS$
cross contain $H_1,\dots,H_\sS$.
\end{lem}

\begin{proof}
For each $i \in [\sS_1+1,\sS_2]$ we define
$\g^r_i, \dots, \g^0_i$ recursively by $\g^r_i = \g_i$ and
$\g^{j-1}_i = \pl^{j-1}_{1-\tT^{1/r}} \g^j_i$ for $j \in [r]$.
Clearly each $\g^j_i \sub \pl_{c_j} \g_i$
where $c_j = 1 - (r-j)\tT^{1/r}$.

We claim that each $\mu(\g^j_i) \ge 1-\tT^{j/r}$.
To see this, we argue by induction on $r-j$.
For $r-j=0$ we have $\mu(\g^r_i) \ge 1-\tT$ by assumption.
For the induction step, consider any $j \in [r]$
and uniformly random $A \sub B \sub [n]$
with $|A|=j-1$ and $|B|=j$.
Given any $A \notin \g^{j-1}_i$
we have $\mb{P}(B \notin \g^j_i) \ge \tT^{1/r}$,
so $1 - \mu(\g^j_i) \ge \tT^{1/r} (1 - \mu(\g^{j-1}_i))$.
The claim follows.

Next we will construct a cross embedding
$\phi$ of $H'_1,\dots,H'_{\sS_2}$ in
$\pl_c \g_1, \dots, \pl_c \g_{\sS_2}$.
We recall that $H'_i=H_i$ for $i \in [\sS_1]$
and all $H'_i$ are defined on $V_1$, which is
the disjoint union of $V(H_1),\dots,V(H_{\sS_1})$.
We proceed in $\sS_1$ steps,
defining $\phi$ on $V(H_t)$ at step $t$.
When $\phi$ has been defined on
$U_t := \bigcup_{i \le t} V(H_i)$,
we say $\phi$ is \emph{$t$-good} if
$\phi(e \cap U_t) \in \g^{|e \cap U_t|}_i$
for each $i \in [\sS_2]$ and $e \in \g_i$
with $e \cap U_t \ne \es$.

We note that if $\phi$ is $t$-good then
$\phi(H_i) \sub \g^r_i = \g_i = \pl_c \g_i$ for all $i \in [t]$
and if $\phi$ is $\sS_1$-good then
$\phi(H_i) \sub \pl_c \g_i$ for all $i \in [\sS_2]$.
As $\phi$ defined on $U_0=\es$ is trivially $0$-good,
it remains to show for any $t \in [\sS_1]$ that we can extend
any $(t-1)$-good $\phi$ to a $t$-good embedding.

For clarity of exposition, we start by showing the case $t=1$.
Obtain $\h_1$ from $\g_1$ by removing any edge $e$ such that
$f \notin \g^{|f|}_i$ for some $\es \ne f \subset e$
and $i \in [\sS_2]$ with $V(H_i) \cap V(H_1) \ne \es$.
There are at most $r\DD^2$ such $i$,
so by a union bound and the above claim we have
$\mu(\h_1) \ge \mu(\g_1) -  r\DD^2 2^r \tT^{1/r} > n^{-1/K}$.
We can assume that $G$ is $r$-partite,
so by Theorem \ref{thm:CFS} we can find an embedding $\phi'_1$
of $N_1 := \{e \in G: e \cap V(H_1) \ne \es \}$ in $\h_1$.
Now $\phi = \phi'\mid_{V(H_1)}$ is $1$-good.

Now we consider general $t \in [\sS_1]$.
Obtain $\h_t$ from $(\g_t)^\es_{\phi(U_{t-1})}$
by removing any edge $e$ such that $f \notin \g^{|f|}_i$
for some  $\es \ne f \sm \phi(A') \subset e$
where $A \in H_i$ with $V(H_i) \cap V(H_t) \ne \es$
and $A' = A \cap U_{t-1}$.
For any such non-empty $A'$, as $\phi$ is $(t-1)$-good
we have $\phi(A') \in \g^{|A'|}_i$,
so $\mu((\g^j_i)^{A'}_{A'}) \ge 1 - (j-|A'|) \tT^{1/r}$
for any $|A'| \le j \le r$. Thus a union bound gives
$\mu(\h_t) \ge \mu(\g_t) - |U_{t-1}|k/n - r\DD^2 2^r r\tT^{1/r} > n^{-1/K}$.
Now as in the case $t=1$ we obtain a $t$-good extension
by embedding $N_t := \{e \in G: e \cap V(H_t) \ne \es \}$
in $\h_t$ and restricting to $V(H_t)$.

Thus we have constructed a cross embedding
$\phi$ of $H'_1,\dots,H'_{\sS_2}$ in
$\pl_c \g_1, \dots, \pl_c \g_{\sS_2}$.
To complete the proof we extend $\phi$
to a cross embedding $H_1,\dots,H_\sS$
in $\g_1,\dots,\g_\sS$, which requires
$\phi(e \sm V_1) \in (\g_i)^{e \cap V_1}_{e \cap V_1}$
for all $e \in H_i$, $i \in [\sS_1+1,\sS]$;
this is possible by Lemma \ref{lem:nearcomplete}.
\end{proof}

\subsection{Enhanced upgrading}

This subsection provides further results
on upgrading uncapturability to globalness with enhanced parameters
that will be crucial in later proofs. We start by showing that
every family has a restriction that is global or large.

\begin{lem} \label{lem:focus}
Let $b,r \in \mb{N}$, $\aA>1$ and
$\f \subset \tbinom{[n]}{k}$ with $k \ge br$.
Then there is $B \sub [n]$ with $|B| \le br$
such that if $\mu(\f^B_B) < \aA^b \mu(\f)$
then $\f^B_B$ is $(r,\aA\mu(\f^B_B))$-global
with $\mu(\f^B_B) \ge \aA^{1_{B \ne \es}}\mu(\f)$.
\end{lem}

\begin{proof}
We consider $\f_0,\f_1,\dots$, where $\f_0=\f$, and if $i<b$
and $\f_i$ is not $(r,\aA\mu(\f_i))$-global then we let
$\f_{i+1} = (\f_i)^{B_i}_{B_i}$ so that $|B_i| \le r$
and $\mu(\f_{i+1}) > \aA \mu(\f_i)$. When this sequence
terminates at some $\f_t$ we let $B = \bigcup_{i \le t} B_i$.
Clearly $\f^B_B=\f_t$ has the required properties.
\end{proof}

By iterating the previous result we obtain the following upgrading lemma.

\begin{lem} \label{lem:upgrade2}
Suppose $b,r,m \in \mb{N}$ and for each $i \in [m]$ that $\aA_i>1$
and $\f_i \subset \tbinom{[n]}{k_i}$ with $rb \le k_i \le n/2rm\aA_i$
is $(rbm,\bB_i)$-uncapturable with $\aA_i^b \bB_i > 2rmk_i/n$.
Then there are disjoint $B_1,\dots,B_m$
with each $|B_i| \le rb$ such that,
setting $\g_i = (\f_i)^{B_i}_B$ where $B = \bigcup_i B_i$,
if $\mu(\g_i) < \aA_i^b \bB_i/2$ then
$\g_i$ is $(r,4\aA_i \mu(\g_i))$-global
with $\mu(\g_i) > \aA_i^{1_{B_i \ne \es}} \bB_i/2$.
\end{lem}

\begin{proof}
We will choose $B_1,\dots,B_m$ sequentially and define
$\f^0_i,\dots,\f^m_i$ for $i \in [m]$ by $\f^0_i = \f_i$,
$\f^i_j = (\f^{i-1}_j)^{\es}_{B_i}$ for $j \ne i$
and $\f^i_i = (\f^{i-1}_i)^{B_i}_{B_i}$.
At step $i$, we have $\mu(\f^{i-1}_i) \ge \bB_i$
by uncapturability of $\f_i$, so by Lemma \ref{lem:focus}
we can choose $B_i$ with $|B_i| \le rb$
such that if $\mu(\f^i_i) < \aA_i^b \mu(\f^{i-1}_i)$
then $\f^i_i$ is $(r,\aA\mu(\f^i_i))$-global
with $\mu(\f^i_i) \ge \aA_i^{1_{B_i \ne \es}} \bB_i$.
After step $m$, for any $i \in [m]$
we have $\g^m_i = \g_i = (\f_i)^{B_i}_B$.
If $\mu(\f^i_i) \ge \aA_i^b \mu(\f^{i-1}_i)$ then
$\mu(\g_i) \ge \aA_i^b \bB_i - rmk_i/n \ge \aA_i^b \bB_i/2$.
Otherwise, $\f^i_i$ is $(r,\aA_i\mu(\f^i_i))$-global
with $\mu(\f^i_i) \ge \aA_i^{1_{B_i \ne \es}}\mu(\f)$,
and $(n/2k_i\aA_i,\mu(\f^i_i)/2)$-uncapturable
by Lemma \ref{lem:globaluncap},
so $\mu(\g_i) > \mu(\f^i_i)/2 \ge \aA_i^{1_{B_i \ne \es}} \bB_i/2$,
and $\g_i$ is $(r,4\aA_i\mu(\g_i))$-global
by Lemma \ref{lem:globalrestrict}.
\end{proof}

For our final upgrading lemma we apply the previous one twice:
the idea is that the globalness from the first application provides
the second application with much better uncapturability.

\begin{lem} \label{lem:upgrade3}
Suppose $b,r,m \in \mb{N}$ and for each $i \in [m]$
that $\f_i \subset \tbinom{[n]}{k_i}$ with $rb \le k_i \le n/2rmb^2$
is $(2m,\bB_i)$-uncapturable with $\bB_i > 8rmk_i/bn$.
Then there are disjoint $B_1,\dots,B_m$
with each $|B_i| \le rb+2$ such that,
setting $\g_i = (\f_i)^{B_i}_B$ where $B = \bigcup_i B_i$,
if $\mu(\g_i) < 2^b \bB_i/8$ then
$\g_i$ is $(r,8\mu(\g_i))$-global
with $\mu(\g_i) > 2^{1_{B_i \ne \es}} \bB_i/8$.
\end{lem}

\begin{proof}
We start by applying Lemma \ref{lem:upgrade2}
with $(b,1,2)$ in place of $(\aA_i,r,b)$.
This gives disjoint $S_1,\dots,S_m$
with each $|S_i| \le 2$ such that,
setting $\h_i = (\f_i)^{S_i}_S$ where $S = \bigcup_i S_i$,
if $\mu(\h_i) < b^2 \bB_i/2$ then
$\h_i$ is $(1,4b\mu(\h_i))$-global
with $\mu(\h_i) > \bB_i/2$.

We claim that each $\h_i$ is $(rbm,\bB_i/4)$-uncapturable.
Indeed, this holds by a union bound
if $\mu(\h_i) \ge b^2 \bB_i/2$, as then
$\mu((\h_i)^\es_B) \ge \mu(\h_i) - |J|k_i/n \ge \bB_i/4$
whenever $|J| \le rbm$, as $\bB_i \ge 8rmk_i/bn$.
On the other hand, if $\h_i$ is $(1,4b\mu(\h_i))$-global
with $\mu(\h_i) > \bB_i/2$
then $\h_i$ is $(n/2bk_i,\mu(\h_i)/2)$-uncapturable
by Lemma \ref{lem:globaluncap},
so $(rbm,\bB_i/4)$-uncapturable, as $k_i \le n/2rmb^2$.

Now we can apply Lemma \ref{lem:upgrade2} again
to $\h_1,\dots,\h_m$ with $(2,r,b)$ in place of $(\aA_i,r,b)$.
This gives disjoint $S'_1,\dots,S'_m$
with each $|S'_i| \le rb$ such that,
setting $\g_i = (\h_i)^{S'_i}_{S'}$ where $S' = \bigcup_i S'_i$,
if $\mu(\g_i) < 2^b \bB_i/8$ then
$\g_i$ is $(r,8\mu(\g_i))$-global
with $\mu(\g_i) > 2^{1_{S'_i \ne \es}} \bB_i/8$.
Thus $B_i=S_i \cup S'_i$ for $i \in [m]$ are as required.
\end{proof}

\subsection{Refined capturability for matchings}

Here we prove the following sharper version of Lemma \ref{lem:matchuncap},
obtaining cross matchings under a much weaker uncapturability condition.

\begin{lem} \label{lem:matchuncap+}
Let $C \gg K \gg d \ge 1$ and $\f_i \sub \tbinom{[n]}{k_i}$
with $k \le k_i \le Kk$ for $i \in [s]$, where $2d \le k \le n/Cs$.
Suppose $\f_i$ is $(2dm, (2mk_i/n)^d )$-uncapturable for $i \in [m]$
and $\mu(\f_i) > 12(s + Km\log\tfrac{n}{mk})k_i/n$ for $i>m$.
Then $\f_1,\dots,\f_s$ cross contain a matching.
\end{lem}

\begin{proof}
We start by upgrading uncapturability to globalness.
We apply Lemma \ref{lem:upgrade2} with
$r=1$, $b=2d$, $\aA_i = \sqrt{n/mk_i}$, $\bB_i = (mk_i/n)^d$
noting that each $rb \le k_i \le n/2rm\aA_i$
and $\aA_i^b \bB_i = 2^d > 2rmk_i/n$,
obtaining $B = \bigcup_i B_i$ with each $|B_i| \le 2d$
such that each $\g_i = (\f_i)^{B_i}_B$
is $(r,4\aA_i \mu(\g_i))$-global with $\mu(\g_i) > (2mk_i/n)^d/2$.
We note by Lemma \ref{lem:globaluncap} that
$\g_i$ is $(n/8\aA_ik_i,(2mk_i/n)^d/4)$-uncapturable.
Now we pass to the biased setting:
we let $p_i = k_i/n$ and note that $\h_i = \g_i^\ua$
is $(n/8\aA_ik_i,(2mk_i/n)^d/16)$-uncapturable by Lemma \ref{lem:bin}.

Now we will apply Lemma \ref{thm:sharpset}.2 to choose $S_1,\dots,S_m$
with each $|S_i| < K\log\tfrac{n}{mk}$ and define
$\h^0_i,\dots,\h^m_i$ for $i \in [s]$ by $\h^0_i = \h_i$,
$\h^i_j = (\h^{i-1}_j)^{\es}_{S_i}$ for $j \ne i$
and $\h^i_i = (\h^{i-1}_i)^{S_i}_{S_i}$.
At step $i$, we have $\mu(\h^{i-1}_i) \ge (2mk_i/n)^d/16$
by uncapturability of $\h_i$,
as $\sum_{j<i} |S_j| < Km\log\tfrac{n}{mk}$
and $n/8\aA_ik_i \ge \tfrac{1}{8} \sqrt{nm/Kk}$,
using $n/mk \ge C \gg K$.

Applying Lemma \ref{thm:sharpset}.2 with $\eta < 1/2d$ and
$\sqrt{K}$ in place of $K$ we obtain $S_i \subset [n]$ with
$|S_i| \le \sqrt{K}\log \mu(\h^{i-1}_i)^{-1} < K\log\tfrac{n}{mk}$
and $\mu_{Kp_i}(\h^i_i) \ge \mu^\eta > \sqrt{mp_i}$,
so $\mu_{Kp_i}(\h^m_i) \ge  \sqrt{mp_i} - |S|Kp_i > 3m(Kp_i)$.
For $i>m$, by Lemma \ref{lem:bin} and a union bound we have
$\mu_{p_i}(\h^m_i) > \mu(\f_i)/4 - |S|p_i > 3sp_i$.
Thus by Lemma \ref{lem:matchlargelimit} there is a cross matching
in $\h^m_1,\dots,\h^m_s$, and so in $\f_1,\dots,\f_s$.
\end{proof}

\subsection{Lifted embeddings}

We conclude this section by proving Lemma \ref{lem:embedmix}
which completes the proof of Theorem \ref{thm:juntarefined}.
As mentioned earlier, the proof becomes more complicated
as the uniformity $k$ of our family decreases.
When it is quite large we can bound the fat shadow using Fairness,
but otherwise we must rely on the weaker estimates from Lemma \ref{lem:fat},
so there are additional technical challenges, resolved by enhanced upgrading
and in one case pausing the shadow embedding for a preliminary lifting step.

\begin{proof}[Proof of Lemma \ref{lem:embedmix}]
Let $C \gg C_1 \gg \tT^{-1} \gg \eps^{-1} \gg r\DD$ and $C < k < n/Cs$.
Let $G,H_1,\dots,H_\sS$ be as in Setup \ref{set:embed} with $\sS_1 \le \tT \sS$.
Let $\f_i \subset\tbinom{[n]}{k}$ for $i \in [\sS_1]$ and
$\f_i \subset\tbinom{[n]}{k-1}$ for $i \in [\sS_1+1,\sS]$.
Suppose $\f_i$ is $(C_1 \sS_1, \eps \sS_1 k/n)$-uncapturable for $i \in [\sS_1]$,
that $\mu(\f_i) \ge 1-\tT$ for $i \in [\sS_1+1,\sS_2]$, and
$\mu(\f_i) \ge \bB := e^{-k/C_1} + C_1 sk/n$ for $i \in [\sS_2+1,\sS]$.
We need to show that $\f_1,\dots,\f_\sS$ cross contain $H_1^+,\dots,H_\sS^+$.

We consider cases according to the size of $k$.
We start with the case $k \ge \sqrt{C_1} \log \tfrac{n}{\sS_1}$,
for which we will use enhanced upgrading.
We apply Lemma \ref{lem:upgrade3} to $\f_1,\dots,\f_{\sS_1}$
with $m=\sS_1$, $b = C_1 + \log_2 \tfrac{s}{m}$,
each $\bB_i = \eps mk/n$ and $2r$ in place of $r$, noting that
$2rb \le k \le n/2rmb^2$ and $\bB_i > 8rmk/bn$.
This gives disjoint $B_1,\dots,B_m$
with each $|B_i| \le 2rb+2$ such that,
setting $\g_i = (\f_i)^{B_i}_B$ where $B = \bigcup_i B_i$,
if $\mu(\g_i) < 2^b \eps mk/8n$ then
$\g_i$ is $(2r,8\mu(\g_i))$-global
with $\mu(\g_i) > \eps mk/8n > m/n \ge e^{-k/\sqrt{C_1}}$.
For $i \in [\sS_1+1,\sS]$, writing $\g_i = (\f_i)^\es_B$,
we have $\mu(\g_i) \ge \mu(\f_i) - |B|k/n \ge e^{-k/C_1} + C_1 sk/2n$.

By Fairness (Proposition \ref{prop:Fair}),
with $\sqrt{C_1}$ in place of $C$,
writing $c_i=(1-\eps)\mu(\g_i)$ for $i \in [\sS]$ we have
$\mu(\pl^{r'}_{c_i} \g_i ) \ge 1-\eps$ for $r' \in \{r-1,r\}$,
so $\pl_{c_1} \g_1,\dots,\pl_{c_\sS} \g_\sS$ cross contain
a copy $\phi(H_1),\dots,\phi(H_\sS)$ of $H_1,\dots,H_\sS$
by Lemma \ref{lem:nearcomplete}. We write $V' = \im \phi$
and consider $\h_1,\dots,\h_s$ corresponding
to the edges $A_1,\dots,A_s$ of $H_1,\dots,H_\sS$,
where for each edge $A_j$ of $H_i$ with $i \in [\sS]$
we let $\h_j = (\g_i)^{\phi(A_j)}_{V'}$.
To complete the proof of this case it suffices to show
that $\h_1,\dots,\h_s$ cross contain a matching.

To do so, we verify the conditions of Lemma \ref{lem:matchuncap}.
Consider any $A_j \in H_i$. If $i>\sS_1$ or $i \in [\sS_1]$
with $\mu(\g_i) \ge 2^b \eps mk/8n > C_1^2 sk/n$
then $\mu(\h_j) \ge c_i - |V'|k/n > C_1 sk/3n$.
Now consider $i \in [\sS_1]$ such that
$\g_i$ is $(2r,8\mu(\g_i))$-global with $\mu(\g_i) > \eps mk/8n$.
Then $\h_j$ and $\h'_j = (\g_i)^{\phi(A_j)}_{\phi(A_j)}$
are $(r,16\mu(\g_i))$-global by Lemma \ref{lem:globalrestrict}.
As $\mu(\h'_j) > c_i = (1-\eps)\mu(\g_i)$,
by Lemma \ref{lem:globaluncap}
$\h'_j$ is $(n/40k,\mu(\h'_j)/2)$-uncapturable,
so $\mu(\h_j) \ge \mu(\h'_j)/2 > \eps mk/20n$,
and $\h_j$ is $(n/80k,\mu(\h_j)/2)$-uncapturable
again by Lemma \ref{lem:globaluncap}.
Thus the required conditions hold.

Henceforth we can assume $k < \sqrt{C_1} \log \tfrac{n}{\sS_1}$.
In this case we upgrade uncapturability to globalness
using Lemma \ref{lem:upgrade1} to obtain
disjoint $S_1,\dots,S_{\sS_1}$
with each $|S_i| \le 2r$ such that,
setting $\g_i = (\f_i)^{S_i}_S$ where $S = \bigcup_i S_i$,
whenever $\mu(\g_i) < \bB$ we have
$S_i = \es$ and $\g_i$ is $(2r,2\bB)$-global
with $\mu(\g_i) > \eps \sS_1 k/n$.
For $i>\sS_1$ we set $\g_i = (\f_i)^\es_S$ and note that
$\mu(\g_i) \ge \mu(\f_i) - |S|k/n > \bB/2$.
As before, for any $i \notin [\sS_1+1,\sS_2]$
with $\mu(\g_i)>\bB/2$ Fairness gives
$\mu(\pl^{r'}_{c_i} \g_i ) \ge 1-\eps$ for $r' \in \{r-1,r\}$,
where $c_i=(1-\eps)\mu(\g_i)$. For $i \in [\sS_1+1,\sS_2]$
we have the better bound $\mu(\pl^{r'}_{c_i} \g_i ) \ge 1-\sqrt{\tT}$
where $c_i = 1-\sqrt{\tT}$ from Lemma \ref{lem:markov}.
For $i \in I := \{i : \mu(\g_i)<\bB/2 \}$
we note that $\g_i$ is $G^+$-free, as $S_i=\es$,
so we can bound the fat shadow by Lemma \ref{lem:fat}:
we take $\ell=k$, use $(2\eps)^{-1} \gg r\DD$ in place of $C$,
and write $c_i = \mu(\g_i)/2\tbinom{k}{r} \ge \mu(\g_i)/2k^r$,
to obtain \[ \mu(\pl^r_{c_i} \g_i) \ge
 \big( (\mu(\g_i)/2-(s/n)^{2k\eps}) n/sk^2 \big)^{r/(k-1)}
 \ge z:= (\sS_1/sk^2)^{2r/k} - (s/n)^{r\eps}.\]

Next we consider the case that $k \ge 2C_1 \log \tfrac{s}{\sS_1}$.
Then $z \ge 1-\eps$, so $\pl_{c_1} \g_1,\dots,\pl_{c_\sS} \g_\sS$
cross contain a copy $\phi(H_1),\dots,\phi(H_\sS)$ of $H_1,\dots,H_\sS$
by Lemma \ref{lem:nearcomplete}. With notation as in the previous case,
it remains to show that $\h_1,\dots,\h_s$ cross contain a matching.
To do so, we verify the conditions of Lemma \ref{lem:matchuncap+},
taking $m = |I|$, $d=2$ and $K = \eps^{-1}$.
Consider any $A_j \in H_i$. If $i \notin I$ then
$\mu(\h_j) \ge \bB/3 - |\im\phi|k/n
> 12(s + \eps^{-1}|I|\log\tfrac{n}{k|I|} )k/n$,
as $|I|/n \le \sS_1/n < e^{-k/\sqrt{C_1}}$,
so $|I|k/n \cdot \log\tfrac{n}{k|I|} < k^2 e^{-k/\sqrt{C_1}} < \bB^2$.
Now suppose $i \in I$, so that $\g_i$ is $(2r,2\bB)$-global
with $\mu(\g_i) > \eps \sS_1 k/n$.
Then $\h_j$ and $\h'_j = (\g_i)^{\phi(A_j)}_{\phi(A_j)}$
are $(r,4\bB)$-global by Lemma \ref{lem:globalrestrict}.
As $\mu(\h'_j) > c_i \ge \mu(\g_i)/2k^r$,
by Lemma \ref{lem:globaluncap}
$\h'_j$ is $(a,\mu(\h'_j)/2)$-uncapturable, where
$a = \mu(\g_i) n/8k\bB > \eps \sS_1/8\bB > rs \ge |\im\phi|$
as $\sS_1/s \ge e^{-k/2C_1} \ge \sqrt{\bB}$, since
$ks/n < k\DD \sS_1/n < \DD k e^{-k/\sqrt{C_1}}$.
Hence $\mu(\h_j) \ge \mu(\h'_j)/2 > \mu(\g_i)/4k^r > 2(2|I|k/n)^2$,
and $\h_j$ is $(4|I|,\mu(\h_j)/2)$-uncapturable
again by Lemma \ref{lem:globaluncap}.
Thus the required conditions hold.

It remains to consider the case $k < 2C_1 \log \tfrac{s}{\sS_1}$.
We start by applying \ref{lem:embedfat}
to $(\pl^r_{c_i} \g_i: i \in [\sS_2])$ with
$\tT_0 = \sqrt{\sS_1/\sS} \le \sqrt{\tT}$ in place of $\tT$,
recalling for $i \in [\sS_1+1,\sS_2]$ that
$\mu(\pl^r_{c_i} \g_i) \ge 1-\sqrt{\tT} \ge 1-\tT_0$ and
$\mu(\pl^r_{c_i} \g_i) \ge 1-\eps$ for $i \in [\sS_1] \sm I$,
and noting for $i \in I$ that
$\mu(\pl^r_{c_i} \g_i) \ge \tT_0^{1/2r} + n^{-\eps} + r\DD\sS_1/n$.
This gives a cross embedding $\phi$ of $H'_1,\dots,H'_{\sS_2}$
in $(\pl_{cc_i} \g_i: i \in [\sS_2])$, where $c = 1 - \tT_0^{1/r}$.

Next we extend $(\phi(H'_i): i \in [\sS_1]) = (\phi(H_i): i \in [\sS_1])$
to a cross embedding $(\phi(H^+_i): i \in [\sS_1])$ in $(\g_i: i \in [\sS_1])$,
by finding a cross matching in $(\h_j: j \in [s_1])$ corresponding
to the edges $A_1,\dots,A_{s_1}$ of $H_1,\dots,H_{\sS_1}$,
where for each edge $A_j$ of $H_i$ with $i \in [\sS_1]$
we let $\h_j = (\g_i)^{\phi(A_j)}_{\im\phi}$.
This is possible by Lemma \ref{lem:matchuncap+},
which applies similarly to the previous case,
where for uncapturability of $\h'_j$ we note that
now $|\im\phi| \le rs_1 \le r\DD\sS_1$.

Finally, we extend to a cross embedding
$(\phi(H^+_i): i \in [\sS])$ in $(\g_i: i \in [\sS])$
by finding a cross copy of $(A_j \sm V_1: s_1 < j \le s)$
in $(\h_j: s_1 < j \le s)$,
where for each edge $A_j$ of $H_i$ with $\sS_1 < i \le \sS$
we let $\h_j = (\g_i)^{\phi(A_j \cap V_1)}_{\im\phi}$.
This is possible by Lemma \ref{lem:embedlarge}, as each
$\mu(\h_j) \ge \mu(\g_i)-\DD\sS_1 k^2/n > \bB/4$,
using $k < 2C_1 \log \tfrac{s}{\sS_1}$ and $\sS_1 \le \tT \sS$.
\end{proof}

\newpage

\part{Exact Tur\'an results} \label{part:turan}

This final part of our paper contains our exact results
on the Tur\'an numbers of expanded hypergraphs.
We prove the Huang--Loh--Sudakov Conjecture
on cross containment of matchings in the first section.
The second section contains the proof of our Tur\'an result
for critical graphs (Theorem \ref{thm:critical}).
We conclude in the third section by proving
the F\"uredi--Jiang--Seiver conjecture on expanded paths;
the proof will apply to any graph satisfying
a certain generalised criticality condition.

\section{The Huang--Loh--Sudakov Conjecture} \label{sec:hls}

Here we prove Theorem \ref{thm: cross EMC bound},
which establishes the Huang--Loh--Sudakov Conjecture.
In the first subsection we prove a strong stability version
that has independent interest.
We then deduce the exact result in the second subsection.

\subsection{A strong stability result}

Here we  prove the following strong approximate version
of the Huang--Loh--Sudakov conjecture, which will be refined
to obtain the exact result in the following subsection.

\begin{thm}	\label{thm:HLSapprox}
Let $0 < C^{-1} \ll \eps$ and $\f_i \sub \tbinom{[n]}{k_i}$
with $C \le k_i \le n/Cs$ for all $i \in [s]$.
If $\f_1,\ldots,\f_s$ are cross free of a matching and each
$|\f_i| \geq |\s_{n,k_i,s-1}| - (1-\eps)\tbinom{n-1}{k_i-1}$
then there is $J \in \tbinom{[n]}{s-1}$ so that
$| \f_i \sm \s_{n,k_i,J} | \leq \eps \tbinom{n-1}{k_i-1}$
for all $i\in [s]$.
\end{thm}

The idea of the proof will be to consider
$A = \{a_1,\ldots, a_{\ell }\} \sub [n]$ maximal
such that there are distinct $b_1,\ldots, b_{\ell }$
so that all $(\f_{b_i})^{a_i}_{a_i}$ are large.
This motivates the setting of the following lemma.  

\begin{lem} \label{lem:preHLS}
Let $0 < C^{-1} \ll \bB \ll \eps \leq 1$
and $m,\ell , n,s,k_1,\ldots, k_s \in {\mathbb N}$
with $\ell \leq m \leq s$ and each $k_i \leq n/Cs$.
Suppose ${\cal F}_i \subset \tbinom {[n]}{k_i}$
and $J_i := \big \{j\in [n]:
\mu \big ( ({\cal F}_i)^{j}_j \big ) \geq \bB \big \}$
for each $i \in [s]$ are such that
\begin{enumerate}[(a)]
\item there are distinct $a_1,\ldots, a_{\ell}\in [n]$
with $a_i \in J_i$ for $i\in [\ell ]$;
\item $\mu \big ( ({\cal F}_i)^{\es }_{J_i}) \geq \eps (m-|J_i|)k_i/n$
and $J_i \subset A := \{a_1,\ldots, a_{\ell }\}$ for each $i \in [\ell +1,m]$;
\item $\mu \big ( {\cal F}_i \big ) \geq Ck_is/n$ for all $i\in [m+1,s]$.
\end{enumerate}
Then $\f_1,\ldots,\f_s$ cross contain a matching.
\end{lem}
	
\begin{proof}
It suffices to check the conditions of Lemma \ref{lem:matchuncap}
for $\g_1,\ldots,\g_s$ defined by
$\g_i = (\f_i)^{a_i}_A$ for $i \in [\ell]$
and $\g_i = (\f_i)^\es_A$ otherwise.
We do so with $m-\ell$ in place of $m$
and $(\g_i: \ell < i \le m)$ in place of $\f_1,\dots,\f_m$.
For $i \in [s] \sm [m]$ we have
$\mu(\g_i) \ge \mu(\f_i)-|A|k/n \ge Ck_is/2n$.
Similarly, for $i \in [\ell]$ we have
$\mu(\g_i) \ge \mu((\f_i)^{a_i}_{a_i})-|A|k/n \ge \bB/2 \ge Ck_is/2n$.
For $i \in [\ell+1,m]$ we note by definition of $J_i$
that $\g_i$ is $(1,2\bB)$-global with
$\mu(\g_i) \ge \mu( (\f_i)^\es_{J_i} ) - |A \sm J_i| \bB k/n
\ge \eps (m-\ell)k_i/n$,
so $(\eps (m-\ell)/4\bB ,\eps (m-\ell)k_i/2n)$-uncapturable
by Lemma \ref{lem:globaluncap}.
Thus the required conditions hold.
\end{proof}

We deduce our stability result as follows.
	
\begin{proof}[Proof of Theorem \ref{thm:HLSapprox}]
Let $0 < C^{-1} \ll \bB \ll \eps \le 1/2$
and $\f_i \sub \tbinom{[n]}{k_i}$
with $k_i \le n/Cs$ for all $i \in [s]$.
Let $J_1,\dots,J_s$ be as in Lemma \ref{lem:preHLS}.
Let $A = \{a_1,\ldots, a_{\ell }\} \sub [n]$ be maximal
such that there are distinct $b_1,\ldots, b_{\ell }$ with
$a_i \in J_{b_i}$ for all $i\in [\ell ]$. Without loss of generality
we may assume $b_i = i$ for all $i\in [\ell ]$. By maximality,
we have $J_i \subset \{a_1,\ldots, a_{\ell }\}$
for all $i \in [\ell +1, s]$.
		
We may assume $\ell<s$, and that
$\mu \big ( ( {\cal F}_h )^{\es }_{J_h} \big ) < .1\eps (s-|J_h|)k_h/n$
for some $h\in [\ell +1,s]$, otherwise Lemma \ref{lem:preHLS}
provides the required cross matching. Noting that
$|\s_{n,k_h,s-1}| - (1-\eps)\tbinom{n-1}{k_h-1}
\le |\f_h| \le |\s_{n,k_h,J_h}| + .1\eps (s-|J_h|)\tbinom{n-1}{k_h-1}$,
we see that $|J_h|=s-1=\ell$, $h=s$ and $J_h=A$.
Now for each $i \in [s-1]$, as $a_i \in A=J_h$
we can apply the same argument switching the roles
of ${\cal F}_i$ and ${\cal F}_h$ to deduce
$\mu \big ( ( {\cal F}_i )^{\es }_{J_i} \big ) < .1\eps k_h/n$
and $J_i=A$. The theorem follows.
\end{proof}

\subsection{The exact result}

To complete the proof of the Huang--Loh--Sudakov Conjecture
we will upgrade the approximate result of the previous subsection
to an exact result via the following bootstrapping lemma
(stated in a more general form than needed here
as we will also use it for our other Tur\'an results).

\begin{lem} \label{lem:bootstrap}
Let $C \gg \bB^{-1} \gg d \ge 1$
and $\f_i \subset \tbinom {[n]}{k_i}$
for all $i\in [s]$ with $\sum_{i=1}^s k_i \le n/C$.
Suppose $\f_1,\dots,\f_s$ are cross free
of some hypergraph  $G = \{e_1,\dots,e_s\}$
with $|e_i|=k_i$ for each $i \in [s]$ and
$e_s \cap \bigcup_{i=1}^{s-1} e_i = \es$.
If $\sum_{i=1}^{s-1} (1-\mu(\f_i)) \le \aA \in (0,\bB)$
then $\mu(\f_s) \le (\aA k_s/n)^d$.
\end{lem}

\begin{proof}
Let $k = n - n/C$ and $\g_s = \f_s^\ua \cap \tbinom{[n]}{k}$.
Then $\f_1,\dots,\f_{s-1},\g_s$ are cross free of $G'$
obtained from $G$ by enlarging $e_s$ to $e'_s$ of size $k$.
Suppose for contradiction that $\mu(\f_s) > (\aA k_s/n)^d$.
Let $t \in [k_s]$ be minimal so that
$|\f_s| = (\aA k_s/n)^d \tbinom{n}{k_s} \ge \tbinom{n-t}{k_s-t}$.
Then $(\aA k_s/n)^d < (k_s/n)^{t-1}$,
so if $t>2d$ then $\aA < (k_s/n)^{t/2d}$.
By Kruskal-Katona $|\g_s| \ge \tbinom{n-t}{k-t}$,
so $\mu(\g_s) \ge (1-2/C)^t > \sqrt{\aA}$,
as if $t \le 2d$ then $(1-2/C)^t > (1-2/C)^{2d} > \sqrt{\bB}$
or otherwise $\aA^{2d/t} < k_s/n \le C^{-1} < (1-2/C)^{4d}$.
Now we let $\phi:V(G') \to [n]$ be a uniformly random injection.
Let $E$ be the event that $\phi(e'_s) \notin \g_s$
or $\phi(e_i) \notin \f_i$ for some $i \in [s-1]$.
Then $1 = \mb{P}(E) \le 1-\mu(\g_s) + \sum_{i \in [s-1]} (1-\mu(\f_i))
< 1 - \sqrt{\aA} + \aA$, contradiction.
\end{proof}

Theorem \ref{thm: cross EMC bound} will now follow by combining
Theorem \ref{thm:HLSapprox} and Lemma \ref{lem:bootstrap}.

\begin{proof}[Proof of Theorem \ref{thm: cross EMC bound}]
Let $0 < C \ll \eps \ll 1$ and $\f_i \subset \tbinom {[n]}{k_i}$
with $|{\cal F}_i| \geq |{\cal S}_{n,k_i,s-1}|$
and $k_i \le n/C$ for all $i\in [s]$,
Suppose $\f_1,\dots,\f_s$ have no cross matching.
By Theorem \ref{thm:HLSapprox} there is $J \in \tbinom{[n]}{s-1}$
such that $\mu \big (  (\f_i)^{\es }_J \big ) = \eps_i k_i/|V|$
with $V = [n]\sm J$ and $\eps_i \leq \eps$ for all $i\in [s]$.
We may assume that $\eps_s$ is maximal.

Next we claim that we can list the elements of $J$
as ${\bf j} = (j_1,\ldots, j_{s-1})$ so that
\[ M_{\bf j} := \sum _{i\in [s-1]}
\mu \big (({\cal F}_i)^{j_i}_J\big )
\geq s-1 - \eps_s. \]
To see this, we note that
$ \mb{E}_{\bf j} M_{\bf j} = \mb{E}_{i \in [s-1]}
\sum _{j\in J}  \mu \big (({\cal F}_i)^{j}_J \big )$
when ${\bf j}$ is uniformly random.
As each $({\cal F}_i)^{I}_J \subset
({\cal S}_{n,k_i,s-1}) ^{I}_J$ whenever
$\es \ne I \subset J$ and
$\mu ({\cal F}_i) \geq \mu ({\cal S}_{n,k_i,s-1})$,
we have $0 \le \mu(\f_i) - \mu(\s_{n,k_i,s-1})
\le \mu( (\f_i)^\es_J) - k|V|^{-1} \sum_{j \in J} (1-\mu((\f_i)^j_J))$,
so $\sum_{j \in J} \mu((\f_i)^j_J)) \ge s-1-\eps_s$.
The claim follows.

Now let ${\cal H}_i = ({\cal F}_i)^{j_i}_J \subset \tbinom {V}{k-1}$
for all $i\in [s-1]$, and ${\cal H}_s =
({\cal F}_s)^{\es }_J \subset \tbinom {V}{k-1}$.
Then ${\cal H}_1,\ldots, {\cal H}_s$ have no cross matching,
$\sum _{i\in [s-1]} (1-\mu ({\cal H}_i)) \leq \eps_s$
and $\mu ({\cal H}_s) = \eps_s k_s/|V|$.
Therefore $\eps_s=0$ by	Lemma \ref{lem:bootstrap} with $d=1$.
By choice of $\eps_s$ we deduce $\eps_i = 0$ for all $i\in [s]$.
Thus ${\cal F}_i = {\cal S}_{n,k_i,J}$ for all $i\in [s]$.
\end{proof}

\section{Critical graphs}

In this section we prove Theorem \ref{thm:critical},
which gives exact Tur\'an results for expanded
critical $r$-graphs of bounded degree. In fact,
we will prove the following strong stability version.

\begin{thm} \label{thm:crit+}
Let $G \in \g(r,\DD,s)$ be critical and $C \gg \bB^{-1} \gg dr\DD$.

Suppose $\f \subset \tbinom{[n]}{k}$ with $C \le k \le n/Cs$
is $G^+$-free and $|\f| \ge |\s_{n,k,\sS-1}|-\eps\tbinom{n-1}{k-1}$
with $\eps \in (0,\bB)$.

Then there is $J \in \tbinom{[n]}{\sS-1}$ with
$|\f \sm \s_{n,k,J}| \le \eps^d \tbinom{n-1}{k-1}$.

Furthermore, if $k \le \sqrt{n}$
and $|\f| \ge |\s_{n,k,J}|-\bB\tbinom{n-r}{k-r}$
then $\f \subset \s_{n,k,J}$.
\end{thm}

In the first subsection we will describe the strategy
of the proof and complete the proof,
assuming a certain bootstrapping lemma
that will be proved in the second subsection.

\subsection{Strategy}

Recall that an $r$-graph $G$ is {\em critical}
if it has an edge $e$ such that
$\cc{G \sm e} = \tau(G \sm e) < \tau(G) = \cc{G}$.
Thus we can adopt the following set-up.

\begin{setup} \label{set:crit}
Let $G \in \g'(r,s,\DD)$ be critical.
Fix a crosscut $S$ in $G^+(r+1)$ with $|S|=\sS:=\cc{G}$ and
$\{G^x_x: x \in S\} = \{H_i: i \in [\sS]\}$ with $|H_\sS|=1$.
Let $I = \{i \in [\sS-1]: V(H_i) \cap V(H_\sS) \ne \es\}$.
\end{setup}

The following bootstrapping lemma will be proved in the next subsection.
It shows that if we cannot find a cross embedding of $H^+_1,\dots,H^+_\sS$
as in the above set up, if all but one of the families are nearly complete
then the last must be very small.

\begin{lem} \label{lem:crit}
Let $G,H_1,\dots,H_\sS$ be as in Setup \ref{set:crit}.
Let $C \gg \bB^{-1} \gg dr\DD$ and $\f_i \subset \tbinom{[n]}{k_i}$
with $k_i \in [k/2,k]$ for $i \in [\sS]$, where $C \le k \le n/Cs$.
Suppose $\f_\sS$ is $G^+$-free,
$\sum_{i=1}^{\sS-1} (1-\mu(\f_i)) \le \eps \le \bB$,
$\mu(\f_\sS) \ge \eps^d k/n$ and
$1-\mu(\f_i) \le \eps_0 := 2\eps/\sS$ for all $i \in I$.
Then $\f_1,\dots,\f_\sS$ cross contain $H^+_1,\dots,H^+_\sS$.
\end{lem}

We conclude this subsection
by deducing Theorem \ref{thm:crit+} from Lemma \ref{lem:crit}.

\begin{proof}[Proof of Theorem \ref{thm:crit+}]
By Theorem \ref{thm:juntarefined} (refined junta approximation)
there is $J \in \tbinom{[n]}{\sS-1}$ such that
$|\f \sm {\cal S}_{n,k,J}| = \dD \tbinom{n-1}{k-1}$
with $\dD^{-1} \gg dr\DD$. We write $J = \{j_1,\dots,j_{\sS-1}\}$,
$\f_i = \f^{j_i}_J$ for $i \in [\sS-1]$ and $\f_\sS=\f^\es_J$.
Note that $\f_\sS$ is $G^+$-free.
We may assume $I = [|I|]$ and $|\f_1| \ge \dots \ge |\f_{\sS-1}|$. Now
\begin{align*}\mu(\f) & \le \mu(\f^\es_J) + \mu(\s_{n,k,J})
- \tfrac{k-1}{n-|J|} \sum_{i=1}^{\sS-1} (1-\mu(\f_i)) \\
& \le \dD k/n + \mu(\f) + \eps k/n
- \tfrac{k}{2n} \sum_{i=1}^{\sS-1} (1-\mu(\f_i)),
\end{align*}
so $\sum_{i=1}^{\sS-1} (1-\mu(\f_i)) \le 2(\eps+\dD)$.
Now for each $i \in I$ we have
$1-\mu(\f_i) \le 4r\DD(\eps+\dD)/\sS$
as if $\sS \le 2|I| \le 2r\DD$ this follows from
$1-\mu(\f_i) \le 2(\eps+\dD)$, or otherwise from
$1-\mu(\f_i) \le \tfrac{2(\eps+\dD)}{\sS-|I|}$.

As $\f_1,\dots,\f_\sS$ are cross free of $H^+_1,\dots,H^+_\sS$
as in Setup \ref{set:crit}, Lemma \ref{lem:crit}
with $(2r\DD(\eps+\dD),2d)$ in place of $(\eps,d)$
gives $\dD k/n = \mu(\f_\sS)<(2r\DD(\eps+\dD))^{2d} k/n$.
As $\eps^{-1},\dD^{-1} \gg dr\DD$ we have $((2r\DD)(\eps+\dD))^{2d}
= (2r\DD)^{2d} \sum_{i=0}^{2d} \tbinom{2d}{i} \eps^i \dD^{2d-i}
< (\eps^d + \dD)/2$, so $\dD < \eps^d$, i.e.\
$|\f^\es_J| = |\f_\sS| < \eps^d \tbinom{n-1}{k-1}$.

Finally, let $k \le \sqrt{n}$ and suppose for contradiction
that $|\f| \ge |\s_{n,k,J}|-\bB\tbinom{n-r}{k-r}$
but there is some $A \in \f \sm \s_{n,k,J}$.
By the previous statement with $d=1$ and
$\eps = \bB\tbinom{n-r}{k-r} \tbinom{n-1}{k-1}^{-1}$
we have $|\f^\es_J| \le \bB\tbinom{n-r}{k-r}$,
so $|\s_{n,k,J} \sm \f| \le 2\bB\tbinom{n-r}{k-r}$.
We fix any $R \in \tbinom{A}{r}$ and
a bijection $\phi:A_s \to R$, where $H_\sS=\{A_s\}$
and define $\g_1,\dots,\g_{s-1}$ by $\g_j = (\f_i)^{\phi(A'_j)}_A$
whenever $A_j$ is an edge of $H_i$ with $A'_j = A_j \cap A_s$.
For each $j \in [s-1]$, writing $r_j=|A'_j|+1 \in [r]$, we have
$\tbinom{n-k-r_j}{k-r_j} - |\g_j| \le |\s_{n,k,J} \sm \f|$,
so as $\tbinom{n-k-r}{k-r} \ge .1\tbinom{n}{k-r}$ for $k \le \sqrt{n}$
we have $1-\mu(\g_j) \le 20\bB < 1/2$.
However, now $\g_1,\dots,\g_{s-1}$ cross contain
$A_1 \sm A_s,\dots,A_{s-1} \sm A_s$ by Lemma \ref{lem:embedlarge},
so we have the required contradiction.
\end{proof}

\subsection{Bootstrapping}

Now we complete the proof of Theorem \ref{thm:crit+}
by proving Lemma \ref{lem:crit}. The idea is to reduce to the case that
the critical edge is disjoint from all other edges,
so that we can apply Lemma \ref{lem:bootstrap}.

\begin{proof}[Proof of Lemma \ref{lem:crit}]
Let $G,H_1,\dots,H_\sS$ be as in Setup \ref{set:crit}.
Let $C \gg \bB^{-1} \gg dr\DD$ and $\f_i \subset \tbinom{[n]}{k_i}$
with $k_i \in [k/2,k]$ for $i \in [\sS]$, where $C \le k \le n/Cs$.
Suppose $\sum_{i=1}^{s-1} (1-\mu(\f_i)) \le \eps \le \bB$,
$\mu(\f_\sS) \ge \eps^d k/n$ and
$1-\mu(\f_i) \le \eps_0 := 2\eps/\sS$ for all $i \in I$.

We need to show that
$\f_1,\dots,\f_\sS$ cross contain $H^+_1,\dots,H^+_\sS$.
Write $G = \{A_1,\dots,A_s\}$ where $H_\sS=\{A_s\}$
and $A = A_s \cap \bigcup_{i<s} A_i$.
It suffices to find an injection $\phi:A \to [n]$
such that Lemma \ref{lem:bootstrap} provides a cross embedding
of $e_1^+,\dots,e_s^+$ in $\g_1,\dots,\g_s$, where
for each edge $A_j \in H_i$ we define $e_j = A_j \sm A_s$
and $\g_j = (\f_i)^{\phi(A \cap A_j)}_{\phi(A)}$.
We note that if $A \cap A_j=\es$ then
$1-\mu(\g_j) \le 2(1-\mu(\f_i))$ for any choice of $\phi$.
Also, for uniformly random $\phi$ we have
$\mb{P}(\mu(\g_j) \ge 1 - \sqrt{\eps_0} ) > 1 - \sqrt{\eps_0}$
whenever $i \in I$ by Lemma \ref{lem:markov}.

Next suppose $\mu(\f_\sS) \ge e^{-k\bB}$.
Then Fairness (Proposition \ref{prop:Fair}) gives
$\mb{P}(\mu(\g_s) \ge \mu(\f_\sS)/2)>1/2$.
By a union bound we can fix $\phi$ with
$\sum_{i=1}^{s-1} (1-\mu(\g_i)) \le 2\eps + |I|\sqrt{\eps_0}
\le \aA := 2\DD\sqrt{\eps}$
and $\mu(\g_s) \ge \mu(\f_\sS)/2 \ge (\aA k/n)^{3d}$.
Then Lemma \ref{lem:bootstrap} applies as required.

It remains to consider the case $\mu(\f_\sS) < e^{-k\bB}$.
We will apply Lemma \ref{lem:fat} to show that we can fix $\phi$ with
$\sum_{i=1}^{s-1} (1-\mu(\g_i)) \le 2\eps + |I|\sqrt{\eps_0}
\le \aA := 2\DD\sqrt{\eps}$ as above and
$\mu(\g_s) \ge c := \mu(\f_\sS)/2k^r
\ge e^{k\bB} \mu(\f_\sS) \cdot \mu(\f_\sS)/2k^r
\ge \mu(\f_\sS)^2 \ge (\aA k/n)^{6d}$.
Again this will suffice by Lemma \ref{lem:bootstrap}.
Lemma \ref{lem:fat} with $\ell=k$ gives $\mb{P}(\mu(\g_s) \ge c)
\ge (\mu(\f_\sS)/2)^{r/k} \ge \eps^{1/4} n^{-2r/k}$,
so we are done unless $\eps^{1/4} n^{-2r/k} < |I|\sqrt{\eps_0}$,
which implies $\sS^2 n^{-8r/k} < (2\DD)^4 \eps$.
As $\eps \ll \DD^{-1}$ this implies $k < n^\bB$, say.
Furthermore, we can assume $\f_\sS$ is $(2r,\mu(\f_\sS) \bB n/sk)$-global,
otherwise we can apply the above argument with some $(\f_\sS)^R_R$
in place of $\f_\sS$ to get $\mb{P}(\mu(\g_s) \ge c)
\ge (\mu(\f_\sS) \bB n/2sk)^{r/k} \ge \eps^{1/4} s^{-2r/k} > |I|\sqrt{\eps_0} $.

Now we claim that $\pl^r_c \f_\sS$ is $G$-free.
This will suffice to complete the proof,
as then Lemma \ref{lem:fat} gives the improved estimate
$\mu(\pl^r_c \f_\sS) \ge (\eps^d/ks)^{2r/k} - (s/n)^\bB
> |I|\sqrt{\eps_0}$, using $s \le r\sS < n^{8r/k}$.
To see the claim, we suppose $\phi(G) \sub \pl^r_c \f_\sS$
and will obtain a contradiction by finding a cross matching
in $\h_1,\dots,\h_s$, where for each edge $A_j$ of $G$
we let $\h_j = (\f_\sS)^{\phi(A_j)}_{\im\phi}$.
We verify the conditions of Lemma \ref{lem:matchuncap+},
with $(s,s,d,2)$ in place of $(s,m,d,K)$.
As $\f_\sS$ is $(2r,\mu(\f_\sS) \bB n/sk)$-global,
each $\h_j$ is $(r,2\mu(\f_\sS) \bB n/sk)$-global
by Lemma \ref{lem:globalrestrict}.
Also, $\f_\sS$ is $(\bB^{-1} s,\mu(\f_\sS)/2)$-uncapturable
by Lemma \ref{lem:globaluncap}, so each
$\mu(\h_j) \ge \mu(\f_\sS)/2 \ge \eps^d k/2n$,
and each $\h_j$ is $(s/2\bB, \eps^d k/4n)$-uncapturable
by Lemma \ref{lem:globaluncap}.
As $\sS^2 n^{-8r/k} < (2\DD)^4 \eps$ and $k<n^\bB$
we have $\eps^d k/n > (3sk/n)^d$, and so the conditions of
Lemma \ref{lem:matchuncap+} hold. But this is a contradiction,
as then $\h_1, \ldots, \h_s$ cross contain a matching.
Therefore $\pl^r_c \f_\sS$ is $G$-free, as claimed.
\end{proof}

\section{The F\"uredi--Jiang--Seiver Conjecture}

In this section we prove the F\"uredi--Jiang--Seiver Conjecture
on the Tur\'an numbers of expanded paths. As previously mentioned,
for paths of odd length the conjecture follows from our result
on critical graphs (Theorem \ref{thm:critical}),
so it remains to consider paths of even length.
We will consider the more general setting of expansions
of (normal) graphs ($r$-graphs with $r=2$)
satisfying the following generalised criticality property.
Recall that we denote the crosscut and transversal
numbers of an $r$-graph $G$ by $\cc{G}$ and $\tau(G)$,
and that $\cc{G} \ge \tau(G)$.
Consider any $G$ with $\tau(G) = \cc{G}$.
We say $G$ is \emph{$a_1$-degree-critical}
if (i) $\cc{G-x}<\cc{G}$ for some $x$ of degree $|G^x_x| \le a_1$,
and (ii) $\tau(G-x)=\tau(G)$ for any $x$ with $|G^x_x|<a_1$.
We say $G$ is \emph{$a_2$-matching-critical}
if (i) $\cc{G \sm M}<\cc{G}$ for some matching $M$ with $|M| \le a_2$,
and (ii) $\tau(G \sm M)=\tau(G)$ for any matching $M$ with $|M|<a_2$.
We say $G$ is \emph{$(a_1,a_2)$-critical} if it is both
$a_1$-degree-critical and $a_2$-matching-critical.

We note that even paths and cycles are $(2,2)$-critical,
and that any $G$ is critical (in the sense defined above)
if and only if $G$ is $(a_1,1)$-critical,
where $a_1$ is the minimum possible degree of any vertex
belonging to any minimum size crosscut of $G^+$.
The significance of the generalised definition is that
it enables to show that the following natural construction
is extremal for the Tur\'an problem for $G^+$.
For any $T \sub [n]$ we write
$\g_{n,k}(T) = \{A \in \tbinom{[n]}{k}: T \sub A\}$
for the family in $\tbinom{[n]}{k}$ generated by $T$.
For $\mc{T} \subset \{0,1\}^n$ we write
$\g_{n,k}(\mc{T}) = \bigcup_{T \in \mc{T}} \g_{n,k}(T)$.
We let $\f_{n,k,G} = \g_{n,k}(\mc{T})$
where $\mc{T}$ is the disjoint union of $\cc{G}-1$ singletons
and a graph $F_{a_1a_2}$ with as many edges as possible subject to
having no vertex of degree $\ge a_1$ or matching of size $\ge a_2$.
Then $\f_{n,k,G}$ is $G^+$-free by definition of $(a,b)$-criticality.
We will show that it is extremal. When $G$ is a path of even length
this will complete the proof of the F\"uredi--Jiang--Seiver Conjecture.

\begin{thm} \label{thm:ab}
Let $G \in \g(2,\DD,s)$ be $(a_1,a_2)$-critical,
$C \gg a_2\DD$ and $C \leq k \leq n/Cs$.
Then $\ex(n,G^+(k)) = |\f_{n,k,G}|$.
\end{thm}

Moreover, we will prove the following strong stability version.

\begin{thm} \label{thm:ab+}
Let $G \in \g(2,\DD,s)$ be $(a_1,a_2)$-critical
and $C \gg \bB^{-1} \gg a_2d\DD$.

Suppose $\f \subset \tbinom{[n]}{k}$
with $C \le k \le n/Cs$ is $G^+$-free.
If $|\f| \ge |\s_{n,k,\sS-1}|$ then
$|\f \sm \g_{n,k}(\mc{T})| \le \bB^{-1} \tbinom{n-3}{k-3}$
for some $\mc{T} = \{\{x\}: x \in J\} \cup F$
where $J \in \tbinom{[n]}{\sS-1}$
and $F \subset \tbinom{[n] \sm J}{2}$ with $|F| \le |F_{a_1a_2}|$.

Moreover, if $|\f| \ge |\f_{n,k,G}| - \eps \tbinom{n-2}{k-2}$ with $\eps \in (0,\bB)$
then $\mu(\f \sm \g) \le (\eps k/n)^d$ for some copy $\g$ of $\f_{n,k,G}$,
where if $k \le \sqrt{n}$ then $\f \sub \g$.
\end{thm}

Throughout this section we adopt the following set up.

\begin{setup} \label{set:ab}
Let $G \in \g'(2,s,\DD)$ be $(a_1,a_2)$-critical with $\cc{G}=\sS$.
Let $\b = \{B_i: i \in [a]\}$ be a $r$-graph matching with $r \in [2]$,
and $\b' = \{B'_i: i \in [a]\} \subset G$,
where if $r=2$ then $a=a_2$ and each $B'_i = B_i$
or if $r=1$ then $a=a_1$ and each $B'_i = B_i \cup \{x\}$
for some vertex $x$ of degree $a$.
Let $S=\{s_1,\dots,s_{\sS-1}\}$ be a crosscut in $(G\sm\b')^+$
and let $H_i = G^{s_i}_{s_i}$ for $i \in [\sS-1]$.
Let $I = \{i \in [\sS-1]: V(H_i) \cap V(\b) \ne \es\}$.
\end{setup}

We prove a bootstrapping lemma in the next subsection
and then deduce Theorem \ref{thm:ab+} in the following subsection.

\subsection{Bootstrapping}

In this subsection we prove the following bootstrapping lemma,
which is analogous to Lemma \ref{lem:crit}, except that
rather than concluding that some family is small
we conclude that some family is capturable.

\begin{lem} \label{lem:ab}
With notation as in Setup \ref{set:ab},
let $C \gg \bB^{-1} \gg ad\DD$ and $C \le k \le n/Cs$.
Let $\f_i \subset \tbinom{[n]}{k_i}$
with $k_i \in [k/2,k]$ for $i \in [\sS-1]$
and $\f'_i \subset \tbinom{[n]}{k'_i}$
with $k'_i \in [k/2,k]$ for $i \in [a]$ be such that
$\f_1,\dots,\f_{\sS-1},\f'_1,\dots,\f'_a$ are cross free of
$H_1^+,\dots,H_{\sS-1}^+,B_1^+,\dots,B_a^+$.
Suppose $\sum_{i=1}^{s-1} (1-\mu(\f_i)) \le \eps \le \bB$
and $1-\mu(\f_i) \le \eps_0 := 2\eps/\sS$ for all $i \in I$.
Then some $\f'_i$ is $(\bB^{-1},\gG_i + (k/n)^d)$-capturable,
where $\gG_i < \eps^d$,
and if $\f'_i$ is $G^+$-free then $\gG_i < \eps^d k/n$.
\end{lem}

The proof requires the following lemma which is
analogous to Lemma \ref{lem:bootstrap}.

\begin{lem} \label{lem:bootstrap2}
Let $C \gg C' \gg ad$,
$\f_i \subset \tbinom {[n]}{k_i}$ for $i\in [s]$
and $\f'_i \subset \tbinom{[n]}{k'_i}$ for $i \in [a]$
with $\sum_{i=1}^s k_i + \sum_{i=1}^a k'_i \le n/C$.
Suppose $(\f_1,\dots,\f_s,\f'_1,\dots,\f'_a)$
are cross free of $G = (e_1,\dots,e_s,e'_1,\dots,e'_a)$
with each $|e_i|=k_i$, $|e'_i|=k'_i$ and $e \cap e'_i = \es$
for all $i \in [a]$ and $e'_i \ne e \in G$.
If $\sum_{i=1}^s (1-\mu(\f_i)) < 1/2$ then
some $\f'_i$ is $(C',(k'_i/n)^d)$-capturable.
\end{lem}

\begin{proof}
Let $k = n/2a$ and for each $i \in [a]$ let
$\g_i = (\f'_i)^\ua \cap \tbinom{[n]}{k}$.
Then $(\f_1,\dots,\f_s,\g_1,\dots,\g_a)$ are cross free of $G'$
obtained from $G$ by enlarging each $e'_i$ to $e^*_i$ of size $k$.
Suppose for contradiction that each
$\f'_i$ is $(C',(k'_i/n)^d)$-uncapturable.
Then an argument of Dinur and Friedgut \cite{dinur2009intersecting}
(apply Russo's Lemma and Friedgut's junta theorem)
shows that each $\mu(\g_i) > 1-1/2a$.
Consider a uniformly random injection $\phi:V(G') \to [n]$.
Let $E$ be the event that some $\phi(e_i) \notin \f_i$
or some $\phi(e^*_i) \notin \g_i$.
Then $1 = \mb{P}(E) \le \sum_{i \in [s]} (1-\mu(\f_i))
+ \sum_{i \in [a]} (1-\mu(\g_i)) < 1/2 + 1/2$, contradiction.
\end{proof}

\begin{proof}[Proof of Lemma \ref{lem:ab}]
With notation as in Setup \ref{set:ab},
let $C \gg \bB^{-1} \gg b \gg d \gg a\DD$
and $C \le k \le n/Cs$.
Let $\f_i \subset \tbinom{[n]}{k_i}$
with $k_i \in [k/2,k]$ for $i \in [\sS-1]$
and $\f'_i \subset \tbinom{[n]}{k'_i}$
with $k'_i \in [k/2,k]$ for $i \in [a]$ be such that
$\f_1,\dots,\f_{\sS-1},\f'_1,\dots,\f'_a$ are cross free of
$H_1^+,\dots,H_{\sS-1}^+,B_1^+,\dots,B_a^+$.
Suppose $\sum_{i=1}^{s-1} (1-\mu(\f_i)) \le \eps \le \bB$
and $1-\mu(\f_i) \le \eps_0 := 2\eps/\sS$ for all $i \in I$.
Suppose for contradiction that
each $\f'_i$ is $(\bB^{-1},\gG_i + (k/n)^d)$-uncapturable,
where either $\gG_i \ge \eps^d$
or $\f'_i$ is $G^+$-free and $\gG_i \ge \eps^d k/n$.

We start by upgrading uncapturability to globalness.
By Lemma \ref{lem:upgrade2} with $(b,4,a)$ in place of $(b,r,m)$
and each $\aA_i = n/kb$, $\bB_i = \gG_i + (k/n)^d$,
noting that $8b \le k \le n/8a(n/bk)$,
$4ba < \bB^{-1}$ and $(n/kb)^b (k/n)^d > n/k \gg 1$,
there is a set $S'$ partitioned into
$S'_1,\dots,S'_a$ with each $|S'_i| \le 8b$ such that
each $\g^0_i := (\f'_i)^{S'_i}_{S'}$ is $(8,4\mu(\g^0_i)n/kb)$-global
with $\mu(\g^0_i) > \aA_i^{1_{S'_i \ne \es}} \bB_i/2$.
We have $2\mu(\g^0_i) > \eps^d + (k/n)^d$,
unless $\f'_i$ is $G^+$-free and $S'_i=\es$,
in which case $\g^0_i$ is a restriction of $\f'_i$,
so is also $G^+$-free, with $2\mu(\g^0_i) > \eps^d k/n + (k/n)^d$.

Next we define $\g'_i := (\f'_i)^{S_i}_S$ with enhanced globalness,
obtaining $S$ partitioned into $S_1,\dots,S_a$ by letting
$S_i = S'_i$ if $\g^0_i$ is $(4,\mu(\g^0_i) \bB n/sk)$-global,
or otherwise letting $S_i = S'_i \cup R_i$ where $|R_i| \le 4$ and
$\g^1_i := (\g^0_i)^{R_i}_{R_i}$ has $\mu(\g^1_i) > \mu(\g^0_i) \bB n/sk$.
We also define $\g_i = (\f_i)^\es_S$ for $i \in [\sS-1]$
and note that each $1-\mu(\g_i) \le 2(1-\mu(\f_i))$.

By Lemma \ref{lem:globalrestrict},
each $\g^1_i$ or $\g'_i$ is $(4,2\mu(\g^0_i) \bB n/sk)$-global if $R_i=\es$
or $(4,8\mu(\g^0_i)n/kb)$-global otherwise. By Lemma \ref{lem:globaluncap},
each $\g^1_i$ is $(b/8,\mu(\g^1_i)/2)$-uncapturable,
so $\mu(\g'_i) > \mu(\g^1_i)/2 \ge \mu(\g^0_i)/2$.
Thus $2\bB^{-1} \mu(\g'_i) \ge \gG'_i + (k/n)^d$, and
\begin{enumerate}[(i)]
\item $\g'_i$ is $(4,8\mu(\g'_i)n/kb)$-global
with $\gG'_i \ge \eps^d/s$, or
\item $\g'_i$ is $G^+$-free
and $(4,2\mu(\g'_i) \bB n/sk)$-global
with $\gG'_i \ge \eps^d k/n$.
\end{enumerate}
Indeed, if option (i) does not hold then
$\g^0_i$ is $G^+$-free with $2\mu(\g^0_i) > \eps^d k/n + (k/n)^d$,
and also is $(4,\mu(\g^0_i) \bB n/sk)$-global,
so $R_i=\es$ and $\g'_i$ is a restriction of $\g^0_i$,
so is also $G^+$-free.

We will show that $\g_1,\dots,\g_{\sS-1},\g'_1,\dots,\g'_a$
cross contain $H_1^+,\dots,H_{\sS-1}^+,B_1^+,\dots,B_a^+$,
thus obtaining the required contradiction. It suffices to
find an injection $\phi:B \to [n]$, where $B=\bigcup_{i=1}^a B_i$,
such that Lemma \ref{lem:bootstrap2} provides a cross embedding
of $e_1^+,\dots,e_s^+$ in $\g_1,\dots,\g_s$, where
for each edge $A_j \in H_i$ we define $e_j = A_j \sm B$
and $\h_j = (\g_i)^{\phi(B \cap A_j)}_{\phi(B)}$,
or if $A_j=B_i$ we define $e_j = A_j \sm B = \es$
and $\h_j = (\g'_i)^{\phi(B_i)}_{\phi(B)}$.

We note that if $B \cap A_j=\es$ then each
$1-\mu(\h_j) \le 2(1-\mu(\g_i))$ for any $\phi$.
We consider $\phi$ obtained by choosing independent uniformly random
injections $\phi_i:B_i \to [n]$ for each $i \in [a]$.
Then $\mb{P}(\phi \text{ is injective}) \ge 1-2a^2/n$ and
$\mb{P}(\mu(\h_j) \ge 1 - \sqrt{\eps_0} ) > 1 - 2\sqrt{\eps_0}$
whenever $A_j \in \bigcup_{i \in I} H_i$ by Lemma \ref{lem:markov}.
We write $E_i$ for the event that $\phi_i(B_i) \in \pl_{c_i} \g'_i$,
where $c_i = b^{-.3} \mu(g'_i)$. It suffices to show that
conditional on $E_i$ each $\h'_i := (\g'_i)^{\phi_i(B_i)}_{\phi(B)}$
is $(\sqrt{b},(k/n)^{2d})$-uncapturable,
and that $\mb{P}(E_i) \ge \eps_0^{1/3a}$.

For uncapturability, we recall that
$\g'_i$ is $(4,8\mu(\g'_i)n/kb)$-global
with $2\bB^{-1} \mu(\g'_i) \ge (k/n)^d$.
Thus $\h'_i$ and $\h''_i := (\g'_i)^{\phi_i(B_i)}_{\phi_i(B_i)}$
are $(2,8\mu(\g'_i)n/kb)$-global by Lemma \ref{lem:globalrestrict}.
Conditional on $E_i$ we have $\mu(\h''_i) > c_i$,
so $\h''_i$ is $(b^{.7}/16,\mu(\h''_i)/2)$-uncapturable
by Lemma \ref{lem:globaluncap}.
Then $\mu(\h'_i) \ge \mu(\h''_i)/2 \ge b^{-.3} \mu(\g'_i)/4$,
so $\h'_i$ is $(b^{.7}/32,\mu(\h'_i)/2)$-uncapturable
by Lemma \ref{lem:globaluncap},
and so $(\sqrt{b},(k/n)^{2d})$-uncapturable.

It remains to show $\mb{P}(E_i) \ge \eps_0^{1/3a}$.
We may assume $\mu(\g'_i) < e^{-k\bB}$,
otherwise this holds easily by Fairness (Proposition \ref{prop:Fair}).
As $2\bB^{-1} \mu(\g'_i) \ge (k/n)^d$ this gives $k < n^\bB$.
By Lemma \ref{lem:fat} with $\ell=b^{.1}$ we are done unless
$\eps_0^{1/3a} > \mb{P}(E_i) = \mu(\pl_{c_i} \g'_i) \ge (\mu(\g'_i)/2)^{2/\ell}$,
which implies $\gG'_i + (k/n)^d \le 2\bB^{-1}\mu(\g'_i) < (\eps/s)^{b^{.05}}$.
As $\gG'_i < \eps^d/s$ we have option (ii) above, so $\g'_i$ is $G^+$-free.
As $\eps^d k/n \le \gG'_i < (\eps/s)^{b^{.05}}$
we also have $s < \eps n^{b^{-.05}}$.

Now we claim that $\pl^2_{c_i} \g'_i$ is $G$-free.
This will suffice to complete the proof,
as then Lemma \ref{lem:fat} gives the improved estimate
$\mu(\pl^2_{c_i} \g'_i) \ge (\eps^d k/sb + k/n - (s/n)^2)^{b^{-.02}}
> (\eps/s)^{b^{-.01}}$.
To see the claim, we suppose $\phi'(G) \sub \pl^2_{c_i} \g'_i$
and will obtain a contradiction by finding a cross matching
in $\a_1,\dots,\a_s$, where for each edge $A_j$ of $G$
we let $\a_j = (\g'_i)^{\phi'(A_j)}_{\im\phi'}$.
We verify the conditions of Lemma \ref{lem:matchuncap+},
with $(s,s,d,2)$ in place of $(s,m,d,K)$.
As $\g'_i$ is $(4,2\mu(\g'_i) \bB n/sk)$-global,
each $\h_j$ is $(2,4\mu(\g'_i) \bB n/sk)$-global
by Lemma \ref{lem:globalrestrict}.
Also, $\g'_i$ is $(s/4\bB,\mu(\g'_i)/2)$-uncapturable
by Lemma \ref{lem:globaluncap}, so each
$\mu(\h_j) \ge \mu(\g'_i)/2 \ge \bB \eps^d k/4n$,
and each $\h_j$ is $(s/8\bB, \bB \eps^d k/8n)$-uncapturable
by Lemma \ref{lem:globaluncap}.
As $s < \eps n^{b^{-.05}}$ and $k<n^\bB$ we have
$\bB \eps^d k/8n > (3sk/n)^d$,
so the required conditions hold.
\end{proof}

\subsection{Strong stability}

We conclude with the proof of the main result of this section.

\begin{proof}[Proof of Theorem \ref{thm:ab+}]
Let $G \in \g(2,\DD,s)$ be $(a_1,a_2)$-critical
and $C \gg \bB^{-1} \gg b \gg d \gg a_2 \DD$.
Suppose $\f \subset \tbinom{[n]}{k}$
with $C \le k \le n/Cs$ is $G^+$-free
and $|\f| \ge |\s_{n,k,\sS-1}|$.

By Theorem \ref{thm:juntarefined} (refined junta approximation)
there is $J \in \tbinom{[n]}{\sS-1}$ such that
$|\f \sm \s_{n,k,J}| = \dD \tbinom{n-1}{k-1}$
with $\dD^{-1} \gg bd\DD$. We write $J = \{j_1,\dots,j_{\sS-1}\}$,
let $\f_i = \f^{j_i}_J$ for $i \in [\sS-1]$,
say with $|\f_1| \ge \dots \ge |\f_{\sS-1}|$,
and note that $\f^\es_J$ is $G^+$-free.
As in the proof of Theorem \ref{thm:crit+},
we have $\sum_{i=1}^{\sS-1} (1-\mu(\f_i)) \le 2\dD$,
so $1-\mu(\f_i) \le 4r\DD\dD/\sS$
for any $i \le \min\{r\DD,\sS-1\}$.

As $G$ is $a_2$-matching-critical, we can define
$H^2_1,\dots,H^2_{\sS-1}$, $B^2_1,\dots,B^2_{a_2}$ and $I^2$
as in Setup \ref{set:ab} with $r=2$ and $a=a_2$,
where we identify $I^2$ with $[|I^2|]$.
Letting $\f'_i = \f^\es_J$ for $i \in [a_2]$,
we have  $\f_1,\dots,\f_{\sS-1},\f'_1,\dots,\f'_{a_2}$ cross free of
$(H^2_1)^+,\dots,(H^2_{\sS-1})^+,(B^2_1)^+,\dots,(B^2_{a_2})^+$,
so $\f^\es_J$ is $(b,(2\dD)^d k/n + (k/n)^d)$-capturable
by Lemma \ref{lem:ab}. We fix $J' \in \tbinom{[n \sm J]}{b}$
so that $\mu(\f^\es_{J \cup J'}) < (2\dD)^d k/n + (k/n)^d$.

As $G$ is $a_1$-degree-critical, we can define
$H^1_1,\dots,H^1_{\sS-1}$, $B^1_1,\dots,B^1_a$ and $I^1$
as in Setup \ref{set:ab} with $r=1$ and $a=a_1$,
where we identify $I^1$ with $[|I^1|]$. For each $x \in J'$,
letting $\f'_i = \f^x_{J \cup \{x\}}$ for $i \in [a_1]$,
we have  $\f_1,\dots,\f_{\sS-1},\f'_1,\dots,\f'_{a_1}$ cross free
of $(H^1_1)^+,\dots,(H^1_{\sS-1})^+,(B^1_1)^+,\dots,(B^1_{a_2})^+$,
so $\f^x_{J \cup \{x\}}$ is $(b,(2\dD)^d k/n + (k/n)^d)$-capturable
by Lemma \ref{lem:ab}. We fix $J_x \in \tbinom{[n] \sm (J \cup \{x\})}{b}$
so that $\mu((\f^x_{J \cup \{x\}})^\es_{J_x}) < (2\dD)^d k/n + (k/n)^d$.

Let $F = \{T \in \tbinom{[n] \sm J}{2}: \mu(\f^T_{T \cup J}) > bk/n\}$.
Then $F \subset F' := \{ xy: x \in J', y \in J_x\}$ and $|F'| \le b^2$.
Writing $\mc{T} = \{\{x\}: x \in J\} \cup F$, we have
$|\f \sm \g_{n,k}(\mc{T})| \le |\f^\es_{J \cup J'}|
+ \sum_{x \in J'} |\f^x_{J \cup \{x\} \cup J_x}|
+ \sum_{T \in F'} |\f^T_{T \cup J}|$,
so $\mu(\f \sm \g_{n,k}(\mc{T}))
\le ((2\dD)^d k/n + (k/n)^d)(1+bk/n) + (bk/n)^3$.
Writing $\g := \g_{n,k}(\mc{T})$,
as $|\f \sm \g| \ge  |\f \sm \s_{n,k,J}| - |\g_{n,k}(F)|$
we also have $\mu(\f \sm \g) \ge \dD k/n - (bk/n)^2$.
We deduce $\dD k/n \le (2\dD)^d k/n + 2(bk/n)^2$, so $\dD \le 3bk/n$,
giving $|\f \sm \g| \le 2b^3 \tbinom{n-3}{k-3}$.

To complete the proof of the first statement of the theorem,
it remains to show $|F| \le |F_{a_1a_2}|$. To see this, note that
otherwise $F$ contains some $F_0 = (T_i: i \in [a_r])$,
where $r=2$ and $F_0$ is a matching or $r=1$ and $F_0$ is a star.
Writing $\f'_i = \f^{T_i}_{J \cup T_i}$, we have
$\f_1,\dots,\f_{\sS-1},\f'_1,\dots,\f'_{a_2}$ cross free
of $(H^r_1)^+,\dots,(H^r_{\sS-1})^+,(B^r_1)^+,\dots,(B^r_{a_r})^+$,
so some $\f'_i$ is $(b/2,(k/n)^d)$-capturable by Lemma \ref{lem:ab}.
However, $\mu(\f'_i)>bk/n$ as $T_i \in \f$, so we have a contradiction.

Now suppose $|\f| \ge |\f_{n,k,G}| - \eps \tbinom{n-2}{k-2}$
with $\eps \in (0,\bB)$. We have
\[ \mu(\f)  \le \mu(\f \sm \g) + \mu(\g)
- \tfrac{k}{2n} \sum_{i=1}^{\sS-1} (1-\mu(\f_i))
- \tfrac{k^2}{2n^2} \sum_{T \in F} (1-\mu(\f^T_{J \cup T})), \]
where $\mu(\f \sm \g) \le 2(bk/n)^3$ and
$\mu(\g) \le \mu(\f_{n.k.G}) - (|F_{a_1a_2}|-|F|)k^2/2n^2
\le \mu(\f) + (|F_{a_1a_2}|-|F|+2\eps)k^2/2n^2$.
Thus $|F|=|F_{a_1a_2}|$,
so $\g := \g_{n,k}(\mc{T})$ is a copy of $\f_{n,k,G}$, and
\[\sum_{i=1}^{\sS-1} (1-\mu(\f_i))
+ \sum_{T \in F} (1-\mu(\f^T_{J \cup T})) \le 3\eps.\]
Next we suppose for contradiction that
$\mu(\f \sm \g) > (\eps k/n)^d$.
We fix some $T \in \tbinom{[n] \sm J}{2} \sm F$
with $\mu(\f^T_{J \cup T}) > (\eps k/n)^{d+2}$.
By maximality of $F_{a_1a_2}$ we can fix a matching
$T_1,\dots,T_{a_2}$ in $F$ with $T_{a_2}=T$.
Writing $\f'_i = \f^{T_i}_{J \cup T_i}$, we have
$\f_1,\dots,\f_{\sS-1},\f'_1,\dots,\f'_{a_2}$ cross free
of $(H^2_1)^+,\dots,(H^2_{\sS-1})^+,(B^2_1)^+,\dots,(B^2_{a_2})^+$.
Thus Lemma \ref{lem:bootstrap} gives the required contradiction,
so $\mu(\f \sm \g) \le (\eps k/n)^d$, as required.

Finally, let $k \le \sqrt{n}$ and suppose for contradiction
that there is some $A \in \f \sm \g$. From the previous statement
we have $|\g \sm \f| \le 2\bB \tbinom{n-2}{k-2}$.
We fix any $T \in \tbinom{[n] \sm J}{2}$ with $T \sub A$,
a matching $T_1,\dots,T_{a_2}$ in $F$ with $T_{a_2}=T$,
and a bijection $\phi:B^2_{a_2} \to T$.
Writing $A'_j = A_j \cap A_s$ for each edge $A_j$ of $G$,
where $A_s=B^2_{a_2}$, we define $\g_1,\dots,\g_{s-1}$ by
$\g_j = (\f_i)^{\phi(A'_j)}_A$ if $A_j \in H_i$ with $i \in [\sS-1]$
or $\g_j = (\f^\es_J)^{\phi(A'_j)}_A$ if $A_j = B^2_i$ with $i \in [a_2-1]$.
For each $j \in [s-1]$, writing $r_j=|A'_j|+1 \in [2]$, we have
$\tbinom{n-k-r_j}{k-r_j} - |\g_j| \le |\g \sm \f|$,
so as $\tbinom{n-k-2}{k-2} \ge .1\tbinom{n}{k-2}$ for $k \le \sqrt{n}$
we have $1-\mu(\g_j) \le 20\bB < 1/2$.
However, now $\g_1,\dots,\g_{s-1}$ cross contain
$A_1 \sm A_s,\dots,A_{s-1} \sm A_s$ by Lemma \ref{lem:embedlarge},
so we have the required contradiction.
\end{proof}

\section*{Concluding remarks}
\label{section: concluding remarks}

We are optimistic that our sharp threshold result
in the sparse regime will have many applications
in the same vein as the applications
of the classical sharp threshold results,
e.g.\ to Percolation \cite{benjamini2012sharp},
Complexity Theory \cite{friedgut1999sharp},
Coding Theory \cite{kudekar2017reed},
and Ramsey Theory \cite{friedgut2016sharp}.

In particular, it may be possible to estimate
the location of thresholds in the spirit of
the Kahn-Kalai conjecture \cite[Conjecture 2.1]{kahn2007thresholds}
that the threshold probability $p_c(H)$
for finding some graph $H$ in $G(n,p)$
should be within a log factor of
its `expectation threshold' $p_E(H)$
(the probability at which every subgraph $H'$ of $H$
we expect at least one copy of $H'$).
This question is interesting when $|V(H)|$
depends on $n$, e.g.\ if $H$ is a bounded degree spanning tree
it predicts $p_c(H) = O(n^{-1}\log n)$,
which was a longstanding open problem, recently
resolved by Montgomery \cite{montgomery2018trees}.

To obtain similar results from our sharp threshold theorem 
one needs to show that the property of containing $H$
is not `local': writing $\mu_p = \mb{P}(H \sub G(n,p))$,
this means that if we plant any set $E$ of
$O(\log \mu_p^{-1})$ edges we still have
$\mb{P}(H \sub G(n,p) \mid E \sub G(n,p)) \le  \mu_p^{O(1)}$.
An open problem is to apply this approach to estimate
other thresholds that are currently unknown,
e.g.\ the threshold for containing any given $H$
of maximum degree $\DD$.

Our variant of the Kahn-Kalai conjecture on isoperimetric stability
is only effective in the $p$-biased setting for small $p$,
whereas the corresponding known results
\cite{keller2018dnf, keevash2018edgeiso}
for the uniform measure are substantial weaker.
This leaves our current state of knowledge
in a rather peculiar state, as in many related problems
the small $p$ case seems harder than the uniform case!
A natural open problem is give a unified approach
extending both results for all $p$.

Another compelling open problem is to generalise Hatami's Theorem
to the sparse regime, i.e.\ to obtain a density increase
from $\mu_{p}\left(f\right) = o\left(1\right)$
to $\mu_{q}\left(f\right)\ge 1-\eps$
under some pseudorandomness condition on $f$;
we expect that a such result would have
profound consequences in Extremal Combinatorics.

Lastly, in relation to our hypergraph Tur\'an results,
our notion of generalised criticality seems very restrictive
(indeed, we do not know of any examples where it is applicable
besides those mentioned above),
so it would be interesting to find more refined parameters
of expanded hypergraphs that determine the Tur\'an number
${\sf ex}(n, G^{+}(k))$ for more general classes of graphs.

\subsection*{Acknowledgment} We would like to thank Yuval Filmus, Ehud Friedgut, Gil Kalai, Nathan
Keller, Guy Kindler, and Muli Safra for various helpful comments and suggestions.


\bibliographystyle{plain}
\bibliography{refs}

\end{document}